\documentclass[a4paper,12pt]{scrartcl}

\usepackage[backref=true,style=alphabetic,maxnames=5]{biblatex}
\bibliography{kan}

\usepackage[left=2.5cm,right=2.5cm,top=2.7cm,bottom=2.7cm]{geometry} 
\usepackage[utf8]{inputenc}
\usepackage[english]{babel}
\usepackage{amssymb}
\usepackage{amsmath}
\usepackage{latexsym}
\usepackage[bookmarks]{hyperref}
\usepackage{amsthm}
\usepackage{graphicx}
\usepackage{bbm}
\usepackage{verbatim}
\usepackage{epigraph}
\usepackage{mathtools}
\usepackage{authblk}
\usepackage[pdflatex]{crop}
\usepackage{tikz}
\usepackage{tikz-cd}
\usepackage{wasysym}
\usepackage{adjustbox}
\usepackage{microtype}
\usepackage{chngcntr}
\usepackage{csquotes}
\usepackage[capitalize]{cleveref}
\usepackage{stmaryrd}	
\allowdisplaybreaks
\usetikzlibrary{shapes.geometric,fit}
\usetikzlibrary{decorations.pathmorphing}

\setkomafont{disposition}{\normalfont\bfseries}
\newcommand{\auth}[0]{Paolo Perrone}
\newcommand{\tit}[0]{Kan extensions are partial colimits}
\newcommand{\kw}[0]{Categorical probability, Giry monad, Kantorovich monad.}

\hypersetup{
pdfauthor={\auth},%
pdftitle={\tit},%
colorlinks, linktocpage=true, pdfstartpage=1, pdfstartview=FitV,%
breaklinks=true, pdfpagemode=UseNone, pageanchor=true, pdfpagemode=UseOutlines,%
plainpages=false, bookmarksnumbered, bookmarksopen=true, bookmarksopenlevel=1,%
hypertexnames=true, pdfhighlight=/O,%
urlcolor=black, linkcolor=black, citecolor=black, 
}
\numberwithin{equation}{section}

\theoremstyle{plain}
\newtheorem{thm}{Theorem}[section]
\newtheorem{lemma}[thm]{Lemma}
\newtheorem{prop}[thm]{Proposition}
\newtheorem{cor}[thm]{Corollary}

\newtheorem{deph}[thm]{Definition}

\theoremstyle{definition}
\newtheorem{remark}[thm]{Remark}
\newtheorem{eg}[thm]{Example}

\DeclareMathOperator{\Hom}{Hom}

\DeclareMathOperator{\im}{Im}

\newcommand{\Z}{\mathbb{Z}}
\newcommand{\N}{\mathbb{N}}

\newcommand{\Q}{\mathbb{Q}}

\newcommand{\C}{\mathbb{C}}
\newcommand{\R}{\mathbb{R}}

\newcommand{\cat}[1]{{\mathsf{#1}}} 
\newcommand{\ar}[2][]{\arrow{#2}{#1}}
\newcommand{\nat}[2][]{\arrow[Rightarrow]{#2}{#1}} 
\newcommand{\idar}[2][]{\arrow[equal]{#2}{#1}} 

\newcommand{\op}{\mathrm{op}}

\DeclareMathOperator*{\colim}{colim}
\newcommand{\id}{\mathrm{id}} 
\newcommand{\Lan}[2]{\mathrm{Lan}_{#2}{#1}} 

\DeclareMathOperator*{\bigunion}{\bigcup}

\newcommand{\U}{\mathcal{U}}
\newcommand{\E}{\mathcal{E}}
\newcommand{\e}{\mathrm{E}}

\def\groth{\begingroup\textstyle \int \! \endgroup}

\let\originalleft\left
\let\originalright\right
\renewcommand{\left}{\mathopen{}\mathclose\bgroup\originalleft}
\renewcommand{\right}{\aftergroup\egroup\originalright}

\tikzset{shorten <>/.style={shorten >=#1,shorten <=#1}}
\tikzset{Rightarrow/.style={double equal sign distance,>={Implies},->},
triple/.style={-,preaction={draw,Rightarrow}}}

\title{\tit}

\author[1]{Paolo Perrone\footnote{Correspondence: paolo.perrone [at] cs.ox.ac.uk}}
\affil[1]{\small University of Oxford, England, U.K.}

\author[2]{Walter Tholen}
\affil[2]{\small York University, Toronto ON, Canada}

\date{}

\begin{document}

\maketitle

\begin{abstract}
 One way of interpreting a left Kan extension is as taking a kind of ``partial colimit'', whereby one replaces parts of a diagram by their colimits. We make this intuition precise by means of the ``partial evaluations'' sitting in the so-called bar construction of monads. The (pseudo)monads of interest for forming colimits are the monad of diagrams and the monad of small presheaves, both on the (huge) category CAT of locally small categories. Throughout, particular care is taken to handle size issues, which are notoriously delicate in the context of free cocompletion.  
 
 We spell out, with all 2-dimensional details, the structure maps of these pseudomonads. Then, based on a detailed general proof of how the ``restriction-of-scalars'' construction of monads extends to the case of pseudoalgebras over pseudomonads, we define a morphism of monads between them, which we call ``image''. This morphism allows us in particular to generalize the idea of ``confinal functors'', i.e.~of functors which leave colimits invariant in an absolute way. This generalization includes the concept of absolute colimit as a special case.
  
 The main result of this paper spells out how a pointwise left Kan extension of a diagram corresponds precisely to a partial evaluation of its colimit. This categorical result is analogous to what happens in the case of probability monads, where a conditional expectation of a random variable corresponds to a partial evaluation of its center of mass. 
\end{abstract}

\newpage
\tableofcontents

\section{Introduction}

Kan extensions are a prominent tool of category theory, to the extent that, already in the preface to the first edition of \cite{maclane}, Mac Lane declared that \emph{``all concepts of category theory are Kan extensions''}, a claim reinforced more recently in \cite[Chapter~1]{riehl}. 
However, they are also considered to be a notoriously slippery concept, especially by newcomers to the subject. One of the most powerful pictures that help understanding how they work may be the idea that Kan extensions, especially in their pointwise form, ``replace parts of a diagram with their best approximations, either from the right or from the left''. 
In other words, Kan extensions can be seen as taking limits or colimits of ``parts'' of a diagram. The scope of this paper is making this intuition mathematically precise. 

We make use of the concept of \emph{partial evaluation}, which was introduced in \cite{pev}, and which is a way to formalize ``partially computed operations'' in terms of monads. The standard example is that ``$1+2+3+4$'' may be evaluated to ``$10$'', but also \emph{partially evaluated} to ``$3+7$'', whereby \emph{parts} of the given sum have been replaced by their sums. 

Just as monads on sets may be seen as encoding different algebraic structures and operations, here we consider pseudomonads on categories which encode the operation of \emph{taking colimits}. 
We are in particular interested in two pseudomonads: the \emph{monad of diagrams} and the \emph{monad of small presheaves} (also known as the \emph{free (small) cocompletion monad}). Both monads are known in the literature, but certainly not presented in sufficient detail as needed for our purposes. To make the paper more accessible, we therefore decided to spell out their definition in full detail, in \Cref{monadofdiagrams} and \Cref{smallpsh}. 

The definitions of pseudomonads, pseudoalgebras, and their morphisms are also hard to find in the literature in sufficient detail. For this reason, to avoid any ambiguity, we have given a detailed account of them in \Cref{pseudomonads}.
Readers who are familiar with these pseudomonads, and with the concepts of pseudomonads in general, may skip these sections, with the exception of \Cref{seccocompalg}, \Cref{immm} and \Cref{restrictionofscalars}, which contain new results.

Here is what the novel content of this work consists of. First of all, we introduce the concept of ``image presheaf'', which takes a diagram and forms a presheaf that can be considered the ``free colimit'' of the diagram. This induces a morphism of monads from the monad of diagrams to the monad of small presheaves, which in turn gives a ``pullback'' functor between the categories of algebras (we prove the 2-dimensional version of this statement in \Cref{restrictionofscalars}). 
This morphism of monads is not injective in any sense. Indeed, it turns out that diagrams with isomorphic image presheaves have the same colimit, in a very strong sense, analogous to ``differing by a confinal functor''. We indeed generalize the theory of confinal functors, and connect it to the theory of absolute colimits -- both because we need that in order to prove the subsequent statements, and because it should be interesting for its own sake.

We then turn to the central topic of this paper and study partial evaluations for both monads. We prove that partial evaluations for the monad of diagrams correspond to pointwise left Kan extensions along split opfibrations, by invoking the Grothendieck correspondence between split opfibrations and functors into $\cat{Cat}$. 
For the monad of small presheaves, we show that partial evaluations correspond to pointwise left Kan extensions along arbitrary functors. This result may be summarized in the following way: given small presheaves $P$ and $Q$ on a locally small, small-cocomplete category, $Q$ is a partial colimit of $P$ if and only if they can be written as image presheaves of small diagrams $D$ and $D'$, in such a way that $D'$ is the left Kan extension of $D$ along some functor. More concisely, \emph{Kan extensions are partial colimits}, as claimed by the paper's title. 

This result is analogous to, and was motivated by, an analogous result in measure theory involving probability monads, where partial evaluations (or ``partial expectations'') correspond exactly to conditional expectations (\Cref{condexpthm}).
Indeed, one could say that ``if coends are like integrals, then Kan extensions are like conditional expectations''.
(See \Cref{integrals} for more on this.)

As usual, when one talks about free cocompletion, one has to be very careful with size issues. This is why some parts of this work, such as the proof of \Cref{imimeso}, appear to be rather technical. The payoff is that the main theorems of this work will hold for arbitrary (small) colimits in arbitrary (locally small) categories, beyond the trivial case of preorders.

\paragraph{Outline.}
\addcontentsline{toc}{subsection}{Outline}
In \Cref{monadofdiagrams} we study the category of diagrams in a given category, and show that the construction gives a pseudomonad on the 2-category of locally small categories. While this construction seems to be known, its details don't seem to have been spelled out previously.
The content of \Cref{seccocompalg}, however, seems to be entirely new. We show that cocomplete categories, equipped with a choice of colimit for each diagrams, are pseudoalgebras over this pseudomonad, and that not all pseudoalgebras are of that form.

In \Cref{imagepsh} we define the concept of ``image presheaf'', which can be seen as a ``free colimit of a diagram'', or as a ``colimit blueprint''. We show that ``having the same image presheaf'' is a strong and consistent generalization both of the theory of confinal functors (\Cref{confinallemma}), and of the concept of absolute colimit (\Cref{absolutelemma}). 
As far as we know, this generalization is new.

In \Cref{smallpsh} we study small presheaves and show that they form a pseudomonad. Again, this is known, but here we spell out the construction in much greater detail than previous accounts have done. This enables us to establish the new result presented in \Cref{immm}: the image presheaf construction forms a morphism of pseudomonads from diagrams to presheaves.

The principal new results of this paper appear in  \Cref{partialcolimits}. \Cref{weakform} and \Cref{strongform} state that partial colimits for the monad of diagrams correspond to pointwise left Kan extensions of diagrams along split opfibrations. In \Cref{kanpevP} we prove that partial colimits for the free cocompletion monad correspond to pointwise left Kan extensions of diagrams along arbitrary functors.
Then, in \Cref{integrals}, we compare this categorical result to the analogous measure-theoretic fact that partial expectations for probability monads correspond (in some cases) to conditional expectations of random variables. This is in line with the famous analogy between coends and integrals. 

Finally, in \Cref{pseudomonads} we recall the (known) definition of pseudomonads and pseudoalgebras, and of the categories they form, which we use in the rest of the paper.
We also provide a 2-dimensional version of the ``restriction of scalars'' construction (\Cref{rescalthm}), where a morphism of monads induces a functor between the categories of their algebras in opposite direction. As far as we know, this 2-dimensional version has not appeared in the literature previously..

\paragraph{Acknowledgements.}
\addcontentsline{toc}{subsection}{Acknowledgements}
The first author would like to thank Bartosz Milewski and David Jaz Myers for the insight on coends and weighted limits, Joachim Kock and Emily Riehl for enlightenment on some of the higher-dimensional aspects, and Tobias Fritz for further helpful insight.

The first author would also like to thank Sean McKenna, as well as David Spivak and MIT as a whole, for all the support during the 2020 pandemic, and Nathanael Arkor for the development of the app Quiver, which proved to be very helpful in writing some of the diagrams in this document.

The first author was affiliated to the Massachusetts Institute of Technology (MIT) for most of the time of writing. This research was partially funded by the Fields Institute (Canada) and by AFOSR grants FA9550-19-1-0113 and FA9550-17-1-0058 (U.S.A.).
The second author acknowledges partial financial support by the Natural Science and Engineering Council of Canada under the Discovery Grants Program (grant no.~501260).

\paragraph{Categorical setting, notation, and conventions.}
\addcontentsline{toc}{subsection}{Categorical setting, notation, and conventions}
As it is to be expected when one talks about generic colimits, size issues are relevant. Here are our conventions.

All the categories in this work (except $\cat{CAT}$) are assumed locally small. We denote by $\cat{Cat}$ the 2-category of small categories, and by $\cat{CAT}$ the 2-category of possibly large, locally small categories.
Note that $\cat{CAT}$ is itself larger than a large category (some authors call it a ``huge'' category).

When we say ``category'', without specifying the size, we will always implicitly refer to a possibly large, locally small category.

Similarly, by ``cocomplete category'' we always mean a possibly large, locally small category which admits all \emph{small} colimits.

\section{The monad of diagrams}\label{monadofdiagrams}

In this section we define the monad of diagrams. The first source for it that we are aware of is Guitart’s article \cite{guitart}, but without an explicit construction.
We give in detail all the structure maps, and in \Cref{seccocompalg} we prove that cocomplete categories with a choice of colimits are pseudoalgebras (but not all pseudoalgebras are in this form). 
The notions of pseudomonad and pseudoalgebra that we use are given in detail in \Cref{pseudomonads}.

Note that, differently from some of the literature, we use the following slightly generalized notion of \emph{morphism of diagrams} (also used, for example, in Guitart's original work \cite{guitart}). Moreover, in order to avoid size issues, we require every diagram to be small.

\begin{deph}\label{defdiag}
Let $\cat{C}$ be a locally small category. 

\begin{itemize}
\item We call a \emph{diagram in $\cat{C}$} a small category $\cat{J}$ together with a functor $D:\cat{J}\to\cat{C}$. Throughout this work, all the diagrams will be implicitly assumed to be of this form (i.e.~be small).

\item Given diagrams $(\cat{J},D)$ and $(\cat{J'},D')$ in $\cat{C}$, we call a \emph{morphism} of diagrams a functor $R:\cat{J}\to\cat{J'}$ together with a natural transformation $\rho:D\to D'\circ R$, i.e.~a diagram in $\cat{CAT}$ as the following.
$$
\begin{tikzcd}[row sep=small]
 \cat{J} 
  \ar[""{name=D,below}]{dr}{D} 
  \ar{dd}[swap]{R} \\
 & \cat{C} \\
 \cat{J'} 
  \ar{ur}[swap]{D'} 
  \ar[Rightarrow,from=D, "\rho", shorten <= 0.5em, shorten >= 1em, near start]
\end{tikzcd}
$$

\item Given diagrams $(\cat{J},D)$ and $(\cat{J'},D')$ in $D$ and morphisms of diagrams $(R,\rho), (R',\rho'): (\cat{J},D) \to (\cat{J'},D')$, we call a \emph{2-cell} of diagrams a natural transformation $\alpha:R\Rightarrow R'$ such that the following 2-cells are equal.
$$
\begin{tikzcd}[column sep=large]
 \cat{J} 
  \ar[""{name=D,below, pos=0.7}]{dr}{D} 
  \ar[bend right, ""{name=RP, right}]{dd}[swap]{R'} \ar[bend left, ""{name=R, left}]{dd}[pos=0.4]{R} \\
 & \cat{C} \\
 \cat{J'} 
  \ar{ur}[swap]{D'} 
  \ar[Rightarrow,from=D, "\rho", shorten <= 0.5em, shorten >= 1.5em, near start]
  \ar[Rightarrow,from=R,to=RP,"\alpha"]
\end{tikzcd}
\qquad=\qquad\begin{tikzcd}[column sep=large]
 \cat{J} 
  \ar[""{name=D,below}]{dr}{D} 
  \ar[bend right]{dd}[swap]{R'} \\
 & \cat{C} \\
 \cat{J'} 
  \ar{ur}[swap]{D'} 
  \ar[Rightarrow,from=D, "\rho'", shorten <= 0.5em, shorten >= 1em, swap]
\end{tikzcd}
$$
\end{itemize}

We denote by $\cat{Diag(C)}$ the 2-category of diagrams in $\cat{C}$, their morphisms, and their 2-cells. (Sometimes we will still denote by $\cat{Diag(C)}$ the underlying 1-category.)
\end{deph}

Note that the definition of morphism of diagrams is slightly more general than just a natural transformation between parallel functors. 
This is still compatible with the traditional intuitive picture of ``deforming a diagram into another one'', provided that one notices the following. In principle $R$ may be not essentially surjective, so one should visualize the natural transformation $\rho$ as ``deforming'' the figure drawn by $D$ into a \emph{subfigure} of the one drawn by $D'$. 

Note moreover that:
\begin{itemize}
 \item For $\cat{C}$ locally small, $\cat{Diag(C)}$ is locally small too;
 \item The forgetful functor $\cat{Diag(C)}\to \cat{Cat}$ given by the domain is a fibration (via precomposition), and it is an opfibration (via left Kan extensions) if and only if $\cat{C}$ is cocomplete (see for example \cite[Proposition~2.8]{decomposition}).
\end{itemize}

In the rest of this section we show that $\cat{C}\rightsquigarrow\cat{Diag(C)}$ is part of a pseudomonad on $\cat{CAT}$, and that cocomplete categories with a choice of colimit for each diagram are pseudoalgebras, with the structure map given by such chosen colimits. For the precise definitions of pseudomonads and pseudoalgebras, see \Cref{pseudomonads}.

\subsection{Functoriality} 

We show that the assignment $\cat{C}\rightsquigarrow\cat{Diag(C)}$ is part of a 2-functor on $\cat{CAT}$. First of all, let $\cat{C}$ and $\cat{D}$ be locally small categories, and let $F:\cat{C}\to\cat{D}$ be a functor. 

Consider now $\cat{Diag(C)}$ and $\cat{Diag(D)}$ as 1-categories. There is a (1-)functor $F_*:\cat{Diag(C)}\to\cat{Diag(D)}$ induced by $F$ via postcomposition and whiskering, as follows. 
$$
\begin{tikzcd}
 \cat{J} \ar{r}{D} & \cat{C} 
\end{tikzcd}
\qquad\rightsquigarrow\qquad
\begin{tikzcd}
 \cat{J} \ar{r}{D} & \cat{C} \ar{r}{F} & \cat{D}
\end{tikzcd}
$$

$$
\begin{tikzcd}[row sep=small]
 \cat{J} 
  \ar[""{name=D,below}]{dr}{D} 
  \ar{dd}[swap]{R} \\
 & \cat{C} \\
 \cat{J'} 
  \ar{ur}[swap]{D'} 
  \ar[Rightarrow,from=D, "\rho", shorten <= 0.5em, shorten >= 1em, near start]
\end{tikzcd}
\qquad\rightsquigarrow\qquad
\begin{tikzcd}[row sep=small]
 \cat{J} 
  \ar[""{name=D,below}]{dr}{D} 
  \ar{dd}[swap]{R} \\
 & \cat{C} \ar{r}{F} & \cat{D} \\
 \cat{J'} 
  \ar{ur}[swap]{D'} 
  \ar[Rightarrow,from=D, "\rho", shorten <= 0.5em, shorten >= 1em, near start]
\end{tikzcd}
$$
Therefore, $\cat{Diag}$ is an endofunctor of $\cat{CAT}$.

The functor $F_*:\cat{Diag(C)}\to\cat{Diag(D)}$ extends to $2$-cells giving a 2-functor, but we will not need this in order for $\cat{Diag}$ to be a pseudomonad $\cat{CAT}$.

On the other hand, we need to extend $\cat{Diag}$ to the 2-cells of $\cat{CAT}$. So let $\cat{C}$ and $\cat{D}$ be locally small categories, let $F,G:\cat{C}\to\cat{D}$ be functors, and let $\alpha:F\Rightarrow G$ be a natural transformation.
We have an induced natural transformation $\alpha_*:F_*\Rightarrow G_*$ induced via whiskering, as follows. 
To each diagram $(\cat{J},D)$ in $\cat{C}$, we assign the morphism of diagrams $(\id_\cat{J}, \alpha D)$ of $\cat{D}$, i.e.
$$
\begin{tikzcd}
 \cat{J} \ar{r}{D} & \cat{C} 
\end{tikzcd}
\qquad\rightsquigarrow\qquad
\begin{tikzcd}
 \cat{J} \ar{r}{D} & \cat{C} \ar[bend left=35, ""{name=F,below}]{r}{F} \ar[bend right=35, ""{name=G,above}]{r}[swap]{G} & \cat{D}
 \ar[Rightarrow, from=F, to=G, "\alpha"]
\end{tikzcd}
$$
Naturality follows from naturality of $\alpha$.
This makes $\cat{Diag}$ a strict 2-functor $\cat{CAT}\to\cat{CAT}$.

\subsection{Unit and multiplication}

\subsubsection{The unit: one-object diagrams}

The unit of the monad is a map constructing ``one-object diagrams''.
In detail, let $\cat{C}$ be a locally small category. We construct the functor $\eta_\cat{C}:\cat{C}\to \cat{Diag(C)}$ as follows.
$$
\begin{tikzcd}
 C
\end{tikzcd}
\qquad\rightsquigarrow\qquad
\begin{tikzcd}
 \cat{1} \ar{r}{C} & \cat{C}
\end{tikzcd}
$$
$$
\begin{tikzcd}
 C \ar{r}{f} & C'
\end{tikzcd}
\qquad\rightsquigarrow\qquad
\begin{tikzcd}[row sep=small]
 \cat{1} 
  \ar[""{name=D,below}]{dr}{C} 
  \ar{dd}[swap]{\id} \\
 & \cat{C} \\
 \cat{1} 
  \ar{ur}[swap]{C'} 
  \ar[Rightarrow,from=D, "f", shorten <= 0.5em, shorten >= 1em, near start]
\end{tikzcd}
$$
For brevity, we will denote $\eta_\cat{C}$ simply by $\eta$. This is (strictly) natural in the category $\cat{C}$: given a functor $F:\cat{C}\to\cat{D}$, the following diagram commutes strictly.
$$
\begin{tikzcd}
 \cat{C} \ar{d}{\eta} \ar{r}{F} & \cat{D} \ar{d}{\eta} \\
 \cat{Diag(C)} \ar{r}{F_*} & \cat{Diag(D)} 
\end{tikzcd}
$$
Indeed, both paths in the diagram give the following assignment,
$$
\begin{tikzcd}
 C
\end{tikzcd}
\qquad\rightsquigarrow\qquad
\begin{tikzcd}
 \cat{1} \ar{r}{C} & \cat{C} \ar{r}{F} & \cat{D}
\end{tikzcd}
$$
$$
\begin{tikzcd}
 C \ar{r}{f} & C'
\end{tikzcd}
\qquad\rightsquigarrow\qquad
\begin{tikzcd}[row sep=small]
 \cat{1} 
  \ar[""{name=D,below}]{dr}{C} 
  \ar{dd}[swap]{\id} \\
 & \cat{C} \ar{r}{F} & \cat{D} \\
 \cat{1} 
  \ar{ur}[swap]{C'} 
  \ar[Rightarrow,from=D, "f", shorten <= 0.5em, shorten >= 1em, near start]
\end{tikzcd}
$$
using the fact that $F(C)=F\circ C$ if we view $C$ as a functor $1\to\cat{C}$, and that analogously $Ff$ is given by whiskering $f$ (seen as a natural transformation) with $F$.

\subsubsection{Diagrams of diagrams are lax cocones}

Let's now turn to the multiplication. We first notice that an object of $\cat{Diag(Diag(C))}$ is the same as a lax cocone in $\cat{CAT}$ with tip $\cat{C}$, where the indexing category and all the categories appearing in the cone except $\cat{C}$ are required to be small. Let's see how. 
Let $\cat{J}$ be a small category. A functor $D:\cat{J}\to\cat{Diag(C)}$ assigns to each object $J$ of $\cat{J}$ a diagram in $\cat{C}$, i.e.~a small category $D_0J$ together with a functor $D_1J:D_0J\to\cat{C}$:
$$
\begin{tikzcd}
 J
\end{tikzcd}
\qquad\rightsquigarrow\qquad
\begin{tikzcd}
 D_0J \ar{r}{D_1J} & \cat{C}
\end{tikzcd}
$$
and to each morphism $j:J\to J'$ of $\cat{J}$ a morphism of diagrams, which amounts to a functor $D_0j:D_0J\to D_0J'$ together with a natural transformation $D_1j$ as below:
$$
\begin{tikzcd}
 J \ar{d}[swap]{j} \\
 J'
\end{tikzcd}
\qquad\rightsquigarrow\qquad
\begin{tikzcd}[row sep=small]
 D_0J
  \ar[""{name=D,below}]{dr}{D_1J} 
  \ar{dd}[swap]{D_0j} \\
 & \cat{C} \\
 D_0J'
  \ar{ur}[swap]{D_1J'} 
  \ar[Rightarrow,from=D, "D_1j", shorten <= 0.5em, shorten >= 1em, swap, near start]
\end{tikzcd}
$$
Moreover, since we want $D$ to be a functor, we need it to preserve identities and composition, i.e.~$D_0$ needs to be a functor, and $D_1$ needs to satisfy the conditions $D_1(\id_J)=\id_{D_1J}$ and $D_1(j'\circ j)=D_1j'D_0j\circ D_1J$, which are exactly the conditions of lax naturality. In pictures,
$$
\begin{tikzcd}
 J \ar{d}[swap]{\id} \\
 J
\end{tikzcd}
\qquad\rightsquigarrow\qquad
\begin{tikzcd}[row sep=small]
 D_0J
  \ar[""{name=D,below}]{dr}{D_1J} 
  \ar{dd}[swap]{\id} \\
 & \cat{C} \\
 D_0J
  \ar{ur}[swap]{D_1J} 
  \ar[Rightarrow,from=D, "\id", shorten <= 0.5em, shorten >= 1em, swap, near start]
\end{tikzcd}
$$
$$
\begin{tikzcd}
 J \ar{dd}[swap]{j} \\ \\
 J' \ar{dd}[swap]{j'} \\ \\
 J''
\end{tikzcd}
\qquad\rightsquigarrow\qquad
\begin{tikzcd}[column sep=huge]
 D_0J
  \ar[""{name=D,below}]{ddr}{D_1J} 
  \ar{dd}[swap]{D_0j} \\
 \\
 D_0J
  \ar[""{name=MID}]{r}{D_1J'} \ar{dd}[swap]{D_0j'} \ar[Rightarrow,from=D, "D_1j", shorten <= 0.2em, shorten >= 0.2em, swap, near start]  & \cat{C} \\
  \\
 D_0J'' \ar{uur}[swap]{D_1J''}
 \ar[Rightarrow,from=MID, "D_1j'", shorten <= 1.5em, shorten >= 1.5em, swap, pos=0.38]
\end{tikzcd}
$$
In other words, a functor $D:\cat{J}\to\cat{Diag(C)}$ consists of a functor $D_0:\cat{J}\to\cat{Cat}\subseteq\cat{CAT}$, together with a lax cocone $D_1$ in $\cat{CAT}$ under $D_0$ with tip $\cat{C}$. A lax cocone is a lax natural transformation $D_1:D_0\Rightarrow \Delta \cat{C}$, where $\Delta\cat{C}$ is the constant functor at $\cat{C}$.

\subsubsection{The multiplication: the Grothendieck construction}

Given now $D=(D_0,D_1)$ as above, take the Grothendieck construction $\groth D_0$ of $D_0:\cat{J}\to\cat{Cat}\subseteq \cat{CAT}$, which we recall. 
\begin{itemize}
 \item An object of $\groth D_0$ consists of a pair $(J,X)$ where $J$ is an object of $\cat{J}$ and $X$ is an object of the category $D_0J$;
 \item A morphism $(J,X)\to (J',X')$ of $\groth D_0$ consists of a pair $(j,f)$ where $j:J\to J'$ is a morphism of $\cat{J}$, and $f:D_0j(X)\to X'$ is a morphism of the category $D_0J'$. 
\end{itemize}
The short integral sign does not denote a coend here, it is standard for the Grothendieck construction (we use different sizes to avoid confusion, since both symbols are standard notation).
Note that since $\cat{J}$ and all the $D_0J$ are small, $\groth D_0$ is small too. Its set of objects is given by
$$
\coprod_{J\in\cat{J}} D_0J .
$$
Moreover: 
\begin{itemize}
 \item For each object $J$ of $\cat{J}$, the inclusion maps $i_J:D_0J\to \groth D_0$ defined by the coproduct above can be canonically made into functors via 
 $$
 \begin{tikzcd}
  X \ar{d}[swap]{f} \\
  X'
 \end{tikzcd}
 \qquad\rightsquigarrow\qquad
 \begin{tikzcd}
  (J,X) \ar{d}{(\id_J,f)} \\
  (J,X')
 \end{tikzcd}
 $$
 We will call the images of the $i_J$ the \emph{fibers} of $\groth D_0$.
 \item For each morphism $j:J\to J'$ of $\cat{J}$, there is a natural transformation 
 $$
 \begin{tikzcd}[row sep=small]
 D_0J
  \ar[""{name=D,below}]{dr}[near end]{i_J} 
  \ar{dd}[swap]{D_0j} \\
 & \groth D_0  & &  \\
 D_0J' \ar{ur}[near end,swap]{i_{J'}}
  \ar[Rightarrow,from=D, "i_j", shorten <= 0.5em, shorten >= 1em, swap, near start]
\end{tikzcd}
  $$
  whose component at each object $X$ of $D_0J$ is given by 
  $$
  \begin{tikzcd}[column sep=huge]
   (J,X) \ar{r}{(j,\id_{D_0j(X)})} & (J', D_0j(X))
  \end{tikzcd}
  $$
 \item The $i_J$ and $i_j$ assemble into a lax cocone $D_0 \Rightarrow \Delta \groth D_0$, i.e.~the identity and composition conditions are satisfied.
\end{itemize}

It is well known that $\groth D_0$ is the oplax colimit of $D_0$ in $\cat{Cat}$, with the universal lax cocone given by the $i_J$. We now show that it is so also in $\cat{CAT}$, and we also give a \emph{strict} version of the universal property.\footnote{By ``strict'' here we mean ``we give an \emph{isomorphism} of hom-categories, not just an equivalence''. The colimit is still oplax, not strict.}

\begin{prop}\label{oplaxcolimit}
 Let $D_0:\cat{J}\to\cat{Cat}\subseteq \cat{CAT}$ be a small diagram of small categories. Let $D_1:D_0\Rightarrow \cat{C}$ be a lax cocone over $D_0$ in $\cat{CAT}$, with tip $\cat{C}$ locally small (but not necessarily small). There is a unique functor $\groth D_0\to \cat{C}$ such that 
 \begin{itemize}
  \item for all objects $J$ of $\cat{J}$, the following triangle commutes (strictly);
  \begin{equation}\label{ijstrict}
  \begin{tikzcd}
   D_0J \ar{d}[swap]{i_J} \ar{dr}{D_1J} \\
   \groth D_0 \ar{r} & \cat{C}
  \end{tikzcd}
  \end{equation}
  \item for all morphisms $j:J\to J'$ of $\cat{J}$, the following 2-cells coincide.
  \begin{equation}\label{ijlax}
  \begin{tikzcd}[column sep=large]
 D_0J
  \ar[""{name=D,below}]{dr}{D_1J} 
  \ar{dd}[swap]{D_0j} \\
 & \cat{C} \\
 D_0J'
  \ar{ur}[swap]{D_1J'} 
  \ar[Rightarrow,from=D, "D_1j", shorten <= 0.5em, shorten >= 1em, swap, near start]
\end{tikzcd}
\qquad=\qquad
\begin{tikzcd}
 D_0J
  \ar{drrr}{D_1J} 
  \ar[""{name=D,below}]{dr}[near end]{i_J} 
  \ar{dd}[swap]{D_0j} \\
 & \groth D_0 \ar{rr} & & \cat{C} \\
 D_0J' \ar{ur}[near end,swap]{i_{J'}}
  \ar{urrr}[swap]{D_1J'} 
  \ar[Rightarrow,from=D, "i_j", shorten <= 0.5em, shorten >= 1em, swap, near start]
\end{tikzcd}
  \end{equation}
 \end{itemize}
\end{prop}

Denote this functor by $\mu(D_0,D_1)$, or more briefly by $\mu(D)$. It is a diagram in $\cat{C}$. This will give the multiplication of the monad $\cat{Diag}$.

\begin{proof}[Proof of \Cref{oplaxcolimit}]
 Since we want the diagram \ref{ijstrict} to commute strictly, the only possibility to define $\mu(D)$ on objects is as follows. For every object $J$ of $\cat{J}$, and for every object $X$ of $D_0J$,
 $$
 \mu(D) (J,X) \coloneqq D_1J(X) .
 $$
 Just as well, for all the morphisms of $\groth D_0$ in the fiber, i.e.~in the form $(\id_J,f)$ for a morphism $f:X\to X'$ of $D_0J$, we are forced to define
 $$
 \mu(D) (\id_J,f) \coloneqq D_1J(f) .
 $$
 
 Moreover, since we want the condition \eqref{ijlax} to hold, for all morphisms $j:J\to J'$ of $\cat{J}$ we have to require that $\mu(D)$ on the components of $i_j$ has to give the respective component of $D_1j$. Explicitly, for each object $X$ of $D_0J$, 
 $$
 \mu(D) (i_j)_X = \mu(D) (j, \id_{D_0j(X)}) \coloneqq (D_1j)_X .
 $$
 
 The generic morphism $(j,f):(J,X)\to(J',X')$, for $j:J\to J'$ and $f:D_0j(X)\to X'$ can be decomposed as 
 $$
 \begin{tikzcd}[column sep=huge]
  (J,X) \ar{r}{(j,\id_{D_0j(X)})} & (J',D_0j(X)) \ar{r}{(\id_J,f)} & (J',X')
 \end{tikzcd}
 $$
 and so we have determined the action of $\mu(D)$ for all morphisms of $\groth D_0$.
 
 Functoriality of this assignment is routine. 
\end{proof}

\subsubsection{Functoriality and naturality of the multiplication}

We now have to show that the assignment $(D_0,D_1)\mapsto \mu(D_0,D_1)$ is functorial, and that it is natural in the category $\cat{C}$.

To address functoriality, we need to look at morphisms in $\cat{Diag(Diag(C))}$.

Given diagrams $D:\cat{J}\to\cat{Diag(C)}$ and $E:\cat{K}\to\cat{Diag(C)}$, where $\cat{J}$ and $\cat{K}$ are small categories, a morphism of diagrams from $D$ to $E$ amounts to a functor $F:\cat{J}\to \cat{K}$ together with a natural transformation
$$
\begin{tikzcd}[row sep=small]
 \cat{J}
  \ar[""{name=D,below}]{dr}{D} 
  \ar{dd}[swap]{F} \\
 & \cat{Diag(C)} \\
 \cat{K}
  \ar{ur}[swap]{E} 
  \ar[Rightarrow,from=D, "\phi", shorten <= 0.5em, shorten >= 1em, near start]
\end{tikzcd}
$$
Explicitly, $D$ consists of a functor $D_0:\cat{J}\to\cat{Cat}\subseteq\cat{CAT}$ and a lax cocone $D_1:D_0\Rightarrow \cat{C}$, and $E$ has an analogous form. The natural transformation $\phi$ amounts to the following. For each object $J$ of $\cat{J}$, we have a morphism of diagrams 
$$
\begin{tikzcd}[row sep=small]
 D_0J
  \ar[""{name=D,below}]{dr}{\phi_0J} 
  \ar{dd}[swap]{D_0J} \\
 & \cat{C} \\
 E_0FJ
  \ar{ur}[swap]{E_1J} 
  \ar[Rightarrow,from=D, "\phi_1J", shorten <= 0.5em, shorten >= 1em, swap, near start]
\end{tikzcd}
$$
and for each morphism $j:J\to J'$ of $\cat{J}$, the following diagrams have to commute. 
First of all, this diagram of functors has to commute strictly. 
$$
\begin{tikzcd}[sep=large]
 D_0J \ar{d}{D_0j} \ar{r}{\phi_0J} & E_0FJ \ar{d}{E_0Fj} \\
 D_0J' \ar{r}{\phi_0J'} & E_0FJ'
\end{tikzcd}
$$
Moreover, the following composite 2-cells have to coincide, forming a commutative pyramid with 2-cells as lateral faces, whose square base is the commutative square just described above.
\begin{equation}\label{pyramid}
\begin{tikzcd}[column sep=large]
 & \cat{C} \\
 \\ \\ \phantom{E}\\
 D_0J \ar{dr}[swap]{D_0j} \ar[""{name=D1J}]{uuuur}[near start]{D_1J} & & |[alias=E0FJP]| E_0FJ' \ar{uuuul}[swap, near start]{E_1FJ'} \\
 & D_0J 
 \ar[Rightarrow, from=D1J, "D_1j", shorten <= 2em, shorten >= 2em, swap]
 \ar[""{name=D1JP}]{uuuuu}[swap]{D_1J'} \ar{ur}[swap]{\phi_0J'}
 \ar[Rightarrow, from=D1JP, to=E0FJP, "\phi_1J'", shorten <= 2em, shorten >= 1.5em, swap, pos=0.6]
\end{tikzcd}
\qquad=\qquad
\begin{tikzcd}[column sep=large]
 & \cat{C} \\
 \\ \\ 
 & |[alias=E0FJ]| E_0FJ \ar[""{name=E1FJ}]{uuu}[swap, pos=0.4]{E_1FJ} \ar{dr}[swap]{E_0Fj}\\
 D_0J \ar{ur}[swap]{\phi_0J} \ar{dr}[swap]{D_0j} \ar[""{name=D1J}]{uuuur}[near start]{D_1J} & & E_0FJ' \ar{uuuul}[swap, near start]{E_1FJ'}
 \ar[Rightarrow, from=E1FJ, "E_1Fj", shorten <= 2em, shorten >= 3em, pos=0.4] \\
 & D_0J \ar{ur}[swap]{\phi_0J'}
 \ar[Rightarrow, from=D1J, to=E0FJ, "\phi_1J", shorten <= 0.5em, shorten >= 0em, swap, pos=0.7]
\end{tikzcd}
\end{equation}

Form now the Grothendieck construction of $E_0$. We can form a lax cocone over $D_0$ with tip $\groth E_0$ as follows.
$$
\begin{tikzcd}
 J
\end{tikzcd}
\qquad\rightsquigarrow\qquad
\begin{tikzcd}
 D_0J \ar{r}{\phi_0J} & E_0FJ \ar{r}{i_{FJ}} & \groth E_0
\end{tikzcd}
$$

$$
\begin{tikzcd}
 J \ar{d}{j} \\
 J'
\end{tikzcd}
\qquad\rightsquigarrow\qquad
\begin{tikzcd}[row sep=small]
 D_0J \ar{r}{\phi_0J} \ar{dd}[swap]{D_0j} & E_0FJ
  \ar[""{name=D,below}]{dr}{i_{FJ}} 
  \ar{dd}[swap]{E_0Fj} \\
 & & \groth E_0 \\
 D_0J' \ar{r}{\phi_0J'} & E_0FJ' 
  \ar{ur}[swap]{i_{FJ'}} 
  \ar[Rightarrow,from=D, "i_{Fj}", shorten <= 0.5em, shorten >= 1em, swap, near start]
\end{tikzcd}
$$

Note that this is a lax cocone $D_0\Rightarrow\groth E_0$. By the universal property of the Grothendieck construction as an oplax colimit (\Cref{oplaxcolimit}), there is a unique functor $\groth D_0\to \groth E_0$ such that 
\begin{itemize}
  \item for all objects $J$ of $\cat{J}$, the following square commutes;
  \begin{equation}\label{squarestrict}
  \begin{tikzcd}
   D_0J \ar{d}[swap]{i_J} \ar{r}{\phi_0J} & E_0FJ \ar{d}{i_{FJ}} \\
   \groth D_0 \ar{r} & \groth E_0
  \end{tikzcd}
  \end{equation}
  \item for all morphisms $j:J\to J'$ of $\cat{J}$, the following composite 2-cells coincide.
$$
\begin{tikzcd}
 D_0J \ar{r}{\phi_0J} \ar{dd}[swap]{D_0j} & E_0FJ
  \ar[""{name=D,below}]{dr}{i_{FJ}} 
  \ar{dd}[swap]{E_0Fj} \\
 & & \groth E_0 \\
 D_0J' \ar{r}{\phi_0J'} & E_0FJ' 
  \ar{ur}[swap]{i_{FJ'}} 
  \ar[Rightarrow,from=D, "i_{Fj}", shorten <= 1em, shorten >= 2em, swap, near start]
\end{tikzcd}
\qquad=\qquad
\begin{tikzcd}
 D_0J 
  \ar{dd}[swap]{D_0j} 
   \ar[""{name=D,below}]{dr}[near end]{i_J}
  \ar{r}{\phi_0J} & E_0 FJ \ar{dr}{i_{FJ}}\\
 & \groth D_0 \ar{r} & \groth E_0\\
 D_0J' 
   \ar[Rightarrow,from=D, "i_j", shorten <= 1em, shorten >= 2em, swap, near start]
 \ar{ur}[near end,swap]{i_{J'}}
  \ar{r}[swap]{\phi_0J'} & E_0FJ' \ar{ur}[swap]{i_{FJ'}}
\end{tikzcd}
$$
\end{itemize}

Denote this functor by $\mu_0(F,\phi)$.
This gives a triangle 
$$
\begin{tikzcd}[row sep=small]
 \groth D_0
  \ar[""{name=D,below}]{dr}{\mu(D)} 
  \ar{dd}[swap]{\mu_0(F,\phi)} \\
 & \cat{C} \\
 \groth E_0
  \ar{ur}[swap]{\mu(E)} 
\end{tikzcd}
$$
which does not necessarily commute. In order to get a morphism of diagram $\mu(D)\to\mu(E)$ we need to fill the triangle above with a 2-cell which we form as follows. 
Consider the object $(J,X)$ of $\groth D_0$, where $J$ is an object of $\cat{J}$ and $X$ is an object of $D_0J$.
Note that, by the diagram~\ref{ijstrict}, $\mu(D)(J,X)=D_1J(X)$. Analogously, using the diagram~\ref{ijstrict} for $E$ together with the diagram~\ref{squarestrict}, 
$$
\mu(E)(\mu_0(J,X)) = \mu(E)(J,\phi_0J(X)) = E_1FJ(\phi_0J(X)).
$$
We now assign to the object $(J,X)$ the morphism of $\cat{C}$ given by the component of $\phi_1J$ at $X$,
$$
\begin{tikzcd}
 \mu(D)(J,X) = D_1J(X) \ar{r}{(\phi_1J)_X} & E_1FJ(\phi_0J(X)) = \mu(E)(\mu_0(J,X)) .
\end{tikzcd}
$$
Let's now show that this assignment is natural. We will again test this first along the fibers, and then on the opcartesian morphisms of $\groth D_0$. So let $f:X\to Y$ be a morphism of $D_0J$. The following diagram commutes simply by naturality of $\phi_1J$.
$$
\begin{tikzcd}
 \mu(D)(J,X) \ar{d}{\mu(D)(\id_J, f)} \idar{r} & D_1J(X) \ar{d}{D_1J(f)}  \ar{r}{(\phi_1J)_X} & E_1FJ(\phi_0J(X)) \ar{d}{E_1FJ(\phi_0J(f))}  \idar{r} & \mu(E)(\mu_0(J,X)) \ar{d}{\mu(E)(\mu_0(\id_J, f))} \\
 \mu(D)(J,D_0j(X)) \idar{r} & D_1J(X) \ar{r}{(\phi_1J)_X} & E_1FJ(\phi_0J(X)) \idar{r} & \mu(E)(\mu_0(J,X))
\end{tikzcd}
$$
Let now $j:J\to J'$ be a morphism of $\cat{J}$. We have to prove that the following diagram commutes.
$$
\begin{tikzcd}[column sep=huge]
 D_1J(X) \ar{d}{(D_1j)_X} \ar{r}{(\phi_1J)_X} & E_1FJ(\phi_0J(X)) \ar{d}{(E_1Fj)_{\phi_0J(X)}}  \\
 D_1J'(D_0j(X)) \ar{r}{(\phi_1J')_{D_0j(X)}} & E_1FJ'(\phi_0J'(D_0j(X))) 
\end{tikzcd}
$$
This is however exactly \Cref{pyramid}, written out in components.
Therefore we have a natural transformation, which we denote by $\mu_1(F,\phi)$, and we get a morphism of diagrams 
$$
\begin{tikzcd}[row sep=small]
 \groth D_0
  \ar[""{name=D,below}]{dr}{\mu(D)} 
  \ar{dd}[swap]{\mu_0(F,\phi)} \\
 & \cat{C} \\
 \groth E_0
  \ar{ur}[swap]{\mu(E)} 
  \ar[Rightarrow,from=D, "\mu_1", shorten <= 0.5em, shorten >= 1em, swap, near start]
\end{tikzcd}
$$
which makes $\mu$ functorial. (The identity and composition conditions follow by uniqueness.)

\begin{prop}
 The functor $\mu:\cat{Diag(Diag(C))}\to\cat{Diag(C)}$ is strictly natural in $\cat{C}$.
\end{prop}

\begin{proof}
Let $F:\cat{C}\to\cat{D}$ be a functor. The induced functor 
$$
\begin{tikzcd}
 \cat{Diag(Diag(C))} \ar{r}{F_{**}} & \cat{Diag(Diag(D))}
\end{tikzcd}
$$
maps a lax cocone with tip $\cat{C}$ to the lax cocone with tip $\cat{D}$ obtained simply via postcomposition with $F$. 
$$
\begin{tikzcd}[row sep=small]
 D_0J
  \ar[""{name=D,below}]{dr}{D_1J} 
  \ar{dd}[swap]{D_0j} \\
 & \cat{C} \\
 D_0J'
  \ar{ur}[swap]{D_1J'} 
  \ar[Rightarrow,from=D, "D_1j", shorten <= 0.5em, shorten >= 1em, swap, near start]
\end{tikzcd}
\qquad\rightsquigarrow\qquad
\begin{tikzcd}[row sep=small]
 D_0J
  \ar[""{name=D,below}]{dr}{D_1J} 
  \ar{dd}[swap]{D_0j} \\
 & \cat{C} \ar{r}{F} & \cat{D} \\
 D_0J'
  \ar{ur}[swap]{D_1J'} 
  \ar[Rightarrow,from=D, "D_1j", shorten <= 0.5em, shorten >= 1em, swap, near start]
\end{tikzcd}
$$
In other words, $F_{**}(D_0,D_1)=(D_0, F\circ D_1)$. If we take the Grothendieck construction in both cases we get morphisms $\mu(D):\groth D_0 \to \cat{C}$ and $\mu(F_{**}D):\groth D_0\to \cat{D}$. By uniqueness (\Cref{oplaxcolimit}), necessarily $\mu(F_**D)=F\circ\mu(D)$, and therefore $\mu$ is strictly natural.
\end{proof}

\subsection{Unitors and associators}

\subsubsection{Left unitor}

Let $D:\cat{J}\to\cat{C}$ be a diagram. We can apply to it the unit $\eta_{\cat{Diag(C)}}:\cat{Diag(C)}\to\cat{Diag(Diag(C))}$ to form the diagram 
$$
\begin{tikzcd}
 \cat{1} \ar{r}{(\cat{J}, D)} & \cat{Diag(C)}
\end{tikzcd}
$$
corresponding to the following (rather trivial) lax cocone in $\cat{CAT}$ with tip $C$.
$$
\begin{tikzcd}
 \cat{J} \ar{r}{D} \ar{r} & \cat{C}
\end{tikzcd}
$$
We view this as a lax cocone over the diagram $\cat{J}:\cat{1}\to\cat{CAT}$ which maps the unique object of $\cat{1}$ to $\cat{J}$.
If we form the Grothendieck construction as prescribed by the multiplication of the monad, we get the category $\groth \cat{J}$ which is isomorphic to $\cat{J}$, explicitly given as follows:
\begin{itemize}
 \item Objects are pairs $(\bullet, X)$, where $\bullet$ is the unique object of $\cat{1}$ and $X$ is an object of $\cat{J}$;
 \item A morphism $(\bullet, X)\to (\bullet, Y)$ is simply a morphism $f:X\to Y$ of $\cat{J}$.
\end{itemize}
The functor $\mu(\eta(D))$ maps then $(\bullet, X)$ to $DX$ and $f:(\bullet, X)\to (\bullet, Y)$ to $Df$ in $\cat{C}$.

The functor $\groth \cat{J}\to\cat{J}$ given by $(\bullet, X)\mapsto X$ is an isomorphism of categories.
This defines an isomorphism of diagrams $\mu_\cat{C}(\eta_{\cat{Diag(C)}}(D))\to D$, in the category $\cat{Diag(C)}$.
Denote this isomorphism by $\ell$. 
This is the map that we take as left unitor. 

\begin{prop}
 The map $\ell$ induces a modification $\mu\circ(\eta\,\cat{Diag}) \Rrightarrow \id$.
\end{prop}

\begin{proof}
Let's first show that $\ell$ is natural in the diagram $D$. If $(R,\rho):(\cat{J},D)\to (\cat{K},E)$ is a morphism of diagrams, it's easy to see that this diagram commutes strictly,
$$
\begin{tikzcd}
 \groth \cat{J} \ar{d}{\tilde R} \ar{r}{\ell} & \cat{J} \ar{d}{R}\\
 \groth \cat{K} \ar{r}{\ell} & \cat{K}
\end{tikzcd}
$$
where $\tilde R$ is the functor mapping $(\bullet, X)$ to $(\bullet, RX)$, and acting similarly on morphisms. This commutative diagram induces an analogous commutative diagram in $\cat{Diag(C)}$, so that $\ell$ is a natural isomorphism of functors 
$$
\begin{tikzcd}[column sep=small]
\cat{Diag(C)} \ar[""{name=ID}]{dr}[swap]{\id} \ar{r}{\eta} & \cat{Diag(Diag(C))} \ar[Rightarrow, to=ID, "\ell", shorten >= 0.2em, shorten <= 0.2em] \ar{d}{\mu} \\
& \cat{Diag(C)}
\end{tikzcd}
$$

In order to show that $\ell$ is a modification, we have to show that $\ell$ is natural in the category $\cat{C}$ as well, in the sense that for each functor $F:\cat{C}\to\cat{D}$ the following composite 2-cells are equal,
$$
\begin{tikzcd}[column sep=tiny]
& \cat{Diag(D)} \ar{r}{\eta} & \cat{Diag(Diag(D))}  \ar{d}{\mu} \\
\cat{Diag(C)} \ar{r}{\eta} \ar{ur}{F_*} \ar[""{name=ID2}]{dr}[swap]{\id} & \cat{Diag(Diag(C))} \ar{ur}{F_{**}} \ar{d}{\mu} \ar[Rightarrow, to=ID2, "\ell", shorten >= 0.2em, shorten <= 0.2em] & \cat{Diag(D)} \\
& \cat{Diag(C)} \ar{ur}[swap]{F_*}
\end{tikzcd}
\;=\quad\;\,
\begin{tikzcd}[column sep=tiny]
& \cat{Diag(D)} \ar[""{name=ID}]{dr}[swap]{\id} \ar{r}{\eta} & \cat{Diag(Diag(D))} \ar[Rightarrow, to=ID, "\ell", shorten >= 0.2em, shorten <= 0.2em] \ar{d}{\mu} \\
\cat{Diag(C)} \ar{ur}{F_*} \ar[""{name=ID2}]{dr}[swap]{\id} & & \cat{Diag(D)} \\
& \cat{Diag(C)} \ar{ur}[swap]{F_*}
\end{tikzcd}
$$
where the squares without a 2-cell commute (by naturality).

Explicitly, we have to check that given a diagram $D:\cat{J}\to\cat{C}$, the following parallel functors coincide.
$$
\begin{tikzcd}
 \mu(\eta(F_*D)) \ar[shift left]{r}{F_*\ell} \ar[shift right]{r}[swap]{\ell} & F_* D
\end{tikzcd}
$$
Note however that $F_*$ acts on the codomain of the diagram, by postcomposing with $F$, while $\ell$ acts on the domain of the diagram, mapping $\cat{J}$ to its isomorphic copy $\groth\cat{J}$. Therefore both arrows give the same morphism of diagrams, explicitly the following commutative diagram of $\cat{CAT}$,
$$
\begin{tikzcd}[row sep=small]
 \groth \cat{J}
  \ar[""{name=D,below}]{dr}{D} 
  \ar{dd}[swap]{\cong} \\
 & \cat{C} \ar{r}{F} & \cat{D} \\
 \cat{J}
  \ar{ur}[swap]{\tilde D} 
\end{tikzcd}
$$
where $\tilde D(\bullet, X)=D(X)$, and the action on morphisms is similarly defined.
Therefore $\ell$ is a modification. 
\end{proof}

\subsubsection{Right unitor}\label{rightunitdiag}

Let $D:\cat{J}\to\cat{C}$ be a diagram. This time we apply to it the map $\eta_*:\cat{Diag(C)}\to\cat{Diag(Diag(C))}$ given by the functor image of $\eta$ under $\cat{Diag}$. Explicitly, the result is the following lax cocone in $\cat{CAT}$, with tip $\cat{C}$, indexed by $\cat{J}$ via the constant functor $\Delta \cat{1}:\cat{J}\to\cat{CAT}$ at $\cat{1}$. In pictures,
$$
\begin{tikzcd}[row sep=small]
 J \ar{dd}[swap]{j} \\
 \\
 J'
\end{tikzcd}
\qquad\rightsquigarrow\qquad
\begin{tikzcd}[row sep=small]
 \cat{1}
  \ar[""{name=D,below}]{dr}{DJ} 
  \ar{dd} \\
 & \cat{C}  \\
 \cat{1}
  \ar{ur}[swap]{DJ'} 
  \ar[Rightarrow,from=D, "j", shorten <= 0.5em, shorten >= 1em, swap, near start]
\end{tikzcd}
$$

If we form the Grothendieck construction, this time we get the category $\groth \cat{\Delta 1}$ which is again isomorphic to $\cat{J}$, explicitly given as follows:
\begin{itemize}
 \item Objects are pairs $(X,\bullet)$, where $X$ is an object of $\cat{J}$ and $\bullet$ is the unique object of $\cat{1}$; 
 \item A morphism $(X,\bullet)\to (Y,\bullet)$ is simply a morphism $f:X\to Y$ of $\cat{J}$.
\end{itemize}
The functor $\mu(\eta_*(D))$ maps $(X,\bullet)$ to $DX$ and $f:(X,\bullet)\to (Y,\bullet)$ to $Df$ in $\cat{C}$.

Analogously to the left unitor case, we have a functor $\groth \Delta\cat{1}\to\cat{J}$ given by $(X,\bullet)\mapsto X$ which induces an isomorphism of categories.
This defines in turn an isomorphism of diagrams $\mu(\eta_*(D))\to D$, which we denote by $r^{-1}$ (and its inverse by $r$. See \Cref{defpseudomonad} for the convention we are using). 
The map $r$ is the one that we take as right unitor. 

\begin{prop}
 The map $r$ induces a modification $\id\Rrightarrow\mu\circ\eta_* $.
\end{prop}

We omit the proof, since it is analogous to the case of $\ell$.

\subsubsection{Associator}\label{assdiag}

In order to define the associator, we have to look at $\cat{Diag(Diag(Diag(C)))}$. 
So let $D:\cat{J}\to\cat{Diag(Diag(C))}$ be a diagram which assigns to each object $J$ of $\cat{J}$ a diagram of diagrams $D_1J:D_0J\to\cat{Diag(C)}$, itself mapping the object $K$ of $D_0J$ to the diagram $(D_1J)_1:(D_1J)_0K\to\cat{C}$, which, in turn, maps an object $L$ of $(D_1J)_0K$ to the object $(D_1J)_1(L)$ of $\cat{C}$. We could depict the situation as follows. For brevity we omit the action on morphisms, which is similarly constructed. 
$$
\begin{tikzcd}[row sep=0, column sep=tiny]
	{\cat{J}} && {\cat{Diag(Diag(C))}} \\
	{J} & {\rightsquigarrow} & {\bigg\{ D_0J } && {\cat{Diag(C)} \bigg\}} \\
	&& {K} & {\rightsquigarrow} & {\bigg\{ (D_1J)_0K} && {\cat{C} \bigg\}} \\
	&&&& {L} & {\rightsquigarrow} & {((D_1J)_1K)(L)}
	\arrow["{D}", from=1-1, to=1-3]
	\arrow["{D_1J}", from=2-3, to=2-5]
	\arrow["{(D_1J)_1K}", from=3-5, to=3-7]
\end{tikzcd}
$$
We can now take the Grothendieck construction at two different depths, which we can think of as ``joining levels $J$ and $K$'' and ``joining levels $K$ and $L$''. The former way, which is $\mu(D)\in \cat{Diag(Diag(C))}$, gives the following diagram of diagrams (only two levels).
$$
\begin{tikzcd}[row sep=0, column sep=tiny]
	{\groth D_0} && {\cat{Diag(C)}} \\
	{(J\in\cat{J}, K\in D_0J)} & {\rightsquigarrow} & {\bigg\{ (D_1J)_0K} && {\cat{C} \bigg\}} \\
	&& {L} & {\rightsquigarrow} & {((D_1J)_1K)(L)}
	\arrow["{\mu(D)}", from=1-1, to=1-3]
	\arrow["{(D_1J)_1K}", from=2-3, to=2-5]
\end{tikzcd}
$$
The latter way, which is $\mu_*D\in \cat{Diag(Diag(C))}$, gives instead the following diagram of diagrams, which is in general not isomorphic to the former. 
$$
\begin{tikzcd}[row sep=0, column sep=tiny]
	{\cat{J}} && {\cat{Diag(C)}} \\
	{J} & {\rightsquigarrow} & {\bigg\{ \groth (D_1J)_0} && {\cat{C} \bigg\}} \\
	&& {(K\in D_0J, L\in (D_1J)_0K)} & {\rightsquigarrow} & {((D_1J)_1K)(L)}
	\arrow["{\mu_*D}", from=1-1, to=1-3]
	\arrow["{\mu(D_1J)}", from=2-3, to=2-5]
\end{tikzcd}
$$
If we apply once again the Grothendieck construction to the two, we do obtain isomorphic diagrams:
$$
\begin{tikzcd}[row sep=0, column sep=tiny]
    {\groth \mu(D)} && {\cat{C}} \\
	{((J\in\cat{J}, K\in D_0J), L\in (D_1J)_0K)} & {\rightsquigarrow} & {((D_1J)_1K)(L)}
	\arrow[from=1-1, to=1-3]
\end{tikzcd}
$$
and
$$
\begin{tikzcd}[row sep=0, column sep=tiny]
    {\groth \mu_*D} && {\cat{C}} \\
	{(J\in\cat{J}, (K\in D_0J, L\in (D_1J)_0K))} & {\rightsquigarrow} & {((D_1J)_1K)(L)}
	\arrow[from=1-1, to=1-3]
\end{tikzcd}
$$
These are isomorphic as diagrams through the map 
$$
\begin{tikzcd}[row sep=0, column sep=tiny]
    {\groth \mu_*D} && {\groth \mu(D)} \\
	{(J, (K, L))} & {\rightsquigarrow} & {((J, K), L)}
	\arrow[from=1-1, to=1-3]
\end{tikzcd}
$$
and so the following diagram commutes up to isomorphism.
$$
\begin{tikzcd}[sep=tiny]
	{\cat{Diag(Diag(Diag(C)))}} && {\cat{Diag(Diag(C))}} \\
	& {\cong} \\
	{\cat{Diag(Diag(C))}} && {\cat{Diag(C)}}
	\arrow["{\mu_*}", from=1-1, to=1-3]
	\arrow["{\mu}"', from=1-1, to=3-1]
	\arrow["{\mu}", from=1-3, to=3-3]
	\arrow["{\mu}"', from=3-1, to=3-3]
\end{tikzcd}
$$
We call this isomorphism the \emph{associator}, and denote it by $a$. Again, analogously as for the unitors, we have:
\begin{prop}
 The associators assemble to a modification $\mu\circ\mu_* \Rrightarrow \mu\circ\mu$.
\end{prop}

\subsection{Higher coherence laws}

The associator and unitors satisfy coherence conditions that are analogous to the ones of monoidal categories (see \Cref{defpseudomonad} for the precise definition). We spell them out in detail for this case.

\subsubsection{Unit condition}

Instantiating the unit condition of \Cref{defpseudomonad} in our case, we get the following statement, which reminds us of the unit condition of monoidal categories.

\begin{itemize}
 \item Consider a diagram of diagrams as follows.
 $$
 \begin{tikzcd}[row sep=0, column sep=tiny]
	{\cat{J}} && {\cat{Diag(C)}} \\
	{J} & {\rightsquigarrow} & {\bigg\{ D_0J} && {\cat{C} \bigg\}} \\
	&& {K} & {\rightsquigarrow} & {(D_1J)(K)}
	\arrow["{D}", from=1-1, to=1-3]
	\arrow["{D_1J}", from=2-3, to=2-5]
\end{tikzcd}
 $$
 Applying directly the Grothendieck construction we would have a diagram $(J,K)\mapsto (D_1J)(K)$. However, we instead want to insert a ``bullet'' via the unitor, and this can be done in two ways.
 
 \item We can either apply the (inverse of the) left unitor $\ell$ at depth $K$, and then take the Grothendieck construction, obtaining the following diagram.
 $$
 \begin{tikzcd}[row sep=0, column sep=tiny]
	{\groth (\mu_*\eta_*D)_0} && {\cat{C}} \\
	{(J,(\bullet, K))} & {\rightsquigarrow} & {(D_1J)(K)}
	\arrow["{\mu(\mu_*\eta_*D)}", from=1-1, to=1-3]
\end{tikzcd}
 $$ 

 \item Alternatively, we can apply the right unitor $r$ at depth $J$, and again take the Grothendieck construction, obtaining the following \emph{isomorphic} diagram.
 $$
 \begin{tikzcd}[row sep=0, column sep=tiny]
	{\groth(\mu(\eta_*D))_0} && {\cat{C}} \\
	{((J,\bullet), K)} & {\rightsquigarrow} & {(D_1J)(K)}
	\arrow["{\mu(\mu(\eta_*D))}", from=1-1, to=1-3]
\end{tikzcd}
 $$ 
 
 \item Now, not only are the two diagrams isomorphic, but the isomorphism relating them is exactly the associator,
 $$
\begin{tikzcd}[row sep=0, column sep=tiny]
	{(J, (\bullet, K))} & {\rightsquigarrow} & {((J, \bullet), K)}
\end{tikzcd}
$$
which we can view as ``rebracketing''.
\end{itemize}

\subsubsection{Pentagon equation}

Instantiating the pentagon condition of \Cref{defpseudomonad} in our case, we get the following statement, which reminds us of the analogous condition for monoidal categories. Consider a four-level diagram, as follows. 
$$
\begin{tikzcd}[column sep=-0.5em, row sep=0]
	{\cat{J}} && {\cat{Diag(Diag(Diag(C)))}} \\
	{J} & {\rightsquigarrow} & {\bigg\{ D_0J } && {\cat{Diag(Diag(C))} \bigg\}} \\
	&& {K} & {\rightsquigarrow} & {\bigg\{ (D_1J)_0 K} && {\cat{Diag(C)} \bigg\}} \\
	&&&& {L} & {\rightsquigarrow} & {\bigg\{ ((D_1J)_1K)_0L} && {\cat{C} \bigg\}} \\
	&&&&&& {M} & {\rightsquigarrow} & {(((D_1J)_1K)_1L)(M)}
	\arrow["{D}", from=1-1, to=1-3]
	\arrow[from=2-3, to=2-5]
	\arrow["{(D_1J)_1K}", from=3-5, to=3-7]
	\arrow["{((D_1J)_1K)_1L}", from=4-7, to=4-9]
\end{tikzcd}
$$
There are several ways of obtaining a depth-one diagram via applying the Grothendieck construction three times, and they are related to one another via the associators. In particular, we can apply the Grothendieck construction repeatedly starting from the deepest (rightmost) level,
$$
\begin{tikzcd}[column sep=tiny, row sep=0]
    \groth \mu_*\mu_{**} D \ar{rr}{\mu(\mu_*\mu_{**} D)} & & \cat{C} \\
	(J,(K,(L,M))) & \rightsquigarrow & (((D_1J)_1K)_1L)(M)
\end{tikzcd}
$$
or we can start from the outermost (leftmost) level.
$$
\begin{tikzcd}[column sep=tiny, row sep=0]
    \groth \mu(\mu(D)) \ar{rr}{\mu(\mu(\mu(D)))} & & \cat{C} \\
	(((J,K),L),M) & \rightsquigarrow & (((D_1J)_1K)_1L)(M)
\end{tikzcd}
$$
There are now two ways of obtaining the former from the latter via associators, and they are equal. They are induced by the following rebracketings, which form a commutative pentagon (analogous to the one of monoidal categories).
$$
\begin{tikzcd}[sep=tiny]
	&& {(J,((K,L),M))} \\
	{(J,(K, (L,M)))} \\
	&&& {((J,(K,L)),M)} \\
	{((J,K),(L,M))} \\
	&& {(((J,K),L),M)}
	\arrow["{a}"', from=2-1, to=4-1, squiggly]
	\arrow["{a}"', from=4-1, to=5-3, squiggly]
	\arrow["{a_*}", from=2-1, to=1-3, squiggly]
	\arrow["{a}", from=1-3, to=3-4, squiggly]
	\arrow["{a}", from=3-4, to=5-3, squiggly]
\end{tikzcd}
$$

\subsection{Cocomplete categories are algebras}\label{seccocompalg}

In this section we prove the following statements.

\begin{thm}\label{cocompalg}
 Every cocomplete category $\cat{C}$ equipped with a choice of colimit for each diagram has the structure of a pseudoalgebra over $\cat{Diag}$.
\end{thm}

As shown in \Cref{secnotallalg}, the converse of the theorem does not hold: not all pseudoalgebras are in this form.

A definition of pseudoalgebra over a pseudomonad is given in \Cref{pseudomonads}.

\subsubsection{Structure map}

Let $\cat{C}$ be a cocomplete category.
For each diagram $D:\cat{J}\to\cat{C}$, choose a colimit (they are all isomorphic, pick one in each equivalence class). Let's see why this construction is functorial.
A colimit does not just consist of an object of $\cat{C}$, but also of the arrows of the colimiting cocone. 
Denote by $c(D)$ the (chosen) colimit object of $D$, and by $h(D):D\Rightarrow c(D)$ the colimiting cocone. In components, the cocone consists of arrows
$$
\begin{tikzcd}[column sep=large]
 DJ \ar{r}{h(D)_J} & c(D)
\end{tikzcd}
$$
for each object $J$ of $\cat{J}$.
Now consider a morphism of diagrams as follows. 
$$
\begin{tikzcd}[row sep=small]
 \cat{J} 
  \ar[""{name=D,below}]{dr}{D} 
  \ar{dd}[swap]{R} \\
 & \cat{C} \\
 \cat{J'} 
  \ar{ur}[swap]{D'} 
  \ar[Rightarrow,from=D, "\rho", shorten <= 0.5em, shorten >= 1em, near start]
\end{tikzcd}
$$
Using $\rho$ and the colimit cocone $h(D'):D'\Rightarrow c(D')$ we can construct a cocone under $D$, with tip $c(D')$: the one of components
$$
\begin{tikzcd}[column sep=large]
 DJ \ar{r}{\rho_J} & D'RJ \ar{r}{h(D')_{RJ}} & c(D') 
\end{tikzcd}
$$
for each $J$ of $\cat{J}$. This cocone must then factor uniquely through $c(D)$ by the universal property of $c(D)$ as a colimit:
$$
\begin{tikzcd}[column sep=small, row sep=tiny]
 DJ \ar{dd}[swap]{h(D)_J} \ar{dr}{\rho_J} \\
 & D'RJ \ar{dr}{h(D')_{RJ}} \\
 c(D) \ar[dotted]{rr} & & c(D') 
\end{tikzcd}
$$
Denote the resulting map $c(D)\to c(D')$ by $c(R,\rho)$. It is the unique map that makes the diagram above commute for each $J$ of $\cat{J}$. By uniqueness, this assignment preserves identities and composition, and so $c$ is a functor $\cat{Diag(C)}\to\cat{C}$.
Technically speaking, for each diagram one can choose many possible colimits within the same equivalence class. However, once $c(D)$ and $c(D')$ are fixed, the map $c(D)\to c(D')$ is unique. So any choice of such colimit objects gives rise to a functor, and all these functors will be naturally isomorphic, again by uniqueness. 

We can say even more: the map $c:\cat{Diag(C)}\to\cat{C}$ is even \emph{2-functorial} if we view $\cat{Diag(C)}$ as a 2-category (as in \Cref{defdiag}), and $\cat{C}$ (which is a 1-category) as a locally discrete 2-category.
This is made precise by the following lemma.

\begin{lemma}\label{c2functor}
 Let $\cat{C}$ be a cocomplete category and $c:\cat{Diag(C)}\to\cat{C}$ a choice of colimit for each diagram. 
 Consider diagrams $(\cat{J},D)$ and $(\cat{J'},D')$ in $\cat{C}$, morphisms of diagrams $(R,\rho), (R',\rho'): (\cat{J},D) \to (\cat{J'},D')$, and suppose there exists a 2-cell of diagrams $\alpha:(R,\rho)\Rightarrow (R',\rho')$. 
 Then $c(R,\rho) = c(R',\rho')$.
\end{lemma}

\begin{proof}
 Recall that $\alpha$ consists of a natural transformation $\alpha:R\Rightarrow R'$ such that
$$
\begin{tikzcd}[column sep=large]
 \cat{J} 
  \ar[""{name=D,below, pos=0.7}]{dr}{D} 
  \ar[bend right, ""{name=RP, right}]{dd}[swap]{R'} \ar[bend left, ""{name=R, left}]{dd}[pos=0.4]{R} \\
 & \cat{C} \\
 \cat{J'} 
  \ar{ur}[swap]{D'} 
  \ar[Rightarrow,from=D, "\rho", shorten <= 0.5em, shorten >= 1.5em, near start]
  \ar[Rightarrow,from=R,to=RP,"\alpha"]
\end{tikzcd}
\qquad=\qquad\begin{tikzcd}[column sep=large]
 \cat{J} 
  \ar[""{name=D,below}]{dr}{D} 
  \ar[bend right]{dd}[swap]{R'} \\
 & \cat{C} . \\
 \cat{J'} 
  \ar{ur}[swap]{D'} 
  \ar[Rightarrow,from=D, "\rho'", shorten <= 0.5em, shorten >= 1em, swap]
\end{tikzcd}
$$
 This way, for each object $J$ of $\cat{J}$, the following diagram commutes,
 $$
 \begin{tikzcd}[row sep=small]
 & D'RJ \ar{dd}{D\alpha_J} \ar{dr}{h(D')_{R'J}} \\
 DJ \ar{ur}{\rho_J} \ar{dr}[swap]{\rho'_{J}} && c(D') \\
 & D'R'J \ar{ur}[swap]{h(D')_{RJ}}
 \end{tikzcd}
 $$
 where $h(D'):D'\Rightarrow c(D')$ denotes the colimit cone of $c(D')$ under $D'$. 
 The maps $c(R,\rho)$ and $c(R',\rho'): c(D)\to c(D')$ are both defined as the unique maps making the following diagram commute,
 $$
 \begin{tikzcd}
 DJ \ar{dr}{h(D')_{R'J}\circ\rho_J \,=\, h(D')_{RJ}\circ\rho'_J} \ar{d}[swap]{h(D)_J} \\
 c(D) \ar[dotted]{r} & c(D')
 \end{tikzcd}
 $$
 and are therefore equal.
\end{proof}

\subsubsection{Structure 2-cells}

We now give the structure 2-cells of the pseudoalgebras, namely the unitor and the multiplicator.

\begin{lemma}
 The following diagram commutes up to a canonical natural isomorphism. 
 $$
 \begin{tikzcd}[row sep=tiny, column sep=-1em]
  \cat{C} \ar{ddrr}[swap]{\id} \ar{rr}{\eta} &[+50pt] & \cat{Diag(C)} \ar{dd}{c} \\[-5pt]
  & \cong \\[+10pt]
  && \cat{C}
 \end{tikzcd}
 $$
\end{lemma}

We denote the natural isomorphism by $\iota: c\circ\eta\Rightarrow \id_C$.

\begin{proof}
 Let $C$ be an object of $\cat{C}$. The diagram $\eta(C)$ is the one-object diagram whose unique node is given by $C$. A colimit cocone over $\eta(C)$ consists of an object $C'$ together with a specified isomorphism $C\to C'$. Therefore, for any choice of $c$, we get canonically an isomorphism $\iota_C:C\to c(\eta(C))$. The maps $\iota_C$ assemble to a natural isomorphism $\iota:c\circ\eta\Rightarrow \id_C$, since for each $f:C\to C'$ of $\cat{C}$ the following diagram commutes,
 $$
 \begin{tikzcd}
  C \ar{r}{f} \ar[leftarrow]{d}{\iota_C} & C' \ar[leftarrow]{d}{\iota_{C'}} \\
  c(\eta(C)) \ar{r}{c(\eta(f))} & c(\eta(C'))
 \end{tikzcd}
 $$
 since the map $c(\eta(f))$ is defined (by definition of how $c$ acts on morphisms) as the unique map making the diagram above commute.
\end{proof}

\begin{lemma}\label{decomplemma}
 The following diagram commutes up to a canonical natural isomorphism. 
 $$
 \begin{tikzcd}[row sep=small, column sep=0]
  \cat{Diag(Diag(C))} \ar{dd}{\mu} \ar{rr}{c_*} && \cat{Diag(C)} \ar{dd}{c} \\
  & \cong \\
  \cat{Diag(C)} \ar{rr}{c} && \cat{C}
 \end{tikzcd}
 $$
\end{lemma}

Denote the natural isomorphism by $\gamma: c\circ c_* \Rightarrow c\circ\mu$. 

This statement is known in the literature, see for example \cite[Section~40]{homdiag}, as well as \cite[Theorem~3.2]{decomposition}.
Here we present a direct proof. In the proof we can see 
how the objects in the top right corner of the square are, in some sense, partial colimits of the objects in the bottom left corner. This will be made precise in \Cref{partialcolimits}.

\begin{proof}
 Let $D:\cat{J}\to\cat{Diag(C)}$.
 The diagram $c_*D:\cat{J}\to\cat{C}$ is given by the following postcomposition.
 $$
\begin{tikzcd}[row sep=small]
\cat{J} \ar{rrr}{D} && &[-3.5em] \cat{Diag(C)} \ar{rrr}{c} &[-2.5em] && \cat{C} \\
 J \ar{dd}[swap]{j} 
& &  D_0J
  \ar[""{name=D,below}]{drr}{D_1J} 
  \ar{dd}[swap]{D_0j}
& && & 
c(DJ) \ar{dd}{c(Dj)}
  \\
& \rightsquigarrow & & & \cat{C} & \rightsquigarrow \\
 J' & & D_0J'
  \ar{urr}[swap]{D_1J'} 
  \ar[Rightarrow,from=D, "D_1j", shorten <= 0.5em, shorten >= 1em, swap, near start]
& && & c( D_1J')
\end{tikzcd}
$$
 In other words, the nodes of $D$ are diagrams (one for each object of $\cat{J}$), and $c_*$ replaces them by their (chosen) colimit, obtaining a diagram in $\cat{C}$ indexed by $\cat{J}$. 
 
 Recall that $\mu$ is given by the Grothendieck construction, so that $\mu(D)$ is a diagram obtained by the union of the diagrams $DJ$ for each $J$ of $\cat{J}$, plus additional arrows between those subdiagrams, induced by the morphisms of $\cat{J}$. Specifically, for every morphism $j:J\to J'$ of $\cat{J}$ and for every $f:X\to Y$ of $D_0J$, the following square commutes.
 $$
 \begin{tikzcd}
 D_1J (X) \ar{d}{D_1Jf} \ar{r}{(D_1j)_X} & D_1J'(D_0j(X)) \ar{d}{D_1J'(D_0j(f))} \\
 D_1J (Y)  \ar{r}{(D_1j)_Y} & D_1J'(D_0j(Y))
 \end{tikzcd}
 $$
 The object $c(\mu(D))$ is a colimit of the resulting diagram involving all the $j$ and the $f$ as above.
 Recall now the universal property of the Grothendieck construction, and in particular diagrams \eqref{ijstrict} and \eqref{ijlax}. For each object $J$ of $\cat{J}$, the morphism of diagrams given by the inclusion $i_J$ of the fiber over $J$, 
 $$
 \begin{tikzcd}[row sep=small]
  D_0 J \ar{dd}[swap]{i_J} \ar{dr}{D_1J} \\
  & \cat{C} \\
  \groth D_0 \ar{ur}[swap]{\mu(D)}
 \end{tikzcd}
 $$
 induces a map between their (chosen) colimits $c(i_J,\id):c(DJ) \to c(\mu(D))$.
 The maps $c(i_J,\id):c(DJ) \to c(\mu(D))$, for each $J$, assemble to a cocone under $c_*D:\cat{J}\to\cat{C}$, with tip $c(\mu(D))$, meaning that for each morphism $j:J\to J'$ of $\cat{J}$, the following diagram commutes.
 \begin{equation}\label{coconeofcolims}
 \begin{tikzcd}[row sep=small]
  c(DJ) \ar{dd}[swap]{c(Dj)} \ar{dr}{c(i_{J},\id)} \\
  & c(\mu(D)) \\
  c(DJ') \ar{ur}[swap]{c(i_{J'},\id)}
 \end{tikzcd}
 \end{equation}
 Indeed, we can rewrite \eqref{ijlax} as follows,
 $$
\begin{tikzcd}[column sep=large]
 D_0J
  \ar[""{name=D,below, pos=0.7}]{dr}{D_1J} 
  \ar[bend right, ""{name=RP, right}]{dd}[swap]{i_{J'}\circ D_0j} \ar[bend left, ""{name=R, left}]{dd}[pos=0.4]{i_J} \\
 & \cat{C} \\
 \groth D_0
  \ar{ur}[swap]{\mu(D)} 
  \ar[equal,from=D, shorten <= 0.5em, shorten >= 1.5em, near start]
  \ar[Rightarrow,from=R,to=RP,"i_j"]
\end{tikzcd}
\qquad=\qquad\begin{tikzcd}[column sep=large]
 D_0J
  \ar[""{name=D,below}]{dr}{D_1J} 
  \ar[bend right]{dd}[swap]{i_{J'}\circ D_0j} \\
 & \cat{C} \\
 \groth D_0 
  \ar{ur}[swap]{\mu(D)} 
  \ar[Rightarrow,from=D, "D_1j", shorten <= 0.5em, shorten >= 1em, swap]
\end{tikzcd}
$$
 which is exactly the condition for a 2-cells of diagrams $(i_J,\id)\Rightarrow (i_{J'}\circ D_0J,D_1j)$ (see \Cref{defdiag}). 
 By \Cref{c2functor}, then, $c(i_J,\id)=c(i_{J'}\circ D_0J,D_1j)$, which means that \eqref{coconeofcolims} commutes.
 
 As we said, the maps $c(i_J,\id):c(DJ) \to c(\mu(D))$, assemble to a cocone under $c_*D:\cat{J}\to\cat{C}$, with tip $c(\mu(D))$. Therefore, by the universal property of $c(c_*D)$ as a colimit, there exists a unique arrow $c(c_*D)\to c(\mu(D))$, which we denote by $\gamma_D$, making the following diagram commute for all $J$ of $\cat{J}$,
 $$
 \begin{tikzcd}
  c(DJ) \ar{dr}{c(i_J,\id)} \ar{d}[swap]{h(c_*D)_J} \\
  c(c_*D) \ar[dotted]{r}[swap]{\gamma_D} & c(\mu(D))
 \end{tikzcd}
 $$
 where the $h(c_*D)_J$ denote the arrows of the colimiting cocone. Note that for each object $X$ of $D_0J$ we can extend the diagram above to the following commutative diagram,
 \begin{equation}\label{twocones}
 \begin{tikzcd}
  D_1J(X) \ar{d}[swap]{h(DJ)_X} \ar[bend left]{ddr}{h(\mu(D))_X} \\
  c(DJ) \ar{dr}[near start]{c(i_J,\id)} \ar{d}[swap]{h(c_*D)_J} \\
  c(c_*D) \ar[dotted]{r}[swap]{\gamma_D} & c(\mu(D))
 \end{tikzcd}
 \end{equation}
 where $h(DJ)_X$ and $h(\mu(D))_X$ are the components at $X$ of the colimiting cocones of $c(DJ)$ and $c(\mu(D))$.
 
 To show that $\gamma_D$ is an isomorphism, we invoke the Yoneda embedding. Let $K$ be any object of $\cat{C}$. We want to show that the function 
 \begin{equation}\label{hastoiso}
 \begin{tikzcd}[row sep=0]
  \Hom_\cat{C}\big(c(\mu(D)), K\big) \ar{r}{{\gamma_D}_*} & \Hom_\cat{C}\big(c(c_*D), K\big) \\
  f \ar[mapsto]{r} & f\circ\gamma_D ,
 \end{tikzcd}
 \end{equation}
 which is natural in $K$, is a bijection. 
 To this end, we note that by the universal property of colimits, the set on the left is naturally isomorphic (via composing with the components of $h(\mu(D))$) to the subset  
 $$
 S\; \subseteq \;\prod_{J\in\cat{J}} \prod_{X\in D_0J} \Hom_\cat{C}\big( D_1J(X), K \big)
 $$
 whose elements are families of arrows $(k_{J,X}: D_1J(X) \to K)$ such that for each $j:J\to J'$ of $\cat{J}$ and each $f:X\to Y$ of $D_0J$ the following diagram commutes.
 $$
 \begin{tikzcd}[sep=large]
  	{D_1J(X)} && {D_1J'(D_0j(X))} \\
	& {K} \\
	{D_1J(Y)} && {D_1J'(D_0j(Y))}
	\arrow["{k_{J,X}}"', from=1-1, to=2-2]
	\arrow["{k_{J,Y}}", from=3-1, to=2-2]
	\arrow["{D_1J(f)}"', from=1-1, to=3-1]
	\arrow["{D_1J(D_0j(f))}", from=1-3, to=3-3]
	\arrow["{(D_1j)_X}", from=1-1, to=1-3]
	\arrow["{(D_1j)_Y}"', from=3-1, to=3-3]
	\arrow["{k_{J', D_0j(X)}}", from=1-3, to=2-2]
	\arrow["{k_{J', D_0j(Y)}}"', from=3-3, to=2-2]
 \end{tikzcd}
 $$
 Now, for each $J$ of $\cat{J}$, the quantity appearing in $S$ given by 
 $$
 S_J \;\subseteq\; \prod_{X\in D_0J} \Hom_\cat{C}\big( D_1J(X), K \big)
 $$
 and whose elements are arrows $(k_{J,X}: D_1J(X) \to K)$ such that $f:X\to Y$ of $D_0J$ the following diagram commutes,
 $$
 \begin{tikzcd}
  	{D_1J(X)}  \\
	& {K} \\
	{D_1J(Y)} 
	\arrow["{k_{J,X}}", from=1-1, to=2-2]
	\arrow["{k_{J,Y}}"', from=3-1, to=2-2]
	\arrow["{D_1J(f)}"', from=1-1, to=3-1]
 \end{tikzcd}
 $$
 is in natural bijection with 
 $$
 \Hom_\cat{C}\big( c(DJ), K \big)
 $$
 by universal property of the colimit, via composing with the cocones $h(DJ)$. In other words, $S$ is in natural bijection with the subset of 
 $$
 S' \;\subseteq\;\prod_{J\in \cat{J}} \Hom_\cat{C}\big( c(DJ), K \big) 
 $$
 whose elements are arrows $(k'_{J}: c(DJ) \to K)$ such that the following diagram commutes.
 $$
 \begin{tikzcd}
	{c(DJ)} && {C(DJ')} \\
	& {K}
	\arrow["{c(Dj)}", from=1-1, to=1-3]
	\arrow["{k'_J}"', from=1-1, to=2-2]
	\arrow["{k'_{J'}}", from=1-3, to=2-2]
\end{tikzcd}
 $$
 The subset $S'$, again by the universal property of the colimit, is in bijection (via composing with $h(c_*D)$) with $\Hom_\cat{C}\big(c(c_*D), K\big)$, which is exactly at the right side of \eqref{hastoiso}. Since \eqref{twocones} commutes, composing with $\gamma_D$ has the same effect as applying the bijections given by (the inverse of) composing with $h(c_*D)$ and $h(DJ)_X$ (all the bijections are invertible), and then composing with $h(\mu(D))_X$. Therefore \eqref{hastoiso} is a bijection too. By the Yoneda lemma, then, $\gamma_D$ is an isomorphism. 
\end{proof}

\subsubsection{Coherence laws}\label{cocompcoherence}

In order to prove \Cref{cocompalg} it remains to be checked that the unit and multiplication coherence conditions of \Cref{defpseudoalgebra} hold.
Intuitively, such coherence conditions hold by the ``uniqueness property of maps between colimits''. In other words, not only do objects satisfying the same universal property admit an isomorphism between them, but they admit a unique one compatible with the universal property (in our case, the cocone): while colimit objects of a diagram may have many automorphisms (as objects), colimit \emph{cocones} over the same diagram form a contractible groupoid.

Let's see this more explicitly.
The unit condition of \Cref{defpseudoalgebra}, instantiated in our case, says the following.
Let $\cat{C}$ be a cocomplete locally small category, and construct (choose) the functor $c:\cat{Diag(C)}\to\cat{C}$ as before. 
Consider now a diagram $D:\cat{J}\to\cat{C}$. We can apply the map $\eta_*:\cat{Diag(C)}\to\cat{Diag(Diag(C))}$ as in \Cref{rightunitdiag} and obtain the diagram $\eta_*D:\cat{J}\to\cat{Diag(C)}$ as follows.
$$
\begin{tikzcd}[column sep=tiny, row sep=0]
	{\cat{J}} && {\cat{Diag(C)}} \\
	{J} & {\rightsquigarrow} & {\bigg\{ \cat{1} } && {\cat{C} \bigg\}} \\
	&& {\quad \bullet} & {\rightsquigarrow} & {DJ}
	\arrow["\eta_*D", from=1-1, to=1-3]
	\arrow["{DJ}", from=2-3, to=2-5]
\end{tikzcd}
$$
Now we can either 
\begin{itemize}
 \item apply to $\eta_*D$ the map $c_*$, which replaces each one-object diagram $DJ$ with its chosen colimit $c(DJ)$ (isomorphic to $DJ$ via the unitor $\iota$), giving the diagram $c(D-):\cat{J}\to\cat{C}$; or
 \item form the Grothendieck construction and obtain the diagram $(J,\bullet) \mapsto DJ$, with exactly the same image in $\cat{C}$ as $D$, but indexed by a nominally different category, and isomorphic to $D$ via the counit $\rho$ . 
\end{itemize}
Both ways give isomorphic diagrams in $\cat{C}$, which then have isomorphic colimits. The isomorphism between the colimits can be written a priori in two ways: 
\begin{itemize}
 \item It is the one induced by $\gamma: c\circ c_* \Rightarrow c\circ\mu$;
 \item It is the one induced by the isomorphism of diagrams of components $\iota:DJ \to c(DJ)$ for each object $J$ of $\cat{J}$.
\end{itemize}
The unit condition of pseudoalgebras says that these two isomorphism should be equal. This is indeed the case, by uniqueness of the morphism $\gamma$: forming the colimit cocones of $D$ and of $c(D-)$, which are isomorphic diagrams via $\iota$, we have a unique morphism making the following diagram commute for all $J$ of $\cat{J}$,
$$
\begin{tikzcd}
	{DJ} && {c(DJ)} \\
	{c(D)} && {c(c(D-))}
	\arrow["{\iota^{-1}}", from=1-1, to=1-3]
	\arrow["{h(D)}"', from=1-1, to=2-1]
	\arrow["{h(c(D-))}", from=1-3, to=2-3]
	\arrow[from=2-1, to=2-3, dashed]
\end{tikzcd}
$$
which can be seen as either the map $\gamma$ (by definition), or as the map induced by $\iota$, after suitably translating $D$ into $(J,\bullet) \mapsto DJ$ via the right unitor $r$. 
 
The multiplication condition of \Cref{defpseudoalgebra}, again instantiated in our case, says the following.
As in \Cref{assdiag}, let $D\in\cat{Diag(Diag(Diag(C)))}$ be a diagram as follows.
$$
\begin{tikzcd}[row sep=0, column sep=tiny]
	{\cat{J}} && {\cat{Diag(Diag(C))}} \\
	{J} & {\rightsquigarrow} & {\bigg\{ D_0J } && {\cat{Diag(C)} \bigg\}} \\
	&& {K} & {\rightsquigarrow} & {\bigg\{ (D_1J)_0K} && {\cat{C} \bigg\}} \\
	&&&& {L} & {\rightsquigarrow} & {((D_1J)_1K)(L)}
	\arrow["{D}", from=1-1, to=1-3]
	\arrow["{D_1J}", from=2-3, to=2-5]
	\arrow["{(D_1J)_1K}", from=3-5, to=3-7]
\end{tikzcd}
$$
We can now take the colimit progressively, a priori in two ways: first of all, ``from the inside out'', that is,
\begin{itemize}
 \item For each $J$ of $\cat{J}$ and $K$ of $D_0J$, take the (chosen) colimits of the diagrams $(D_1J)_1K$, obtaining the following diagram of diagrams;
$$
\begin{tikzcd}[row sep=0, column sep=tiny]
	{\cat{J}} && {\cat{Diag(C)}} \\
	{J} & {\rightsquigarrow} & {\bigg\{ D_0J} && {\cat{C} \bigg\}} \\
	&& {K} & {\rightsquigarrow} & {c((D_1J)_1K)}
	\arrow["{c_{**}D}", from=1-1, to=1-3]
	\arrow["{c_*(D_1J)}", from=2-3, to=2-5]
\end{tikzcd}
$$
 \item Then, for each $J$ of $\cat{J}$, take the (chosen) colimit of the remaining innermost level diagram $c_*(D_1J)$, obtaining the following diagram;
$$
\begin{tikzcd}[row sep=0, column sep=tiny]
	{\cat{J}} && {\cat{C}} \\
	{J} & {\rightsquigarrow} & {c(c_*(D_1J))} 
	\arrow["{c_*c_{**}D}", from=1-1, to=1-3]
\end{tikzcd}
$$
 \item Finally, take the colimit $c(c_*(c_{**}D))$ of the diagram just obtained.
\end{itemize}
Alternatively, we could
\begin{itemize}
 \item Form the Grothendieck construction of $D$ joining levels $J$ and $K$, obtaining the following diagram of diagrams;
 $$
 \begin{tikzcd}[row sep=0, column sep=tiny]
	{\groth D_0} && {\cat{Diag(C)}} \\
	{(J\in\cat{J}, K\in D_0J)} & {\rightsquigarrow} & {\bigg\{ (D_1J)_0K} && {\cat{C} \bigg\}} \\
	&& {L} & {\rightsquigarrow} & {((D_1J)_1K)(L)}
	\arrow["{\mu(D)}", from=1-1, to=1-3]
	\arrow["{(D_1J)_1K}", from=2-3, to=2-5]
\end{tikzcd}
 $$
 \item Form again the Grothendieck construction, joining level $L$ as well;
 $$
\begin{tikzcd}[row sep=0, column sep=tiny]
    {\groth \mu(D)} && {\cat{C}} \\
	{((J\in\cat{J}, K\in D_0J), L\in (D_1J)_0K)} & {\rightsquigarrow} & {((D_1J)_1K)(L)}
	\arrow["\mu(\mu(D))", from=1-1, to=1-3]
\end{tikzcd}
 $$
 \item Finally, take the colimit $c(\mu(\mu(D)))$ of the resulting diagram.
\end{itemize}

The colimits constructed this way are isomorphic, a priori, in two different ways, using the maps obtained by $\gamma$ in different orders (first inner level, then outer, or vice versa).
However, the different ways coincide, since both colimits come equipped with the following cocones,
$$
\begin{tikzcd}[column sep=large]
	{((D_1J)_1K)(L)} & {c((D_1J)_1K)} & {c(c_*(D_1J))} & {} \\
	{c(\mu(\mu(D)))} && {c(c_*(c_{**}D))}
	\arrow["{h((D_1J)_1K)}", from=1-1, to=1-2]
	\arrow["{h(c_*(D_1J))}", from=1-2, to=1-3]
	\arrow["{h(c_*(c_{**}D))}", from=1-3, to=2-3]
	\arrow["{h(\mu(\mu(D)))}"', from=1-1, to=2-1]
	\arrow[from=2-1, to=2-3, dashed]
\end{tikzcd}
$$
and there is a unique map making the diagram above commute for all $J$, $K$ and $L$.

This finally proves that cocomplete categories are pseudoalgebras of $\cat{Diag}$ (\Cref{cocompalg}).

\subsubsection{Not all algebras are of this form}\label{secnotallalg}

We now want to show the following statement.

\begin{prop}\label{notallalg}
 Not every pseudoalgebra over $\cat{Diag}$ is in the form of \Cref{cocompalg}.
\end{prop}

We use the following known result \cite[Theorem~2.7]{decomposition}.

\begin{thm}\label{colimofdiag}
 Let $\cat{C}$ be a cocomplete category. Then the category $\cat{Diag(C)}$ is cocomplete too. Moreover, the functor $\cat{Diag(C)}\to\cat{Cat}$ which assigns to each diagram $D:\cat{I}\to\cat{C}$ its domain $\cat{I}$ preserves colimits.
\end{thm}

We are now ready to prove the proposition. We will prove it by showing that for \emph{free} pseudoalgebras, in the form $(\cat{Diag(C)},\mu)$, the map $\mu$ is in general not taking colimits of diagrams (of diagrams).

\begin{proof}[Proof of \Cref{notallalg}] 
 Let $\cat{C}$ be a cocomplete category with at least two non-isomorphic objects $X$ and $Y$ and a morphism $f:X\to Y$. Consider now the morphism $\eta(f)$ of $\cat{Diag(C)}$, which can be seen as the morphism of diagrams,
 $$
 \begin{tikzcd}[row sep=small]
 \cat{1} 
  \ar[""{name=D,below}]{dr}{X} 
  \ar{dd}[swap]{\id} \\
 & \cat{C} \\
 \cat{1} 
  \ar{ur}[swap]{Y} 
  \ar[Rightarrow,from=D, "f", shorten <= 0.5em, shorten >= 1em, near start]
\end{tikzcd}
 $$
 and so, in particular, also as a diagram of diagrams (indexed by the walking arrow $\cat{2}$). Denote by $D:\cat{2}\to\cat{Diag(C)}$ this diagram of diagrams. 
 We have that $\mu(D)$, as given by the Grothendieck construction, is a diagram indexed again by $\cat{2}$. Instead, by \Cref{colimofdiag}, the colimit of $D$ in $\cat{Diag(C)}$ is a diagram whose domain must be the colimit of $\id:\cat{1}\to\cat{1}$ in $\cat{Cat}$, which is $\cat{1}$. In particular, this colimit is not isomorphic to $\mu(D)$ .
 
 Therefore, for the free algebra $(\cat{Diag(C)},\mu)$, the algebra structure map $\mu$ is not in the form of \Cref{cocompalg}. (See the end of \Cref{secpseudoalg} for why $(\cat{Diag(C)},\mu)$ is indeed a pseudoalgebra.)
\end{proof}

A structural reason for why not all $\cat{Diag}$-algebras arise this way will be given in \Cref{pbfunctor}.
Conjecturally, the generic $\cat{Diag}$-algebras may be given by taking \emph{oplax} colimits, instead of strict (in 2-categories rather than categories).

\section{Image presheaves}\label{imagepsh}

In this section we define the notion of \emph{image presheaf} of a diagram, which may be interpreted as its ``free'' or ``prototype colimit''. We also extend and generalize the theory of cofinal functors (which we call \emph{confinal}, see \Cref{confinal}) and of absolute colimits, giving conditions for when certain diagrams have isomorphic colimits even after applying a functor to them (\Cref{confinallemma}). This incorporates absolute colimits as a special case (see \Cref{absolutelemma} and the subsequent discussion).

\subsection{Diagrams and presheaves}

Given a diagram $D:\cat{J}\to\cat{C}$, we obtain a presheaf $\im{D}$ on $\cat{C}$ canonically, as follows. 

\begin{deph}
 The \emph{image presheaf} of the diagram $D$, which we denote by $\im D$, 
 is the colimit of the following composite functor,
$$
\begin{tikzcd}
 \cat{J} \ar{r}{D} & \cat{C} \ar{r}{Y} & {[\cat{C}^\op, \cat{Set}]}
\end{tikzcd}
$$
where $Y$ denotes the Yoneda embedding.
\end{deph}

We can view the image as a ``free colimit'', the presheaf obtained as the colimit of representables indexed by the diagram $D$.
As usual, by the universal property of colimits this assignment is functorial.

Equivalently, $\im D$ is the (pointwise) left Kan extension
$$
\im{D} \coloneqq \Lan{1}{D^\op},
$$
as in the following diagram,
$$
\begin{tikzcd}[row sep=small]
 \cat{J}^\op 
  \ar[""{name=ONE,below}]{dr}{1} \ar{dd}[swap]{D^\op} \\
 & \cat{Set} \\
 \cat{C}^\op 
  \ar{ur}[swap]{\im{D}} \ar[Rightarrow,from=ONE, shorten <= 0.5em, shorten >= 1em, "\lambda_D", near start]
\end{tikzcd}
$$
where $1:\cat{J}^\op\to\cat{Set}$ is the constant presheaf at the singleton set $1$, and $\lambda_D$ denotes the universal $2$-cell.
This way one could generalize the definition to the case of weighted diagrams, which is however beyond the scope of the present paper.

Concretely, given an object $C$ of $\cat{C}$, the set $(\im D) (C)$ is the set 
$$
\colim_{J\in \cat{J}} \big( \Hom_\cat{C}(C,DJ) \big).
$$
Its elements are the equivalence classes of arrows of $\cat{C}$ of the form $C\to DJ$, for some object $J$ of $\cat{J}$, where we identify any two arrows $f:C\to DJ$ and $f':C\to DJ'$ whenever there exists a morphism $g:J\to J'$ of $\cat{J}$ such that $f'=Dg\circ f$, as in the following diagram.
$$
\begin{tikzcd}[row sep=tiny]
 J 
  \ar{dd}{g} 
 & & & DJ  
  \ar{dd}{Dg} \\
 & \rightsquigarrow 
 & C 
  \ar{ur}{f} 
  \ar{dr}[swap]{f'} \\
 J' 
 & & & DJ'
\end{tikzcd}
$$
Functoriality of $\im$ is given by pasting arrows and commutative diagrams.

In general, two arrows $f:C\to DJ'$ and $f':C\to DJ'$ are identified if there is a zig-zag of arrows of $\cat{J}$ connecting $J$ and $J'$, which we write as $J\leftrightsquigarrow J'$, such that the following diagram ``commutes''.
$$
\begin{tikzcd}[row sep=tiny]
 J
  \ar[leftrightsquigarrow]{dd} 
  & & & DJ  
  \ar[leftrightsquigarrow]{dd} \\
 & \rightsquigarrow & C 
 \ar{ur}{f} \ar{dr}[swap]{f'} \\
 J' & & & DJ'
\end{tikzcd}
$$
By convention, we say that a triangle containing a zig-zag as the one above commutes if and only if each arrow in the zig-zag gives a commutative triangle.
For later use, we denote by $[J,f]$ the equivalence class in $\im D$ represented by $f:C\to DJ$. 

\subsubsection{The category of elements}\label{catel}

Let $P:\cat{C}^\op\to\cat{Set}$ be a presheaf. Recall the discrete fibration given by the category of elements
$$
\groth^o P \to \cat{C} .
$$
Note that we use again the short integral sign, as we had used for the Grothendieck construction -- but here it denotes the category of elements, as we are using the contravariant version. The category $\groth^o P$ is the category where
\begin{itemize}
 \item Objects consist of pairs $(C,x)$, where $C$ is an object of $\cat{C}$ and $x_C$ is an element of the set $PC$.
 \item A morphism $(C,x)\to (C',y)$ is a morphism $g:C\to C'$ of $\cat{C}$ such that the function $Pg:PC'\to PC$ sends $y\in PC'$ to $x\in PC$.
\end{itemize}
If $C$ is small (resp.~locally small), $\groth^o P$ is small too (resp.~locally small).
The functor $\groth^o P \to C$, which is a discrete fibration, maps $(C,x)$ to $C$ and a morphism of $\groth^o P$ to the underlying morphism of $C$.
The category of elements is functorial in the following way. Let $\alpha:P\to Q$ be a morphism of presheaves, i.e.~a natural transformation
$$
\begin{tikzcd}
 \cat{C}^\op \ar[bend left, ""{name=P, below}]{r}{P} \ar[bend right, ""{name=Q, above},swap]{r}{Q} & \cat{Set} .
 \ar[Rightarrow, from=P, to=Q, "\alpha"]
\end{tikzcd}
$$
we can construct a functor $\groth^o \alpha:\groth^o P \to \groth^o Q$ which makes the following diagram commute,
$$
\begin{tikzcd}[column sep=small]
 \groth^o P \ar{rr}{\groth^o\alpha} \ar{dr} && \groth^o Q \ar{dl} \\
 & \cat{C}
\end{tikzcd}
$$
where the morphisms into $\cat{C}$ are the canonical discrete fibrations.
The functor $\groth^o \alpha$ is constructed as follows.
\begin{itemize}
 \item It maps the object $(C,p)$, where $C$ is an object of $\cat{C}$ and $p\in PC$, to the object $(C,\alpha_C(p))$. Note that $\alpha_C(p)\in QC$;
 \item It maps the morphism $(C,Pf(q))\to (C',q)$ induced by the morphism $f:C\to C'$ of $\cat{C}$ to the morphism $(C,Qf(\alpha_{C'}(q)))\to (C', \alpha_{C'}(q))$ again induced by $f$. Note that the following naturality diagram commutes.
 $$
 \begin{tikzcd}
  PC' \ar{r}{\alpha_{C'}} \ar{d}{Pf} & QC' \ar{d}{Qf} \\
  PC \ar{r}{\alpha_{C}} & QC
 \end{tikzcd}
 $$
\end{itemize}

Consider now a diagram $D:\cat{J}\to\cat{C}$, take its image presheaf $\im D$ and form its category of elements $\groth^o\im D$. Let's see what we get explicitly.
\begin{itemize}
 \item An object of $\groth^o\im D$ consists of an object $C$ of $\cat{C}$ together with an equivalence class $[J,f]$ represented by an object $J$ of $\cat{J}$ and a morphism $f:C\to DJ$ of $\cat{C}$.
 \item A morphism $(C,[J,f])\to (C',[J',f'])$ consists of a morphism $g:C\to C'$ of $\cat{C}$ such that $(\im D)(g)([J',f'])=[J,f]$. This means that there is a zig-zag of arrows of $\cat{J}$ connecting $J$ and $J'$, which we write as $J\leftrightsquigarrow J'$, such that the following diagram commutes.
 $$
 \begin{tikzcd}[row sep=tiny]
   J \ar[leftrightsquigarrow]{dd} & & C \ar{dd}{g} \ar{r}{f} & DJ \ar[leftrightsquigarrow]{dd}  \\ 
  & \rightsquigarrow  \\
   J' & & C' \ar{r}{f'} & DJ'
 \end{tikzcd}
 $$
\end{itemize}

\begin{prop}
Given a morphism of diagrams 
$$
\begin{tikzcd}[row sep=small]
 \cat{J} 
  \ar[""{name=F,below}]{dr}{F} \ar{dd}[swap]{R} \\
 & \cat{C} \\
 \cat{K} 
  \ar{ur}[swap]{F'} 
  \ar[Rightarrow,from=F, "\rho", shorten <= 0.5em, shorten >= 1em, near start]
\end{tikzcd}
$$
the 2-cell $\rho$ factors in the following form,
$$
\begin{tikzcd}[row sep=small]
 \cat{J} \ar{dd}[swap]{R}  \ar{r}{\tilde F} \ar[phantom,bend left,""{name=F, below}]{ddr}
 \ar[phantom,bend right,""{name=G, above,pos=0.7}]{ddr} & \groth^o \im F  \ar{dd} \ar{dr}
   \\
 && \cat{C} \\
 \cat{K} \ar{r}[swap]{\tilde F'} & \groth^o \im F'
  \ar{ur} 
  \ar[Rightarrow,from=F, to=G, "\tilde\rho"]
\end{tikzcd}
$$
where the triangle on the right commutes, with the vertical arrow given by $\groth^o\im(R,\rho)$ (recall that both $\groth^o$ and $\im$ are functorial).
\end{prop}

Before the proof, let's see what the functor $\groth^o\im(R,\rho):\groth^o \im F\to \groth^o \im F'$ looks like. It maps an object 
$$
(C,[J,C\xrightarrow{f} FJ])
$$
of $\groth^o\im F$ to the object 
$$
(C,[RJ,C\xrightarrow{f} FJ \xrightarrow{\rho_J} F'RJ])
$$
of $\groth^o\im F'$. On morphisms, given $g:C\to C'$ in $\cat{C}$ and a zig-zag $J \leftrightsquigarrow J'$ in $\cat{J}$ making the diagram on the left commute, we get the diagram on the right.
$$
 \begin{tikzcd}
   C \ar{d}{g} \ar{r}{f} & FJ \ar[leftrightsquigarrow]{d} \\
   C' \ar{r}{f} & FJ'
 \end{tikzcd}
 \qquad\rightsquigarrow\qquad 
 \begin{tikzcd}
   C \ar{d}{g} \ar{r}{f} & FJ \ar{r}{\rho_J} \ar[leftrightsquigarrow]{d} & F'RJ \ar[leftrightsquigarrow]{d}  \\
   C' \ar{r}{f} & FJ' \ar{r}{\rho_J'} & F'RJ'
 \end{tikzcd}
 $$
 The right-most square commutes by naturality of $\rho$ applied to the zig-zag.

\begin{proof}
 Let $J$ be an object of $\cat{J}$. Let's give the component of $\tilde\rho$ at $J$ explicitly. Note first that 
 $$
 \groth^o\im(R,\rho) (\tilde F (J)) \;=\; (J, [FJ \xrightarrow{\id} FJ \xrightarrow{\rho_J} F'RJ])
 $$
 and that 
 $$
 \tilde F'(R(J)) \;=\; (RJ, [F'RJ \xrightarrow{\id} F'RJ]). 
 $$
 The morphism $\tilde\rho_J: \groth^o\im(R,\rho) (\tilde F (J)) \to \tilde F'(R(J))$ is then given by the following diagram,
 $$
 \begin{tikzcd}
  FJ \ar{d}{\rho} \ar{r}{\id} & FJ \ar{r}{\rho} & F'RJ \ar[leftrightsquigarrow]{d} \\
  F'RJ \ar{rr}{\id} && F'RJ
 \end{tikzcd}
 $$
 with the zig-zag given  by the identity. By construction, whiskering $\tilde\rho$ with the forgetful functor to $\cat{C}$ we get back $\rho$.
\end{proof}

\subsection{Connecting confinal functors and absolute colimits}\label{confinal}

Here we extend a bit the theory of confinal functors,%
\footnote{These are known in the literature also as ``cofinal'', ``coinitial'' and ``final'', terms which may cause confusion. The ``co'' in ``cofinal'' (e.g.~in ``cofinal subnet'') does not denote duality, but rather, follows the Latin particle ``cum'' which means ``with, together''. As such, we feel that ``confinal'' is both closer to the original etymology, and less prone to cause confusion.
The term ``confinal'' (or in German, ``konfinal'') was introduced by Hausdorff for the case of ordered sets \cite[Section~IV.4, page~86]{hausdorff}, and it has been in use at least until \cite[Definition 2.12]{konfinal}. 
} and unify it with the theory of absolute colimits.
A reference for the standard theory is for example given in \cite[Section~2.11]{borceux} (note that there the term ``final functor'' is used instead, for limit-invariant functors, rather than colimit-invariant).

\begin{deph}
 A functor $F:\cat{C}\to\cat{D}$ is called \emph{confinal} if for every object $D$ of $\cat{D}$, the comma category $D/F$ is non-empty and connected.
\end{deph}

The importance of confinal functors is due to the following well-known statement, which is actually an equivalent characterization of confinality.

\begin{prop}\label{whichiso}
 If $F:\cat{C}\to\cat{D}$ is confinal, for every functor $G:\cat{D}\to\cat{E}$ admitting a colimit, the functor $G\circ F:\cat{C}\to\cat{E}$ admits a colimit too, and the map between colimits
 $$
 \colim_{C\in \cat{C}} G(F(C)) \to \colim_{D\in\cat{D}} G(D) 
 $$
 induced by the following morphism of (possibly large) diagrams
 $$
 \begin{tikzcd}[column sep=small]
  \cat{C} \ar{rr}{F} \ar{dr}[swap]{G\circ F} && \cat{D} \ar{dl}{G} \\
  & \cat{E}
 \end{tikzcd}
 $$
 is an isomorphism.
\end{prop}

For a proof, see for example the proof of the very similar statement \cite[Proposition~2.11.2]{borceux} (again, note the different conventions there).

\subsubsection{Refining the comprehension factorization}

We would like now to prove the following statement.

\begin{prop}\label{samecolimit}
 Let $\cat{C}$ be (small-)cocomplete. The (large) colimit of the fibration 
 $$
 \pi:\groth^o\im F \to \cat{C}
 $$
 exists, and it coincides with the (small) colimit of the diagram $F:\cat{J}\to\cat{C}$.
\end{prop}

This is almost an instance of the following known result, sometimes called the ``comprehension factorization schema''.

\begin{thm}[\cite{comprehension}]
 There is an orthogonal $(E,M)$-factorization system on $\cat{Cat}$, where $E$ are the confinal functors and $M$ are the discrete fibrations.
\end{thm}

However, in our case we are not requiring $\cat{C}$ to be small, only locally small. Because of this, and because we need the construction explicitly, we give a dedicated proof.
We construct a functor $\tilde{F}:\cat{J}\to\groth^o\im F$ as follows.
\begin{itemize}
 \item For each object $J$ of $\cat{J}$, define 
 $$
 \tilde{F}J \;\coloneqq\; (FJ,[J,\id_{FJ}]),
 $$
 i.e.~assign to $J$ the equivalence class represented by the identity arrow $FJ\to FJ$ of $\cat{C}$. 
 \item For each morphism $f:J\to J'$, take the map $Ff:FJ\to FJ'$. Notice that we have the following commutative diagram,
 $$
 \begin{tikzcd}
  FJ \ar{d}{Ff} \ar{r}{\id} & FJ \ar{d}{Ff} \\
  FJ' \ar{r}{\id} & FJ'
 \end{tikzcd}
 $$
 so that we have a well-defined morphism of $\groth^o\im F$ (the zig-zag is simply given by the morphism $f$).
\end{itemize}

\begin{prop}\label{propconfinal}
 The functor $\tilde{F}:\cat{J}\to\groth^o\im F$ is confinal.
\end{prop}

This suffices to deduce \Cref{samecolimit}, since the following diagram commutes.
$$
\begin{tikzcd}
 & \groth^o \im F \ar{d}{\pi} \\
 \cat{J} \ar{ur}{\tilde{F}} \ar{r}[swap]{F} & \cat{C}
\end{tikzcd}
$$

\begin{proof}[Proof of \Cref{propconfinal}]
 We need to show that for every object $(C,[J,f])$ of $\groth^o\im F$, the comma category $(C,[J,f])/\tilde{F}$ is non-empty and connected. This is guaranteed by the way the category $\groth^o\im F$ is constructed, as follows.
 
 First, since $\tilde{F}J$ is the equivalence class represented by the identity $FJ\to FJ$, we can consider $f$ as an arrow  $(C,[J,f])\to \tilde{F}J$ of $\groth^o\im F$ and, hence, an object of $(C,[J,f])/\tilde{F}$, 
 as we have the trivially commuting diagram 
 $$
 \begin{tikzcd}
 C \ar{d}{f} \ar{r}{f} & FJ \ar[leftrightsquigarrow]{d} \\
 FJ \ar{r}{\id} & FJ 
 \end{tikzcd}
 $$
 with the zig-zag given by the identity. Now, given any other arrow $f':(C,[J,f])\to \tilde{F}J'$ in $\groth^o\im F$, we have $[J',f']=\im F(f')([J',1_{DJ'}])=[J,f]$. So there is a zig-zag $J\leftrightsquigarrow J'$ in $\cat{J}$, which actually links $f'$ with $f$ in the comma category $(C,[J,f])/\tilde{F}$, as required. 
\end{proof}

\subsubsection{Mutually confinal diagrams}

We will make use of the following well-known fact:
\begin{lemma}\label{confinalsubcat}
 Consider the functors
 $$
 \begin{tikzcd}
  \cat{A} \ar{r}{F} & \cat{B} \ar{r}{G} & \cat{C}
 \end{tikzcd}
 $$
 where $\cat{A}$, $\cat{B}$ and $\cat{C}$ are locally small categories. If $G\circ F$ is confinal and $G$ is fully faithful, then $F$ and $G$ separately are confinal too.
\end{lemma}

\begin{prop}\label{confinallemma}
 Let $\cat{C}$ be a locally small category. 
 Let $D:\cat{J}\to\cat{C}$ and $E:\cat{K}\to\cat{C}$ be small diagrams. The following conditions are equivalent.
 \begin{enumerate}
  \item\label{condsameimage} $\im D$ and $\im E$ are naturally isomorphic;
  \item\label{condsamecolimit} for every locally small category $\cat{D}$ and every functor $F:\cat{C}\to\cat{D}$, the composite diagram $F\circ D:\cat{J}\to\cat{D}$ admits a colimit if and only if $F\circ E:\cat{K}\to\cat{D}$ does, and in that case the two colimits are isomorphic;
  \item\label{condzigzag} $D$ and $E$ are connected by a zigzag in $\cat{Cat}/\cat{C}$ such that all the arrows of the underlying zigzag $\cat{J}\leftrightsquigarrow\cat{K}$ in $\cat{Cat}$ are confinal functors.
 \end{enumerate}
\end{prop}

\begin{deph}
 If the diagrams $D:\cat{J}\to\cat{C}$ and $E:\cat{K}\to\cat{C}$ satisfy any (and, hence, all) of the conditions above, we call them \emph{mutually confinal}.
\end{deph}

One should view the property of being mutually confinal as the \emph{absolute} coincidence of their colimits: existence granted, their colimits remain the same even after applying any other functor.

\begin{proof}[Proof of \Cref{confinallemma}]
 The statement $\ref{condzigzag}\Rightarrow\ref{condsamecolimit}$ is part of the standard theory of confinal functors (see the references).
 The statement $\ref{condsamecolimit}\Rightarrow\ref{condsameimage}$ follows from choosing for $F:\cat{C}\to\cat{D}$ the Yoneda embedding $\eta:\cat{C}\to\cat{PC}$.
 
 The real work is to prove $\ref{condsameimage}\Rightarrow\ref{condzigzag}$.
 To this end, suppose that $\alpha$ is an isomorphism $\im D \cong \im E$. We have an isomorphism between the corresponding categories of elements,
 $$
 \begin{tikzcd}[column sep=small]
  \groth^o \im D \ar{rr}{\groth^o\alpha}[swap]{\cong} \ar{dr} && \groth^o \im E \ar{dl} \\
  & \cat{C}
 \end{tikzcd}
 $$
 together with functors $\tilde{D}:\cat{J}\to \groth^o \im D$ and $\tilde{E}:\cat{K}\to \groth^o \im E$ which are confinal by \Cref{propconfinal}. We have the following diagram of confinal functors.
 $$
 \begin{tikzcd}[column sep=small]
  \cat{J} \ar{d}{\tilde D} && \cat{K} \ar{d}{\tilde E} \\
  \groth^o \im D \ar{rr}{\groth^o\alpha}[swap]{\cong} && \groth^o \im E
 \end{tikzcd}
 $$
 Denote now by $\cat{S}$ the full subcategory of $\groth^o \im E$ given by the joint full image of the two functors $\groth^o\alpha\circ\tilde{D}$ and $\tilde E$. The situation is depicted in the following commutative diagram.
 $$
 \begin{tikzcd}[column sep=small]
  \cat{J} \ar{d}{\tilde D} && \cat{K} \ar{d} \ar{ddr}{\tilde E} \\
  \groth^o \im D \ar{rr} \ar{drrr}[swap]{\groth^o\alpha} && S \ar[hookrightarrow]{dr} \\
  &&& \groth^o \im E
 \end{tikzcd}
 $$ 
 By construction, $S$ is small, since its cardinality is bounded by the one of the disjoint union of the sets of objects of $\cat{J}$ and $\cat{K}$, which are small. Moreover, the resulting functors $\cat{J}\to\cat{S}$ and $\cat{K}\to\cat{S}$ are confinal by \Cref{confinalsubcat}. The resulting diagram
 $$
 \begin{tikzcd}
  \cat{J} \ar{dr}[swap]{D} \ar{r} & \cat{S} \ar{d} & \cat{K} \ar{l} \ar{dl}{E} \\
  & \cat{C}
 \end{tikzcd}
 $$
 gives the desired zigzag (of length 2).
\end{proof}

\subsubsection{Absolute colimits}

An \emph{absolute colimit} is a colimit which is preserved by every functor \cite{absolutecolimits}.
We can redefine the concept of absolute colimits in terms of mutually confinal functor as follows. As we will see, this is equivalent to the usual definition.

\begin{deph}
 Let $\cat{C}$ be a locally small category, and let $D:\cat{J}\to\cat{C}$ be a small diagram. An \emph{absolute colimit} of $D$ is an object $X$ of $\cat{C}$ such that the diagrams $D:\cat{J}\to\cat{C}$ and $X:\cat{1}\to\cat{C}$ are mutually confinal.
\end{deph}

The image presheaf of a one-object diagram is the one given by the Yoneda embedding $y:\cat{C}\to [\cat{C}^\op,\cat{Set}]$, as the following proposition show. 

\begin{prop}\label{unmm}
 For each locally small category $\cat{C}$, the following diagram commutes up to natural isomorphism.
 $$
 \begin{tikzcd}[column sep=small]
  & \cat{C} \ar{dl}[swap]{\eta} \ar{dr}{y} \\
  \cat{Diag(C)} \ar{rr}{\im} && {[\cat{C}^\op,\cat{Set}]}
 \end{tikzcd}
 $$
\end{prop}

\begin{proof}
 Using the definition of image in terms of Kan extensions, and recalling that $\eta(X)$ is the diagram $X:\cat{1}\to\cat{C}$ that picks out the object $X$, we have that $\im(\eta(X))$ is given by the following Kan extension,
 $$
  \begin{tikzcd}[row sep=small]
 \cat{1}
  \ar[""{name=ONE,below}]{dr}{1} \ar{dd}[swap]{X} \\
 & \cat{Set} \\
 \cat{C}^\op 
  \ar{ur}[swap]{\Lan{1}{X}} \ar[Rightarrow,from=ONE, shorten <= 0.5em, shorten >= 1em, "\lambda", swap]
 \end{tikzcd}
 $$
 which is isomorphic to $\Hom_\cat{C}(-,X)$, i.e.~the image of $X$ under the Yoneda embedding.
 The isomorphism is moreover natural in $X$, by the universal property of (free) colimits.
\end{proof}

Therefore, equivalently, an object $X$ is an absolute colimit of the diagram $D:\cat{J}\to\cat{C}$ if and only if $\im D$ is naturally isomorphic to the representable presheaf $\Hom_\cat{C}(-,X)$. 

If we instance \Cref{confinallemma} for this case, we get the following statement.

\begin{prop}\label{absolutelemma}
 Let $\cat{C}$ be a locally small category. 
 Let $D:\cat{J}\to\cat{C}$ and be a small diagram, and let $X$ be an object of $\cat{C}$. The following conditions are equivalent.
 \begin{enumerate}
  \item\label{condsameimage2} $X$ is an absolute colimit of $D$ (i.e.~$\im D \cong \Hom_\cat{C}(-,X)$ naturally);
  \item\label{condsamecolimit2} for every locally small category $\cat{D}$ and every functor $F:\cat{C}\to\cat{D}$, the object $F(X)$ is the colimit in $\cat{D}$ of the composite diagram $F\circ D:\cat{J}\to\cat{D}$;
  \item\label{condzigzag2} $D:\cat{J}\to\cat{C}$ and $X:\cat{1}\to\cat{C}$ are connected by a zigzag in $\cat{Cat}/\cat{C}$ such that all the arrows of the underlying zigzag $\cat{J}\leftrightsquigarrow\cat{1}$ in $\cat{Cat}$ are confinal functors.
 \end{enumerate}
\end{prop}

Note that condition \ref{condsamecolimit2} implies, in particular, that indeed $X$ is a colimit of $D$ (take $F$ to be the identity). Denote the colimit cone by $h:D\Rightarrow X$. 

We can now rewrite condition \ref{condzigzag2} in a more elementary way. Recall that in the proof of \Cref{confinallemma} we had obtained condition \ref{condzigzag} from \ref{condsameimage} by forming the category of elements of the (common) image presheaf, and taking the joint image of the confinal functors from $\cat{J}$ and from $\cat{1}$ to this category of elements.
The category of elements of the representable presheaf $\Hom_\cat{C}(-,X)$ is isomorphic to the slice category $\cat{C}/X$. We therefore have to take the joint image in $\cat{C}/X$ of the two functors at the top of this diagram,
$$
\begin{tikzcd}
 \cat{J} \ar{dr}[swap]{D} \ar{r} & \cat{C}/X \ar{d}{\pi} & \cat{1} \ar{dl}{X} \ar{l} \\
 & \cat{C}
\end{tikzcd}
$$
where the functor $\cat{J}\to\cat{C}/X$ maps an object $J\in\cat{J}$ to the arrow of the colimit cone $h_J:DJ\to X$.
Just as in the proof of \Cref{confinallemma}, denote this joint full image by $\cat{S}$. Now, the resulting functor $\cat{1}\to \cat{S}$ is trivially confinal, since it maps the unique object of $\cat{1}$ to $\id_X\in\cat{C}/X$. 
More interestingly, the proof of \Cref{confinallemma} says that also the resulting functor $\cat{J}\to \cat{S}$ is confinal. The condition is nontrivial for the only object of $\cat{S}$ that does not come from $\cat{J}$, which is the one coming from $\cat{1}$, namely $\id_X\in\cat{C}/X$. For this object, the confinality condition of the functor $\cat{J}\to \cat{S}$ says the following:

\begin{enumerate}
 \setcounter{enumi}{3}
 \item There exist an object $J$ of $\cat{J}$ and an arrow $f:X\to DJ$ of $\cat{C}$ such that the following diagram commutes:
 $$
 \begin{tikzcd}
  X \ar{dr}[swap]{\id} \ar{r}{f} & DJ \ar{d}{h_J} \\
  & X
 \end{tikzcd}
 $$
 and such that moreover, for each object $J'$ of $\cat{J}$ and arrow $f':X\to DJ'$ making a similar diagram commute, there exists a zigzag $J \leftrightsquigarrow J'$ in $\cat{J}$ making the following diagram in $\cat{C}$ commute.
 $$
 \begin{tikzcd}
  & & & DJ' \ar{dddl}{h_J'} \ar[leftrightsquigarrow]{dl} \\
  X \ar{ddrr}[swap]{\id} \ar{urrr}{f'} \ar{rr}[swap]{f} & & DJ \ar{dd}[swap]{h_J} \\ \\
  & &  X
 \end{tikzcd}
 $$
\end{enumerate}

Intuitively, we can interpret this condition as ``the colimit cone eventually has a section, which is in some sense unique''.
This is similar to very well-known statements in the literature, see for example Theorems 2.1 and 4.1 in \cite{absolutecolimits}.
Therefore we can view our theory of mutually confinal functor as a joint generalization both of confinal functors and of absolute colimits.

\section{The monad of small presheaves}\label{smallpsh}

In this section we study small presheaves, and show that they also form a pseudomonad. Moreover, the image map of the previous section gives a morphism of pseudomonads (also explicitly defined in \Cref{pseudomonads}). 
Again, cocomplete categories are pseudoalgebras of this monad, but this time, every pseudoalgebra is of this form. Indeed, considering the long history of (co)completion theory of categories (see \cite{isbell, completion} for early contributions), one should view the monad of small presheaves as the ``free small-cocompletion monad''. The fact that $\cat{Diag}$ admits cocomplete categories as algebras is then to be thought of as an instance of the ``restriction of scalars'' construction, where algebras of a monad can be pulled back along a morphism of monads, see \Cref{restrictionofscalars}.

It is known that small presheaves form a pseudomonad \cite{smallfunctors}. However, we did not find an explicit construction in the literature, so we give one in the present section. Compared to the pseudomonad of \Cref{monadofdiagrams}, this one is weaker: the underlying pseudofunctor is not a strict 2-functor. A short review of the relevant basic definitions can be found in \Cref{pseudomonads}. 
The fact that $\im$ defines a morphism of pseudomonads (\Cref{immm}) seems to be new. 

\subsection{Small presheaves}

\begin{deph}
 A presheaf is called \emph{small} if it is (naturally isomorphic to) the image presheaf of a small diagram.
 
 Denote by $\cat{PC}$ the full subcategory of $[\cat{C}^\op,\cat{Set}]$ whose objects are small presheaves. 
\end{deph}

The image presheaf of a (small) diagram is by definition a small presheaf, so that the functor $\im:\cat{Diag(C)}\to[\cat{C}^\op,\cat{Set}]$ actually lands in $\cat{PC}$. 
We denote the resulting functor $\cat{Diag(C)}\to\cat{PC}$ again by $\im$. This will not cause confusion, since from now on we will only consider small presheaves.

Despite the slightly new terminology, this is a known concept, see for example \cite{smallfunctors}. 
We recall the following facts.
\begin{itemize}
 \item A presheaf is small if and only if it can be written as a \emph{small} colimit of representables \cite[Section~2]{smallfunctors}. Therefore we can think of small presheaves as of forming the \emph{free small cocompletion} of a category.
 \item The category $\cat{PC}$ of small presheaves on a locally small category $\cat{C}$ is itself locally small. This allows us to avoid several size issues when talking about the free cocompletion.
\end{itemize}

Notice also the following fact.
\begin{remark}\label{ffconfinal} 
 Let $\cat{C}$ be a locally small category, and let $P:\cat{C}^\op\to\cat{Set}$ be a small presheaf. Then we know (\Cref{propconfinal}) that there exists a small category $\cat{S}$ and confinal functor $\tilde F:\cat{S}\to \groth^o P$. By \Cref{confinalsubcat}, we can assume that $\tilde F$ is fully faithful, or equivalently that it is the inclusion of a full subcategory.
\end{remark}

For later use in this section, we recall the following known statement, sometimes called the \emph{co-Yoneda lemma} (see \cite[Section~3.10]{kelly}, as well as \cite[Section~2.2]{coend}).\footnote{Often a stronger statement is called ``co-Yoneda lemma'', see \Cref{generalcoyoneda}.} 

\begin{prop}\label{coyoneda}
 Let $\cat{C}$ be a category, and let $H:\cat{C}\to\cat{Set}$ be a functor. There is an isomorphism
 $$
 H(C) \;\cong\; \int^{C'\in\cat{C}} \Hom_\cat{C}(C',C) \times H(C') ,
 $$
 for each object $C$ of $\cat{C}$ and natural in $C$, given by mapping each element $x\in H(C)$ to the equivalence class in the coend above the ordered pair $(\id_C, x)\in \Hom_\cat{C}(C,C) \times H(C)$.  
\end{prop}

\subsection{The pseudofunctor}

Given locally small categories $\cat{C}$ and $\cat{D}$ and a functor $F:\cat{C}\to\cat{D}$, we would like to find an assignment $\cat{PC}\to\cat{PD}$, which maps small presheaves to small presheaves.

\begin{deph}
 Let $F:\cat{C}\to\cat{D}$ be a functor between locally small categories, and let $P$ be a small presheaf on $\cat{C}$. The \emph{pushforward of $P$ along $F$} is the presheaf on $\cat{D}$ given by the following left Kan extension.
 $$
 \begin{tikzcd}[row sep=small]
 \cat{C}^\op 
  \ar[""{name=ONE,below}]{dr}{P} \ar{dd}[swap]{F^\op} \\
 & \cat{Set} \\
 \cat{D}^\op 
  \ar{ur}[swap]{\Lan{P}{F^\op}} \ar[Rightarrow,from=ONE, shorten <= 0.5em, shorten >= 1em, "\lambda", swap]
\end{tikzcd}
 $$
 We denote the resulting presheaf by $F_\sharp P$.
\end{deph}

Equivalently, $F_\sharp P$ is given by the free colimit of $F$, weighted by $P$. By the universal property of (weighted) colimits, it is therefore functorial in $F$.
Note that this definition specifies $F_\sharp P$ only up to isomorphism. As usual, the choice of a particular object within its isomorphism class is de facto irrelevant.

Recall the following fact, which says that Kan extension diagrams can be pasted vertically. While the statement is folklore and a consequence of the simple fact that universal arrows \cite{maclane} compose in an obvious sense, we provide a proof because the explicit isomorphism given in the proof will be of use later. 

\begin{prop}\label{verticalpaste}
 Let $\cat{A}$, $\cat{B}$, $\cat{C}$ and $\cat{D}$ be categories, and let $F:\cat{A}\to\cat{B}$, $G:\cat{B}\to\cat{C}$, $H:\cat{A}\to\cat{D}$ be functors. The left Kan extensions $\Lan{(\Lan{H}{F})}{G}$ and $\Lan{H}{G\circ F}$ are naturally isomorphic. 
 \begin{equation}\label{vpastediag}
 \begin{tikzcd}[column sep=huge, row sep=large]
 \cat{A} 
  \ar[""{name=ONE,below}]{dr}{H} \ar{d}[swap]{F} \\
 \cat{B} \ar[Rightarrow,from=ONE, shorten <= 0.5em, shorten >= 1em, "\lambda_F", swap] \ar[""{name=TWO, below}]{r}[swap]{\Lan{H}{F}} \ar{d}[swap]{G} & \cat{D} \\
 \cat{C}
  \ar{ur}[swap]{\Lan{(\Lan{H}{F})}{G}} 
  \ar[Rightarrow,from=TWO, shorten <= 0.8em, shorten >= 1.6em, "\lambda_G", swap]
 \end{tikzcd}
 \quad\cong\qquad
 \begin{tikzcd}[column sep=huge, row sep=large]
 \cat{A} 
  \ar[""{name=ONE,below}]{dr}{H} \ar{d}[swap]{F} \\
 \cat{B}  \ar{d}[swap]{G} & \cat{D} \\
 \cat{C}
  \ar{ur}[swap]{\Lan{H}{G\circ F}} 
  \ar[Rightarrow,from=ONE, shorten <= 1em, shorten >= 2em, "\lambda_{G\circ F}", near start]
 \end{tikzcd}
 \end{equation}
\end{prop}

\begin{proof}
 By the universal property of $\Lan{H}{F}$, the natural transformation $\lambda_{G\circ F}$ on the right of \eqref{vpastediag} factors uniquely through $\lambda_F$, i.e.~there exists a unique 2-cell $\nu:\Lan{H}{F}\Rightarrow \Lan{H}{G\circ F}\circ G$ such that the following 2-cells are equal.
 $$
 \begin{tikzcd}[column sep=huge, row sep=large]
 \cat{A} 
  \ar[""{name=ONE,below}]{dr}{H} \ar{d}[swap]{F} \\
 \cat{B} \ar[Rightarrow,from=ONE, shorten <= 0.5em, shorten >= 1em, "\lambda_F", swap] \ar[""{name=TWO, below}]{r}[swap]{\Lan{H}{F}} \ar{d}[swap]{G} & \cat{D} \\
 \cat{C}
  \ar{ur}[swap]{\Lan{H}{G\circ F}} 
  \ar[Rightarrow,from=TWO, shorten <= 0.8em, shorten >= 1.6em, "\nu", swap]
 \end{tikzcd}
 \quad=\qquad
 \begin{tikzcd}[column sep=huge, row sep=large]
 \cat{A} 
  \ar[""{name=ONE,below}]{dr}{H} \ar{d}[swap]{F} \\
 \cat{B}  \ar{d}[swap]{G} & \cat{D} \\
 \cat{C}
  \ar{ur}[swap]{\Lan{H}{G\circ F}} 
  \ar[Rightarrow,from=ONE, shorten <= 1em, shorten >= 2em, "\lambda_{G\circ F}", near start]
 \end{tikzcd}
 $$
 Moreover, by the universal property of $\Lan{(\Lan{H}{F})}{G}$, the natural transformation $\nu$ factors uniquely through $\lambda_G$, meaning that there exists a unique natural transformation $\kappa:\Lan{(\Lan{H}{F})}{G}\Rightarrow \Lan{H}{G\circ F}$ such that the following 2-cells are equal,
 $$
 \begin{tikzcd}[column sep=large]
 \cat{B} 
  \ar[""{name=ONE,below}]{dr}{\Lan{H}{F}} \ar{dd}[swap]{G} \\
 & \cat{D} \\
 \cat{C} 
  \ar[bend right, ""{name=SECOND, above}]{ur}[swap]{\Lan{H}{G\circ F}} \ar[bend left, ""{name=FIRST, below}]{ur}
  \ar[Rightarrow,from=ONE, shorten <= 0.5em, shorten >= 2.5em, "\lambda_G", swap, near start]
  \ar[Rightarrow, from=FIRST, to=SECOND, "\kappa"]
 \end{tikzcd}
 \quad=\qquad
 \begin{tikzcd}[column sep=large]
 \cat{B} 
  \ar[""{name=ONE,below}]{dr}{\Lan{H}{F}} \ar{dd}[swap]{G} \\
 & \cat{D} \\
 \cat{C} 
  \ar[bend right]{ur}[swap]{\Lan{H}{G\circ F}} \ar[Rightarrow,from=ONE, shorten <= 0.8em, shorten >= 1.6em, "\nu", near start]
\end{tikzcd}
 $$
 where the unlabeled arrow (for reasons of space) denotes $\Lan{(\Lan{H}{F})}{G}$. We now show that $\kappa$ is an isomorphism, by providing an inverse. By the universal property of $\Lan{H}{G\circ F}$, the composite natural transformation on the left of \eqref{vpastediag} factors uniquely through $\lambda_{G\circ F}$, meaning that there exists a unique natural transformation $\delta:\Lan{H}{G\circ F} \Rightarrow \Lan{(\Lan{H}{F})}{G}$ such that the following 2-cells are equal. 
 \begin{equation}\label{defdelta}
 \begin{tikzcd}[column sep=huge, row sep=large]
 \cat{A} 
  \ar[""{name=ONE,below}]{dr}{H} \ar{d}[swap]{F} \\
 \cat{B}  \ar{d}[swap]{G} & \cat{D} \\
 \cat{C}
  \ar[bend right, ""{name=SECOND, above}]{ur}[swap]{\Lan{(\Lan{H}{F})}{G}} \ar[bend left, ""{name=FIRST, below}]{ur}
  \ar[Rightarrow,from=ONE, shorten <= 1em, shorten >= 3em, "\lambda_{G\circ F}", near start, swap]
  \ar[Rightarrow, from=FIRST, to=SECOND, "\delta"]
 \end{tikzcd}
 \quad = \qquad
 \begin{tikzcd}[column sep=huge, row sep=large]
 \cat{A} 
  \ar[""{name=ONE,below}]{dr}{H} \ar{d}[swap]{F} \\
 \cat{B} \ar[Rightarrow,from=ONE, shorten <= 0.5em, shorten >= 1em, "\lambda_F", swap] \ar[""{name=TWO, below}]{r}[swap]{\Lan{H}{F}} \ar{d}[swap]{G} & \cat{D} \\
 \cat{C}
  \ar[bend right]{ur}[swap]{\Lan{(\Lan{H}{F})}{G}} 
  \ar[Rightarrow,from=TWO, shorten <= 0.8em, shorten >= 1em, "\lambda_G"]
 \end{tikzcd}
 \end{equation}
 where this time the unlabeled arrow denotes $\Lan{H}{G\circ F}$. By the universal properties of the respective Kan extensions, we then have that $\kappa\circ\delta$ has to be the identity natural transformation at $\Lan{H}{G\circ F}$, and $\delta\circ\kappa$ has to be the identity natural transformation at $\Lan{(\Lan{H}{F})}{G}$.
\end{proof}

\begin{cor}\label{smalltosmall}
 Pushforwards of small presheaves exist, are given by pointwise left Kan extensions, and are small.
\end{cor}

\begin{remark}\label{pastewithcoyoneda}
 Since we are dealing with pointwise Kan extensions, we can also express this vertical pasting law in terms of coends, where it is an instance of the co-Yoneda lemma (\Cref{coyoneda}). In particular, let $\cat{A}$ and $\cat{B}$ be locally small categories, let $F:\cat{A}\to\cat{B}$ be a functor, and let $P:\cat{A}^\op\to\cat{Set}$ be a small presheaf.
 Then 
 $$
 F_\sharp P (B) \cong \int^{A\in\cat{A}} P(A) \times \Hom_\cat{B}(B,FA) .
 $$ 
 Let moreover $\cat{C}$ be locally small, and $G:\cat{B}\to\cat{C}$ be a functor. Then
\begin{align*}
G_\sharp F_\sharp P (C) &\cong \int^{A\in\cat{A}}\int^{B\in\cat{B}} P(A) \times \Hom_\cat{B}(B,FA) \times \Hom_\cat{C}(C,GB) \\
&\cong \int^{A\in\cat{A}} P(A) \times \Hom_\cat{C}(C,GFA) \\
&\cong (G\circ F)_\sharp P (C) ,
\end{align*}
where the middle isomorphism, which in the proof \Cref{verticalpaste} was denoted by $\kappa$, is given by the co-Yoneda lemma.  
\end{remark}

Now, given $F:\cat{C}\to\cat{D}$, we have a (chosen) mapping $F_\sharp:\cat{PC}\to\cat{PD}$. For $\cat{P}$ to be a pseudofunctor, we first of all need $F_\sharp$ to be a functor. To this end, let $\alpha:P\to Q$ be a natural transformation between small presheaves on $\cat{C}$. By the universal property of $F_\sharp P$ as a Kan extension, there is a unique 2-cell $F_\sharp P\Rightarrow F_\sharp Q$, which we denote by $F_\sharp \alpha$, which makes the following 2-cells equal.
\begin{equation}\label{defstaralpha}
\begin{tikzcd}[column sep=large]
 \cat{C}^\op 
  \ar[""{name=ONE,below}, bend left]{dr}{P} \ar[""{name=TWOUP, above},""{name=TWODOWN, below, near end}, bend right]{dr}[swap, pos=0.4]{Q} \ar{dd}[swap]{F^\op} \\
 & \cat{Set} \\
 \cat{D}^\op 
  \ar[bend right]{ur}[swap]{F_\sharp Q} 
  \ar[Rightarrow, from=TWODOWN, shorten <= 0.5em, shorten >= 1em, "\lambda_Q"]
  \ar[Rightarrow, from=ONE, to=TWOUP, "\alpha"]
\end{tikzcd}
\qquad=\qquad
\begin{tikzcd}[column sep=large]
 \cat{C}^\op 
  \ar[""{name=ONE,below, near start}, bend left]{dr}{P}\ar{dd}[swap]{F^\op} \\
 & \cat{Set} \\
 \cat{D}^\op 
  \ar[bend right, ""{name=FOUR, below}]{ur}[swap]{F_\sharp Q} \ar[bend left, ""{name=THREE, above}]{ur}[pos=0.9]{F_\sharp P} 
  \ar[Rightarrow, from=ONE, shorten <= 1em, shorten >= 3em, "\lambda_P", swap, near start]
  \ar[Rightarrow, from=THREE, to=FOUR, shorten <=0.5em, shorten >=0.5em, "F_\sharp\alpha"]
\end{tikzcd}
\end{equation}
This makes $F_\sharp$ a functor $\cat{PC}\to\cat{PD}$, where functoriality holds by uniqueness of the cell $F_\sharp \alpha$.
Uniqueness of such cell holds once a choice of $F_\sharp$ has been made.
(One can obtain this 2-cell also using the pointwise characterization of $F_\sharp$ as a coend.)

\subsubsection{Unitor and compositor}\label{unitcompP}

The left Kan extension of $P\in \cat{PC}$ along the identity functor $\id:\cat{C}\to\cat{C}$ is naturally isomorphic to $P$ itself, and this isomorphism is natural in $P$ as well. In other words, $\id_\sharp: \cat{PC}\to\cat{PC}$ is naturally isomorphic to $\id:\cat{PC}\to\cat{PC}$.
Since we are free to choose $\id_\sharp P$ within its isomorphism class, we can in particular pick $\id_\sharp P=P$, so that the unitor of our pseudofunctor is the identity (one speaks of a \emph{normal pseudofunctor}).

With composition, the matters are not so simple. By \Cref{verticalpaste}, or by \Cref{pastewithcoyoneda}, we know that Kan extensions preserve compositions up to a specified natural isomorphism, which we had denoted by $\kappa$. In general we cannot assume that $\kappa$ is the identity, we cannot make that choice consistently across the whole category. However, we can show that $\kappa$ satisfies all the properties of a compositor, and so it makes $\cat{P}$ pseudofunctorial.

As in \Cref{pastewithcoyoneda}, let $\cat{A}$, $\cat{B}$ and $\cat{C}$ be locally small categories, and let $F:\cat{A}\to \cat{B}$ and $G:\cat{B}\to\cat{C}$ be functors. 
Let moreover $P$ be a small presheaf on $\cat{A}$.
The isomorphism $\kappa: G_\sharp(F_\sharp P) \to (G\circ F)_\sharp P$ of $\cat{PC}$ given by the co-Yoneda lemma, as in \Cref{pastewithcoyoneda}, is (strictly) natural in $P$, in $F$, and in $G$, by the universal property of coends.

In order to have pseudofunctoriality it remains to be shown that the compositor $\kappa$ is associative and unital. Unitality is guaranteed by our choice of unitor (identities), we now prove associativity.

\begin{prop}\label{associativityofpf}
 The following ``associativity'' diagram commutes for all locally small categories $\cat{A}$, $\cat{B}$, $\cat{C}$ and $\cat{D}$ and functors $F:\cat{A}\to \cat{B}$, $G:\cat{B}\to\cat{C}$ and $H:\cat{C}\to\cat{D}$. 
 $$
 \begin{tikzcd}
  H_\sharp \circ G_\sharp \circ F_\sharp \nat{r}{\kappa\circ\id} \nat{d}{\id\circ\kappa} & (H\circ G)_\sharp\circ F_\sharp \nat{d}{\kappa} \\
  H_\sharp \circ (G\circ F)_\sharp \nat{r}{\kappa} & (H\circ G\circ F)_\sharp
 \end{tikzcd}
 $$
\end{prop}

One could again invoke the co-Yoneda lemma, but it may be instructive to give a proof by explicitly pasting Kan extensions vertically.
For simplicity, we equivalently prove the statement in terms of the inverse $\kappa^{-1}$.

\begin{proof}
 Let $P$ be a small presheaf on $\cat{A}$. By iterating \eqref{defdelta}, both composite cells
 $$
 
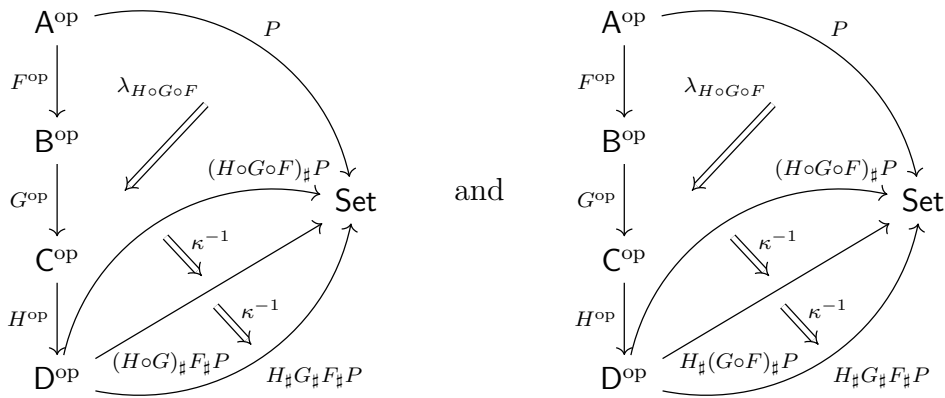
\begin{tikzcd}[row sep=tiny, column sep=large]
  \cat{A}^\op \ar{dd}[swap]{F^\op} \ar[bend left=45, ""{name=P,below}]{dddrr}{P} \\ \phantom{\cat{Set}} \\
  \cat{B}^\op \ar{dd}[swap]{G^\op} \\ 
  && \cat{Set} \\
  \cat{C}^\op \ar{dd}[swap]{H^\op} 
  \ar[Rightarrow, from=P, "\lambda_{H\circ G\circ F}", shorten <= 2em, shorten >= 2em, swap, near start]
  \\  \phantom{\cat{Set}}  \\
  \cat{D}^\op 
  \ar[bend left=45, ""{name=FIRST, below}]{uuurr}[pos=0.85]{(H\circ G\circ F)_\sharp P}
  \ar[""{name=SECONDUP, above}, ""{name=SECONDDOWN, below}]{uuurr}[swap, pos=0.1, xshift=-1ex]{(H \circ G)_\sharp F_\sharp P}
  \ar[bend right=45, ""{name=THIRD, above}]{uuurr}[swap]{H_\sharp G_\sharp F_\sharp P}
  \ar[Rightarrow, from=FIRST, to=SECONDUP, "\kappa^{-1}"]
  \ar[Rightarrow, from=SECONDDOWN, to=THIRD, "\kappa^{-1}"]
 \end{tikzcd}
 \qquad\mbox{and}\qquad
 \begin{tikzcd}[row sep=tiny, column sep=large]
  \cat{A}^\op \ar{dd}[swap]{F^\op} \ar[bend left=45, ""{name=P,below}]{dddrr}{P} \\ \phantom{\cat{Set}} \\
  \cat{B}^\op \ar{dd}[swap]{G^\op} \\ 
  && \cat{Set} \\
  \cat{C}^\op \ar{dd}[swap]{H^\op} 
  \ar[Rightarrow, from=P, "\lambda_{H\circ G\circ F}", shorten <= 2em, shorten >= 2em, swap, near start]
  \\  \phantom{\cat{Set}}  \\
  \cat{D}^\op 
  \ar[bend left=45, ""{name=FIRST, below}]{uuurr}[pos=0.85]{(H\circ G\circ F)_\sharp P}
  \ar[""{name=SECONDUP, above}, ""{name=SECONDDOWN, below}]{uuurr}[swap, pos=0.1, xshift=-1ex]{H_\sharp (G \circ F)_\sharp P}
  \ar[bend right=45, ""{name=THIRD, above}]{uuurr}[swap]{H_\sharp G_\sharp F_\sharp P}
  \ar[Rightarrow, from=FIRST, to=SECONDUP, "\kappa^{-1}"]
  \ar[Rightarrow, from=SECONDDOWN, to=THIRD, "\kappa^{-1}"]
 \end{tikzcd}
 $$
 are equal to the following composition.
 $$
 
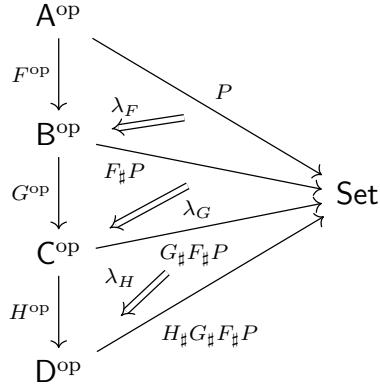
\begin{tikzcd}[row sep=tiny, column sep=large]
  \cat{A}^\op \ar{dd}[swap]{F^\op} \ar[""{name=P,below}]{dddrr}{P} \\ \phantom{\cat{Set}} \\
  \cat{B}^\op \ar{dd}[swap]{G^\op} \ar[""{name=FP, below}]{drr}[near start, swap]{F_\sharp P}
  \ar[Rightarrow, from=P, "\lambda_F", shorten <= 0.5em, shorten >= 0.5em, swap] \\ 
  && \cat{Set} \\
  \cat{C}^\op \ar{dd}[swap]{H^\op} \ar[""{name=GFP, below}]{urr}[near start, swap]{G_\sharp F_\sharp P}
  \ar[Rightarrow, from=FP, "\lambda_G", shorten <= 0.5em, shorten >= 0.5em, near start]   \\  \phantom{\cat{Set}}  \\
  \cat{D}^\op 
  \ar[""{name=SECONDUP, above}]{uuurr}[swap, near start]{H_\sharp G_\sharp F_\sharp P}
  \ar[Rightarrow, from=GFP, "\lambda_H", shorten <= 1.5em, shorten >= 1.5em, swap] 
 \end{tikzcd}
 $$ 
 By the universal property of $(H\circ G\circ F)_\sharp P$ as a Kan extension, then, the two composite 2-cells $(H\circ G\circ F)_\sharp P\Rightarrow H_\sharp G_\sharp F_\sharp P$ are equal.
\end{proof}

This proves that $\cat{P}$ is a pseudofunctor $\cat{CAT}\to\cat{CAT}$.

\subsubsection{Naturality of the image}\label{natimage}

Consider a functor $F:\cat{C}\to\cat{D}$ between locally small categories, and let $D:\cat{I}\to\cat{C}$ be a (small) diagram in $\cat{C}$. One can either form the image presheaf of $D$ and then push it forward along $F$, or one can first form the diagram $F\circ D:\cat{I}\to\cat{D}$, and then take the image presheaf. As we will see shortly, the result is the same, up to coherent isomorphism.

\begin{prop}
 The functors $\im:\cat{Diag(C)}\to\cat{PC}$ form a pseudonatural transformation $\cat{Diag}\Rightarrow\cat{P}$.
\end{prop}

\begin{proof}
 First of all, for each functor $F:\cat{C}\to\cat{D}$ of $\cat{CAT}$ we need a natural isomorphism in the following form.
 \begin{equation}\label{pnatim}
 \begin{tikzcd}
  \cat{Diag(C)} \ar{d}[swap]{F_*} \ar{r}{\im} & \cat{PC} \ar{d}{F_\sharp} \ar[Rightarrow, shorten <= 1em, shorten >= 1em]{dl}{\nu}[swap]{\cong} \\
  \cat{Diag(D)} \ar{r}[swap]{\im} & \cat{PD}
 \end{tikzcd}
 \end{equation}
 Let now $D:\cat{I}\to\cat{C}$ be a small diagram.
 Using the definition of image in terms of Kan extensions, the two routes of \Cref{pnatim} are given by the following Kan extensions, respectively,
 $$
 
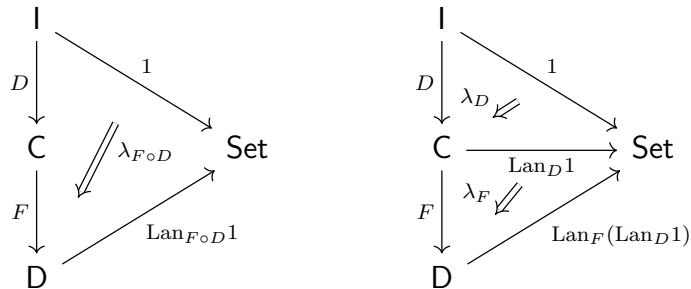
\begin{tikzcd}[column sep=huge, row sep=large]
 \cat{I} 
  \ar[""{name=ONE,below}]{dr}{1} \ar{d}[swap]{D} \\
 \cat{C}  \ar{d}[swap]{F} & \cat{Set} \\
 \cat{D}
  \ar{ur}[swap]{\Lan{1}{F\circ D}} 
  \ar[Rightarrow,from=ONE, shorten <= 1em, shorten >= 2em, "\lambda_{F\circ D}", near start]
 \end{tikzcd}
 \qquad\qquad
 \begin{tikzcd}[column sep=huge, row sep=large]
 \cat{I} 
  \ar[""{name=ONE,below}]{dr}{1} \ar{d}[swap]{D} \\
 \cat{C} \ar[Rightarrow,from=ONE, shorten <= 0.5em, shorten >= 1em, "\lambda_D", swap] \ar[""{name=TWO, below}]{r}[swap]{\Lan{1}{D}} \ar{d}[swap]{F} & \cat{Set} \\
 \cat{D}
  \ar{ur}[swap]{\Lan{(\Lan{1}{D})}{F}} 
  \ar[Rightarrow,from=TWO, shorten <= 0.8em, shorten >= 1.6em, "\lambda_F", swap]
 \end{tikzcd}
 $$
 which we know are naturally isomorphic via the map $\kappa$ of \Cref{verticalpaste} and \Cref{unitcompP}, which is also natural both in $D$ and in $F$. The unit and multiplication conditions correspond to the unitality and associativity condition for $\kappa$.
\end{proof}

\subsection{Unit and multiplication}

\subsubsection{The unit: the Yoneda embedding}

Let $\cat{C}$ be a locally small category. By \Cref{unmm}, the Yoneda embedding $y:\cat{C}\to[\cat{C}^\op,\cat{Set}]$ lands in $\cat{PC}$: representable presheaves are small.
We denote again by $\eta:\cat{C}\to\cat{PC}$ the functor induced by the Yoneda embedding $C\mapsto\Hom_\cat{C}(-,C)$. When this causes confusion because both monads $\cat{Diag}$ and $\cat{P}$ are present, we will denote the two units by $\eta^{\cat{Diag}}$ and $\eta^\cat{P}$.

\begin{prop}
 The unit $\eta:\cat{C}\to\cat{PC}$ is canonically pseudonatural in $\cat{C}$.
\end{prop}

\begin{proof}
 Let $\cat{C}$ and $\cat{D}$ be locally small, and let $F:\cat{C}\to\cat{D}$ be a functor. We have to prove that the following diagram commutes up to coherent isomorphism.
 $$
 \begin{tikzcd}
  \cat{C} \ar{d}{F} \ar{r}{\eta} & \cat{PC} \ar{d}{F_\sharp} \\
  \cat{D} \ar{r}{\eta} & \cat{PD}
 \end{tikzcd}
 $$
 In practice, using the coend description of $F_\sharp$ (via Kan extensions), this amounts to a natural isomorphism of presheaves,
 $$
 \Hom_\cat{D}(-, FC) \;\cong\; \int^{C'\in \cat{C}} \Hom_\cat{C}(C',C) \times \Hom_\cat{D}(-, FC') 
 $$
 which is given by the co-Yoneda lemma (\Cref{coyoneda}), by setting $H(C)=\Hom_{D}(-, FC)$. 
 In particular, the isomorphism is given pointwise, for each object $D$ of $\cat{D}$ by mapping $f:D\to FC$ to the equivalence class of $(\id_C,f)\in \Hom_\cat{C}(C,C)\times\Hom_\cat{D}(D,FC)$.
 It can be checked that, defined this way, the isomorphism respects identities and composition. 
\end{proof}

\subsubsection{The multiplication: free weighted colimits}

Let's now turn to the multiplication of the monad. One could define it as the left-adjoint to the unit, since the monad turns out to be lax idempotent (a.k.a.~Kock-Zöberlein). Here, instead, we define the multiplication directly. 

\begin{deph}
 Let $\cat{C}$ be locally small, and let $\Phi$ be an object of $\cat{PPC}$ (i.e.~a small presheaf on small presheaves). We define $\mu(\Phi)$ as the object of $\cat{PC}$ specified, up to isomorphism, by the following ``free weighted colimit'',
 $$
 \mu(\Phi)(C) \coloneqq \int^{P\in\cat{PC}} \Phi(P) \times P(C)
 $$
 for each object $C$ of $\cat{C}$.
\end{deph}

\begin{remark}\label{coendexists}
By functoriality of colimits, $\mu(\Phi)$ is a presheaf $\cat{C}^\op\to\cat{Set}$. 
Since $PC$ is in general not small, let's show why the coend exists. Since $\Phi$ is small, there exists a small diagram $D:\cat{I}\to\cat{PC}$ such that $\im(D)\cong \Phi$. In other words, 
$$
\mu(\Phi)(C) \;\cong\; \int^{P\in\cat{PC}} \int^{I\in\cat{I}} \Hom_\cat{PC} (P, DI) \times P(C) ,
$$
which by Fubini and by the co-Yoneda lemma (\Cref{coyoneda}) is naturally isomorphic to 
$$
\int^{I\in\cat{I}} DI(C).
$$
This coend exists, since it is a coend in $\cat{Set}$ indexed by a small category, and it gives a small presheaf (since $\cat{PC}$ is cocomplete). 
\end{remark}

Therefore $\mu$ is a functor $\cat{PPC}\to\cat{PC}$ (as usual, defined up to natural isomorphism).

\begin{prop}
 The functor $\mu:\cat{PPC}\to\cat{PC}$ is pseudonatural in the category $\cat{C}$.
\end{prop}

\begin{proof}
 We have to prove that the following diagram commutes up to coherent natural isomorphism.
 $$
 \begin{tikzcd}
  \cat{PPC} \ar{d}{F_{\sharp\sharp}} \ar{r}{\mu} & \cat{PC} \ar{d}{F_\sharp} \\
  \cat{PPD} \ar{r}{\mu} & \cat{PD}
 \end{tikzcd}
 $$
 Given $\Phi\in PPC$, the top right path gives the presheaf
 $$
 \int^{C\in\cat{C}}\int^{P\in \cat{PC}}\Hom_\cat{D}(-, FC) \times \Phi(P) \times P(C) ,
 $$
 which (by Fubini) is isomorphic to 
 \begin{equation}\label{fsharpintp}
 \int^{P\in\cat{PC}} \Phi(P) \times F_\sharp P .
 \end{equation}
 The bottom left path gives the presheaf
 $$ 
 \int^{Q\in \cat{PD}} \int^{P\in\cat{PC}} \Hom_\cat{PD}(Q, F_\sharp P) \times \Phi(P) \times Q ,
 $$
 which by the co-Yoneda lemma (\Cref{coyoneda}) is isomorphic to \eqref{fsharpintp}. 
 One can check that this isomorphism respects identities and composition. 
\end{proof}

\subsection{Unitors, associators, coherence}

The left unitor
$$
\begin{tikzcd}[column sep=0, row sep=small]
	{\cat{PC}} &&[+30pt] & {\cat{PPC}} \\[-10pt]
	&& {\cong} \\
	\\
	&&& {\cat{PC}}
	\arrow["{\eta_\cat{PC}}", from=1-1, to=1-4]
	\arrow["{\mu_\cat{C}}", from=1-4, to=4-4]
	\arrow["{\id}"', from=1-1, to=4-4]
\end{tikzcd}
$$
is given as follows. Starting with a presheaf $P$ in $\cat{PC}$, we can apply the unit to get the following representable presheaf on $\cat{PC}$,
$$
\Hom_{\cat{PC}}(-, P) ,
$$
and then, applying the multiplication, we get the following presheaf,
$$
\int^{Q\in\cat{PC}} \Hom_{\cat{PC}}(Q,P) \times Q(-) \;\cong\; P ,
$$
where the last isomorphism, filling the diagram above, is given by the co-Yoneda lemma (\Cref{coyoneda}), and defines the left unitor $\ell$. 
Naturality in $P$ follows from naturality of the co-Yoneda isomorphism. 
The modification property for $\ell$ against functors $F:\cat{C}\to\cat{D}$ is again an instance of the coherence of colimits (uniqueness of the isomorphism), just as in \Cref{cocompcoherence} and \Cref{associativityofpf}, this time for weighted colimits.

The right unitor 
$$
\begin{tikzcd}[column sep=0, row sep=small]
	{\cat{PC}} &&[+30pt] & {\cat{PPC}} \\[-10pt]
	&& {\cong} \\
	\\
	&&& {\cat{PC}}
	\arrow["{(\eta_\cat{C})_\sharp}", from=1-1, to=1-4]
	\arrow["{\mu_\cat{C}}", from=1-4, to=4-4]
	\arrow["{\id}"', from=1-1, to=4-4]
\end{tikzcd}
$$
is given as follows. Again start with a presheaf $P$ in $\cat{PC}$. This time we apply the map $\eta_\sharp$, to get the following presheaf.
$$
Q\mapsto \int^{C\in\cat{C}} PC \times \Hom_{\cat{PC}}(Q, \eta_\cat{C}(C)) = \int^{C\in\cat{C}} PC \times \Hom_{\cat{PC}}(Q, \Hom_\cat{C}(-,C))
$$
(Note that we cannot apply the Yoneda lemma to simplify the expression on the right.) We now apply the multiplication again, to obtain
\begin{align*}
 X &\mapsto \int^{Q\in \cat{PC}} \int^{C\in\cat{C}} PC \times \Hom_{\cat{PC}}(Q, \Hom_\cat{C}(-,C)) \times Q(X) \\
 &\cong \int^{C\in\cat{C}} P(C) \times \Hom_{C}(X,C) \\
 &\cong P(X) ,
\end{align*}
where both isomorphisms are given again by the co-Yoneda lemma (\Cref{coyoneda}). This gives the right unitor $r$, and the reason why it's a modification is analogous to the one for the left unitor $\ell$.

The associator
$$
\begin{tikzcd}[sep=small]
	{\cat{PPPC}} && {\cat{PPC}} \\
	& {\cong} \\
	{\cat{PPC}} && {\cat{PC}}
	\arrow["{(\mu_C)_\sharp}", from=1-1, to=1-3]
	\arrow["{\mu_C}"', from=3-1, to=3-3]
	\arrow["{\mu_C}", from=1-3, to=3-3]
	\arrow["{\mu_{PC}}"', from=1-1, to=3-1]
\end{tikzcd}
$$
is again an instance of the co-Yoneda lemma (\Cref{coyoneda}). Namely, given $\Phi\in \cat{PPPC}$, the top-right path of the diagram gives the following presheaf
$$
C \mapsto \int^{P\in \cat{PC}} \int^{\Psi\in \cat{PPC}} P(C) \times \Phi(\Psi) \times \Hom_{\cat{PC}}\left( P , \int^{Q\in PC} \Psi(Q) \times Q(-)  \right) 
$$
while the bottom-left path gives the following,
$$
C \mapsto  \int^{\Psi\in \cat{PPC}} \int^{Q\in PC} \Phi(\Psi) \times \Phi(Q) \times Q(C) 
$$
and the two differ by one application of the co-Yoneda lemma (over $P$). This gives the associator $a$, which is a modification for reasons analogous to the above.

Again, the higher coherence conditions hold by the uniqueness of isomorphisms given by the universal property, as in \Cref{cocompcoherence}.

\subsection{Algebras} 

It is well-known that the pseudoalgebras of the pseudomonad $\cat{P}$ are cocomplete categories with a choice of (weighted) colimit, and pseudomorphisms of pseudoalgebras are cocontinuous functors \cite{smallfunctors}. 
Differently from the case of $\cat{Diag}$, this is a complete characterization. 
Let's see in detail how the structure maps look. 

Given a small-cocomplete, locally small category $\cat{C}$, let $e:\cat{PC}\to\cat{C}$ be a choice of weighted colimits, that is,
\begin{equation}\label{wcolim}
e(P) \;\cong\; \int^{X\in\cat{C}} P(X) \otimes X ,
\end{equation}
where $\otimes$ denotes the copower (also known as \emph{tensor}, see for example \cite[Section~3.7]{kelly}). 

The copower satifies a (more general) version of the ``co-Yoneda lemma'' of \Cref{coyoneda}, as follows. See again \cite[Section~3.10]{kelly} for more details.

\begin{prop}\label{generalcoyoneda}
 Let $\cat{C}$ be a category, let $\cat{D}$ be a cocomplete category, and let $H:\cat{C}\to\cat{D}$ be a functor.  There is an isomorphism
 $$
 H(C) \;\cong\; \int^{C'\in\cat{C}} \Hom_\cat{C}(C',C) \otimes H(C') ,
 $$
 for each object $C$ of $\cat{C}$ and natural in $C$, given by selecting the component of $\id_C\in\Hom_\cat{C}(C,C)$ in the copower.
\end{prop}

A similar argument as \Cref{coendexists} shows then that the coend in \eqref{wcolim} exists.
The action of $e$ on morphisms is the one given by the universal property, as usual. 

The unitor and multiplicator of the algebras are given as follows. First of all, the unitor
$$
  \begin{tikzcd}[sep=large]
   \cat{C} \ar[""{name=ID}]{dr}[swap]{\id} \ar{r}{\eta_\cat{C}} & \cat{PC} 
   \ar[Rightarrow, to=ID, "\iota", shorten >= 0.1em, shorten <= 0.3em] \ar{d}{e} \\
   & \cat{C}
  \end{tikzcd}
$$
is the canonical isomorphism given by the general co-Yoneda lemma (\Cref{generalcoyoneda}),
$$
\int^{X\in\cat{C}} \Hom_\cat{C}(X,C) \otimes X \;\cong\; X .
$$
for all $C\in\cat{C}$.

The multiplicator
$$
  \begin{tikzcd}[sep=large]
   \cat{PPC} \ar{d}[swap]{\mu_\cat{C}} \ar{r}{e_\sharp} & \cat{PC} \ar{d}{e}
   \ar[Rightarrow, shorten >= 1.2em, shorten <= 1.2em]{dl}{\gamma}\\
   \cat{PC} \ar{r}[swap]{e} & \cat{C}
  \end{tikzcd}
$$
is also given by the generalized co-Yoneda lemma, as follows,
\begin{align*}
&\;\int^{Y\in\cat{C}} \int^{P\in\cat{PC}} \Phi(P) \times \Hom_\cat{C}\left( Y, \int^{X\in C} P(X)\otimes X \right) \otimes Y  \\
&\cong\; \int^{P\in\cat{PC}} \int^{X\in \cat{C}} \Phi(P) \times P(X) \otimes X 
\end{align*}
for all $\Phi\in\cat{PPC}$. This can be seen as a ``generalized Fubini'' for coends or weighted colimits, analogous to \Cref{decomplemma}.

Again, the coherence conditions can be seen as a matter of uniqueness of the isomorphism by the universal property of weighted colimits.

\subsection{The image is a morphism of monads}\label{immm}

Here we want to show that the image is a (pseudo)morphism of (pseudo)monads, following \Cref{defmorphmonads}. The unit modification $u$ is given by the isomorphism of \Cref{unmm}. Again, the fact that this gives a modification comes from the universal property. 

The multiplication modification $m$ is in the following form,
\begin{equation}\label{multmm}
\begin{tikzcd}[sep=large]
	{\cat{Diag(Diag(C))}} & {\cat{PPC}} \\
	{\cat{Diag(C)}} & {\cat{PC}}
	\arrow["{(\im\im)_\cat{C}}", from=1-1, to=1-2]
	\arrow["{\mu_C^\cat{Diag}}"', from=1-1, to=2-1]
	\arrow["{\mu_C^\cat{P}}", from=1-2, to=2-2]
	\arrow["{\im_C}"', from=2-1, to=2-2]
	\arrow[Rightarrow, "{m_\cat{C}}", from=1-2, to=2-1, shorten <=16pt, shorten >=16pt]
\end{tikzcd}
\end{equation}
where $\im\im$ is shorthand for the following composite.
$$
\begin{tikzcd}
	{\cat{Diag(Diag(C))}} && {\cat{Diag(PC)}} && {\cat{P P C}}
	\arrow["{(\im_C)_*}", from=1-1, to=1-3]
	\arrow["{\im_\cat{PC}}", from=1-3, to=1-5]
\end{tikzcd}
$$
(Note that, since the interchange law of pseudonatural transformations holds only up to natural isomorphism, a priori horizontal composition is not uniquely defined, as in a weak bicategory. The choice we make, which will be consistent throughout the document, is motivated by later convenience.)

Explicitly, $m$ is given as follows. Let $D:\cat{J}\to\cat{Diag(C)}$ be the following diagram of diagrams.
$$
\begin{tikzcd}
 J
\end{tikzcd}
\qquad\rightsquigarrow\qquad
\begin{tikzcd}
 D_0J \ar{r}{D_1J} & \cat{C}
\end{tikzcd}
$$
Then, for every $C\in\cat{C}$, writing images and $\mu$ in terms of coends, the two paths of diagram \eqref{multmm} are the following objects,
$$
\int^{P\in\cat{PC}} \int^{J\in\cat{J}} \Hom_{PC}\left( P, \int^{K\in D_0J} \Hom_\cat{C} (-, D_1J(K)) \right) \times P(C) 
$$
and
$$
\int^{J\in\cat{J}} \int^{K\in D_0(J)} \Hom_\cat{C}(X, D_1J(K)) 
$$
and they are again isomorphic by the co-Yoneda lemma (\Cref{coyoneda}), over $P$.
This isomorphism is what we take as the multiplicator $m$. Again, the fact that it forms a modification follows from uniqueness, and so do the higher coherence conditions of \Cref{defmorphmonads}.

\subsubsection{The pullback functor of algebras}\label{pbfunctor}

In algebra, given a ring morphism $f:R\to R'$, every $R'$-module is canonically an $R$-module too, via the map $f$, and morphisms of $R'$-modules are morphisms of $R$-modules too.
The resulting ``pullback'' functor between the categories of $R'$-modules and $R$-modules is known as the ``restriction of scalars''~\cite[Chapter~II]{bourbakialg}, or ``Weil restriction'' in algebraic geometry~\cite[Section~1.3]{weilrestriction}. 
Not every $R$-module arises this way if $f$ is not an isomorphism (for example, $R$ itself, seen as an $R$-module, does not).

More generally, given a morphism of monads $\lambda: T\to T'$, every $T'$-algebra is canonically a $T$-algebra via $\lambda$, and morphisms of $T'$-algebras are automatically morphisms of $T$-algebras.
This is well known (see, for example, \cite[Theorem~3 in Section~3.6]{ttt}), and the pullback functor is again called ``restriction of scalars'', after its instance for the case of rings.
Again, not every $T$-algebra arises this way (if $\lambda$ is not an isomorphism), for example, the free algebra $(TX,\mu)$, where $X$ is any object, in general does not. 

With $\im:\cat{Diag}\to\cat{P}$ we are witnessing an instance of this phenomenon in higher dimensions: every $\cat{P}$-algebra, i.e.~a cocomplete category, is also a $\cat{Diag}$-algebra, via the map $\im$, which is a morphism of monads. 
As we have seen in \Cref{secnotallalg}, not every $\cat{Diag}$-algebra arises this way, for example, in general free algebras do not. 
The 2-dimensional restriction-of-scalars theorem is given in \Cref{restrictionofscalars} as \Cref{rescalthm}. 

To see that  cocomplete categories as $\cat{Diag}$-algebras indeed arise in this way, note that the following diagram commutes up to natural isomorphism for each cocomplete category $\cat{C}$ (and any choices of colimits $c$ and coends $e$).
$$
\begin{tikzcd}[column sep=0, row sep=small]
	{\cat{Diag(C)}} &&[+30pt] & {\cat{PC}} \\[-10pt]
	&& {\cong} \\
	\\
	&&& {\cat{C}}
	\arrow["{\im_\cat{C}}", from=1-1, to=1-4]
	\arrow["{e}", from=1-4, to=4-4]
	\arrow["{c}"', from=1-1, to=4-4]
\end{tikzcd}
$$
Indeed, the fact that this diagram commutes is, for the last time, an instance of the general co-Yoneda lemma (\Cref{generalcoyoneda}): given $D:\cat{J}\to\cat{C}$,
$$
\int^{X\in\cat{C}} \int^{J\in\cat{J}} \Hom_\cat{C}(X, DJ) \otimes X \;\cong\;\int^{J} DJ,
$$
which is indeed just the colimit of $D$, given up to isomorphism by $c$.
(Compare this with the analogous ``free'' case of \Cref{coendexists}.)

\section{Partial colimits}\label{partialcolimits}\label{secpev}

We review here the basic ideas of partial evaluations, which are a categorical formalization of the idea of ``computing only pieces of an operation''. We will apply this to the operation of colimit encoded by the monads $\cat{Diag}$ and $\cat{P}$.
We will then show that, for both monads, we have a correspondence between Kan extensions and partial evaluations of colimits (Theorems \ref{weakform}, \ref{strongform} and \ref{kanpevP}). Intuitively, these results may be interpreted as the fact that ``a left Kan extension is a partially computed colimit''. 
While the statement for the $\cat{Diag}$ monad is rather straightforward, the corresponding statement for $\cat{P}$ requires quite more work.

\subsection{Partial evaluations}

Partial evaluations were introduced in \cite[Chapter~4]{thesis} for the case of probability monads, and defined for the general case in \cite{pev}. A detailed study of their compositional structure (in general they don't form a category) is given in \cite{act2019}.

\begin{deph}
 Let $(T,\mu,\nu)$ be a monad on $\cat{Set}$, and let $(A,e)$ be a $T$-algebra. Consider the parallel pair of maps 
 $$
 \begin{tikzcd}
  TTA \ar[shift left]{r}{\mu_A} \ar[shift right, swap]{r}{Te} & TA
 \end{tikzcd}
 $$
 of which $e:TA\to A$ is the coequalizer.  
 Given elements $p,q\in TA$, a \emph{partial evaluation} from $p$ to $q$ is an element $r\in TTA$ such that $\mu_A(r)=p$ and $Te(r)=q$. 
 
 If such a partial evaluation exists, we also say that \emph{$q$ is a partial evaluation of $p$} and that \emph{$p$ can be partially evaluated to $q$}.
\end{deph}

\begin{eg}
 Let $(T,\mu,\nu)$ be the free commutative monoid monad. Given a set $X$, the elements of $TX$ can be thought of as formal sums of elements of $X$, for example in the form $x+y+z$ with $x,y,z\in X$. 
 Natural numbers with addition form a commutative monoid, and so $\N$ forms a $T$-algebra. 
 The formal sum 
 $$
 3 + 4 + 5 + 6 ,
 $$
 seen as an element of $T\N$, can be partially evaluated to the formal sum 
 $$
 7 + 11
 $$
 via the element $[3+4]+[5+6]\in TT\N$.
 The interpretation is that ``we haven't summed everything together, but only some of the terms''. 
\end{eg}

An essential property of partial evaluations is that ``they don't change the total result''. This is reflected by the multiplication diagram of the algebra,
$$
\begin{tikzcd}
	{TTA} & {TA} \\
	{TA} & {A}
	\arrow["{Te}", from=1-1, to=1-2]
	\arrow["{\mu}", from=1-1, to=2-1]
	\arrow["{e}", from=2-1, to=2-2]
	\arrow["{e}", from=1-2, to=2-2]
\end{tikzcd}
$$
which can be seen as saying that evaluating (via the map $e$) two formal expressions which differ by a partial evaluation, the (total) result of the evaluation is the same. In the example above, both formal sums evaluate to $18$.

We refer the interested reader to \cite{pev} and \cite{act2019} for more details and examples.
In our case we need a higher-dimensional, weaker analogue of the concept, since we are dealing with pseudomonads on $\cat{CAT}$ rather than monads on $\cat{Set}$.

\begin{deph}\label{defweakpev}
 Let $(T,\mu,\nu)$ be a pseudomonad on $\cat{CAT}$, and let $(\cat{A},e)$ be a $T$-pseudoalgebra. 
 Given objects $P,Q\in \cat{A}$, a \emph{partial evaluation} from $P$ to $Q$ is an object $R$ of $TTA$ such that $\mu_\cat{A}(R)\cong P$ and $Te(R)\cong Q$. 
 
 If such a partial evaluation exists, we also say that \emph{$Q$ is a partial evaluation of $P$} and that \emph{$P$ can be partially evaluated to $Q$}.
\end{deph}

\subsection{Partial evaluations of diagrams}

We now instance \Cref{defweakpev} for the case of the monad $\cat{Diag}$, keeping in mind the following multiplication square, which commutes up to isomorphism.
\begin{equation}\label{weakdiagsquare}
 \begin{tikzcd}[sep=tiny]
  \cat{Diag(Diag(C))} \ar{dd}{\mu} \ar{rr}{c_*} && \cat{Diag(C)} \ar{dd}{c} \\
  & \cong \\
  \cat{Diag(C)} \ar{rr}{c} && \cat{C}
 \end{tikzcd}
\end{equation}

\begin{deph}
 Let $\cat{C}$ be a cocomplete category, and let $D:\cat{J}\to\cat{C}$ and $D':\cat{K}\to\cat{C}$ be small diagrams. A \emph{partial evaluation} from $D$ to $D'$ for the monad $\cat{Diag}$ is an object $E$ of $\cat{Diag(Diag(C))}$ such that $\mu(E)\cong D$ and $c_*E\cong D'$.
 If such an object exists, we also say that $D'$ is a \emph{partial colimit} of $D$ for the monad $\cat{Diag}$.
\end{deph}

As we see shortly, we now establish an equivalence between partial evaluations of diagrams and left Kan extensions along split opfibrations.
There are now two ways of talking about the correspondence. One, which is probably the easiest to state, is as an equivalence of properties.

\begin{thm}\label{weakform}
 Let $\cat{C}$ be a cocomplete category, and let $D:\cat{J}\to\cat{C}$ and $D':\cat{K}\to\cat{C}$ be small diagrams. Then $D'$ is a partial colimit of $D$ (for the monad $\cat{Diag}$) if and only if $D'$ can be written as the (pointwise) left Kan extension of $D$ along a split opfibration.
\end{thm}

Instead of an equivalence of properties we can also write the result as an equivalence of structures, as follows.

\begin{thm}\label{strongform}
 Let $\cat{C}$ be a cocomplete category, and let $D:\cat{J}\to\cat{C}$ and $D':\cat{K}\to\cat{C}$ be small diagrams. There is a bijection between partial evaluations of colimits from $D$ to $D'$ and split opfibrations $F:\cat{J}\to\cat{K}$ exhibiting $D'$ as a (pointwise) left Kan extension of $D$ along $F$.
\end{thm}

$$
\begin{tikzcd}[row sep=small]
 \cat{J} 
  \ar[""{name=ONE,below}]{dr}{D} \ar{dd}[swap]{F} \\
 & \cat{C} \\
 \cat{K} 
  \ar{ur}[swap]{D' \;\cong\; \Lan{D}{F}} \ar[Rightarrow,from=ONE, shorten <= 0.5em, shorten >= 1em, "\lambda", near start, swap]
\end{tikzcd}
$$

We will prove the latter statement, since it clearly implies the former.
We will use the following property of Kan extensions along opfibrations, which can be interpreted as ``fiberwise colimits''.
The following two statements are well known (see for example \href{https://ncatlab.org/nlab/show/Kan+extension#AlongFibrations}{the nLab page on Kan extensions}).

\begin{prop}
 Let $F:\cat{E}\to\cat{B}$ be an opfibration between small categories. For each object $E$ of $\cat{E}$, the inclusion of the fiber $F^{-1}(E)$ into the comma category $F/E$ has a left adjoint. Hence, it is a confinal functor.
\end{prop}

\begin{cor}\label{onlyfibers}
 Let $F:\cat{E}\to\cat{B}$ be an opfibration between small categories, and let $\cat{C}$ be cocomplete and locally small. The (pointwise) left Kan extension of a functor $G:\cat{E}\to\cat{C}$ along $F$ at $B$ can then be computed by a colimit labeled by the fiber of $F$ at $B$:
 $$
 \Lan{G}{F}(B) \;\cong\; \colim_{E\in F^{-1}(B)} G(E) .
 $$
\end{cor}

We are now ready to prove the theorem. The crucial point of the theorem is the correspondence between split opfibrations and (strict!) functors into $\cat{Cat}$. 

\begin{proof}[Proof of \Cref{strongform}]
 First of all, suppose that we have a partial evaluation from $D$ to $D'$. That is, let $E=(E_0,E_1)$ be an object of $\cat{Diag(Diag(C))}$ such that $
 \mu(E)=D$ and $c_*E=D'$. Note that this implies that $E_0$ is necessarily indexed by $\cat{K}$ (up to isomorphism), it is a functor $\cat{K}\to\cat{Cat}$.
 We can now express $c_*(E)$ as the left Kan extension of $\mu(E)$ along the Grothendieck fibration $\pi:\groth E_0\to \cat{K}$. Indeed, using \Cref{onlyfibers}, we have that for each object $K$ of $\cat{K}$,
 \begin{align*}
 \Lan{\mu(E)}{\pi} (K) \;&\cong\; \colim_{(K,X)\in \pi^{-1}(K)} \mu(E)(K,X) \\
 \;&=\; \colim_{X\in E_0K} E_1K(X) \\
 \;&\cong\; c_*(E)(K) .
 \end{align*}
 
 Conversely, let $F:\cat{J}\to\cat{K}$ be a split opfibration and suppose that $D'$ is the left Kan extension of $D$ along $F$.
 $$
\begin{tikzcd}[row sep=small]
 \cat{J} 
  \ar[""{name=ONE,below}]{dr}{D} \ar{dd}[swap]{F} \\
 & \cat{C} \\
 \cat{K} 
  \ar{ur}[swap]{D'} \ar[Rightarrow,from=ONE, shorten <= 0.5em, shorten >= 1em, "\lambda", near start, swap]
\end{tikzcd}
 $$
 Let now $E_0:\cat{K}\to\cat{Cat}\subseteq\cat{CAT}$ be the functor associated with the opfibration $F$. Concretely, this is the functor mapping each object $K$ of $\cat{K}$ to its preimage $F^{-1}(K)$, and each morphism $k:K\to K'$ to the functor $F^{-1}(K)\to F^{-1}(K')$ given by the opcartesian lifts.
 Define also, for each object $K$ of $\cat{K}$, the functor $E_1K:E_0K\to \cat{C}$ to be given by 
 $$
 \begin{tikzcd}
 E_0K = F^{-1}(K) \ar[hookrightarrow]{r} & \cat{J} \ar{r}{D} & \cat{C} .
 \end{tikzcd}
 $$
 For each morphism $k:K\to K'$, define $E_1j'$ to be the 2-cell given by the opcartesian lifts in $\cat{J}$, whiskered by $D$. 
 We have that, by construction, $\mu(E_0,E_1)=D$.
 Moreover, by \Cref{onlyfibers}, $D'=c_*(E_0,E_1)$.
\end{proof}

Note also that, in the hypotheses above, $D$ and $D'$ must necessarily have isomorphic colimits, either because ``Kan extensions can be stacked vertically'' (\Cref{verticalpaste}), or because of the multiplication square \eqref{weakdiagsquare}.
A different but related picture is given in \cite[Proposition~2.9]{decomposition}.

\subsection{Partial evaluations of presheaves}

We now give and prove a similar statement for the monad $\cat{P}$.
Let's instance \Cref{defweakpev} for the case of the monad $\cat{P}$.

\begin{deph}
 Let $\cat{C}$ be a cocomplete category, and let $P$ and $Q$ be small presheaves on $\cat{C}$. A \emph{partial evaluation} from $P$ to $Q$ for the monad $\cat{P}$ is a presheaf on presheaves $\Phi$ in $\cat{PPC}$ such that $\mu(\Phi)\cong P$ and $e_\sharp\Phi\cong Q$.
 If such an object exists, we also say that $Q$ is a \emph{partial colimit} of $P$ for the monad $\cat{P}$.
\end{deph}

This time we establish a correspondence with Kan extensions along any functor, not just along a split opfibration. Moreover, we only have a weak statement, analogous to \Cref{weakform}, an equivalence of properties rather than of structures.

\begin{thm}\label{kanpevP}
 Let $\cat{C}$ be a cocomplete category, and let $P,Q\in \cat{PC}$. The following conditions are equivalent.
 \begin{enumerate}
  \item\label{condPkan} There exists a small diagram $D:\cat{J}\to\cat{C}$ such that $\im D \cong P$, and a small category $\cat{K}$ with a functor $F:\cat{J}\to\cat{K}$ such that $\im (\Lan{D}{F})\cong Q$.
  $$
  \begin{tikzcd}[row sep=small]
 \cat{J} 
  \ar[""{name=ONE,below}]{dr}{D} \ar{dd}[swap]{F} \\
 & \cat{C} \\
 \cat{K} 
  \ar{ur}[swap]{\Lan{D}{F}} \ar[Rightarrow,from=ONE, shorten <= 0.5em, shorten >= 1em, "\lambda", near start, swap]
\end{tikzcd}
  $$
  \item\label{condPpev} There exists a partial evaluation from $P$ to $Q$ for the monad $\cat{P}$, i.e.~an object $\Phi\in \cat{PPC}$ such that $\mu(\Phi)\cong P$ and $e_\sharp(\Phi) \cong Q$.
  $$
  \begin{tikzcd}[column sep=small]
   & \Phi\in \cat{PPC} \ar[mapsto]{dl}[swap]{\mu} \ar[mapsto]{dr}{e_\sharp} \\
   P \in\cat{PC} && Q \in\cat{PC}
  \end{tikzcd}
  $$
 \end{enumerate}
\end{thm}

In other words, $Q$ is a partial colimit of $P$ if and only if it can be written as the image presheaf of the left Kan extension of a diagram with image presheaf $P$.

The proof of this statement, which can be considered the main result of this paper, requires more work than the analogous statement for $\cat{Diag}$, and will have to use two auxiliary lemmas.

\begin{lemma}\label{cofibrep}
 Let $F:\cat{J}\to\cat{K}$ be a functor between small categories. There exist a small category $\cat{H}$, a confinal functor $G:\cat{H}\to\cat{J}$, and a split opfibration $\pi:\cat{H}\to\cat{K}$ such that for every locally small category $\cat{C}$ and every diagram $D:\cat{J}\to\cat{C}$, the pointwise left Kan extension of $D \circ G$ along $\pi$ exists if and only if the one of $D$ along $F$ exists, and in that case the two Kan extensions are naturally isomorphic.
\end{lemma}

We can depict the situation as follows:
$$
\begin{tikzcd}
 \cat{H} \ar{dr}[swap]{\pi} \ar{r}{G} & \cat{J} \ar{d}{F} \ar{r}{D} & \cat{C} \\
 & \cat{K} \ar{ur}[swap]{\Lan{(D\circ G)}{\pi}\;\cong\;\Lan{D}{F}}
\end{tikzcd}
$$
Note that the diagram above does not necessarily commute, nor does any of its two subdiagrams.

\begin{proof}[Proof of \Cref{cofibrep}]
Given an object $K$ of $\cat{K}$, the comma category $F/K$ has 
  \begin{itemize}
   \item As objects, pairs $(J,k)$ where $J$ is an object of $\cat{J}$ and $k:FJ\to K$ is a morphism of $\cat{K}$;
   \item As morphisms, commutative diagrams in the following form,
   $$
   \begin{tikzcd}[column sep=small]
    FJ \ar{rr}{Fj} \ar{dr}[swap]{k} && FJ' \ar{dl}{k'} \\
    & K
   \end{tikzcd}
   $$
   where $j:J\to J'$ is a morphism of $\cat{J}$.
  \end{itemize}
  Now define the functor $F/-:\cat{K}\to\cat{Cat}$ as follows. 
  \begin{itemize}
   \item To each object $K$ of $\cat{K}$ assign the comma category $F/K$;
   \item To each morphism $\ell:K\to K'$ of $\cat{K}$, assign the functor $F/K\to F/K'$ given by post-composition with $\ell$.
  \end{itemize}
Choose now as $\cat{H}$ the Grothendieck construction $\groth F/-$, and notice that it is small. Up to isomorphism, 
  \begin{itemize}
   \item Its objects are triplets $(K,J,k)$ where $K$ is an object of $\cat{K}$, $J$ is an object of $\cat{J}$, and $k:FJ\to K$ is a morphism of $\cat{K}$;
   \item Each morphism $\ell:K\to K'$ of $\cat{K}$ defines an ``opcartesian'' morphism $\ell_*:(K,J,k:FJ\to K)\to (K',J,\ell\circ k: FJ\to K')$;
   \item For each object $K$ of $\cat{K}$, a morphism $j:(J,k:FJ\to K)\to (J', k':FJ'\to K)$ of $F/K$ (i.e.~a morphism $j:J\to J'$ of $\cat{J}$ with $k'\circ Fj= k$) defines a morphism ``in the fiber'' $(K,J,k) \to (K, J', k')$, which we denote again by $j$.
   \item Any other morphism of $\cat{H}$ is a composition of two morphisms in the two forms above.
  \end{itemize}
  The forgetful functor $\pi:\cat{H}\to \cat{K}$ which maps $(K,J,k)$ to $K$ is a split opfibration.
  We can also construct the forgetful functor $G:\cat{H}\to \cat{J}$ which maps $(K,J,k)$ to $J$. This functor is confinal: indeed, let $J$ be an object of $\cat{J}$. 
  \begin{itemize}
   \item The category $J/G$ is nonempty: it always contains 
   $$
   \begin{tikzcd}
    J \ar{r}{\id_J} & J =  G(FJ,J,\id_{FJ}) ;
   \end{tikzcd}
   $$
   \item The category $J/G$ is connected: let $(K',J',k':FJ'\to K')$ and $(K'',J'',k'':FJ''\to K'')$ be objects of $\cat{H}$, and let $f:J\to J'$ and $g:J\to J''$ be morphisms of $\cat{J}$. We can then form the following zigzag in $\cat{H}$:
   $$
   \begin{tikzcd}[column sep=-1.8em]
    &&& (FJ, J, \id_{FJ}) \ar[near end]{dl}[swap]{(Ff)_*} \ar[near end]{dr}{(Fg)_*} \\
    && (FJ', J, Ff)  \ar[near end]{dl}[swap]{{k'}_*} && (FJ'', J, Fg)  \ar[near end]{dr}{{k''}_*} \\
    & (K', J, k'\circ Ff) \ar[near end]{dl}[swap]{f} &&&&  (K'', J, k''\circ Fg) \ar[near end]{dr}{g} \\
     (K',J',k') &&&&&& (K'',J'',k'') .
   \end{tikzcd}
   $$
  \end{itemize}
  
  Suppose now the pointwise left Kan extension 
  $$
  \begin{tikzcd}[row sep=small]
 \cat{H} 
  \ar[""{name=ONE,below}]{dr}{D\circ G} \ar{dd}[swap]{\pi} \\
 & \cat{C} \\
 \cat{K} 
  \ar{ur}[swap]{\Lan{(D\circ G)}{\pi}} \ar[Rightarrow,from=ONE, shorten <= 0.5em, shorten >= 1em, "\lambda", near start, swap]
\end{tikzcd}
  $$
  exists.
 We claim that $\Lan{D\circ G}{\pi}\cong\Lan{D}{F}$. To see this, let $K$ be an object of $\cat{K}$. Since $\pi$ is an opfibration, by \Cref{onlyfibers} we can write $\Lan{(D\circ G)}{\pi}(K)$ up to isomorphism as the colimit of the following composite functor.
 $$ 
 \begin{tikzcd}
 \pi^{-1}(K)  \ar[hookrightarrow]{r} & \cat{H} \ar{r}{G} & \cat{J} \ar{r}{D} & \cat{C}
 \end{tikzcd}
 $$
 Recall now that $H = \groth F/-$, so that $\pi^{-1}(K)\subseteq \cat{H}$ is isomorphic to $F/K$, and the following diagram commutes,
 $$
 \begin{tikzcd}
  F/K \ar{dr}[swap]{U}  \ar[hookrightarrow]{r} & \groth F/- \ar{d}{G} \\
  & \cat{D}
 \end{tikzcd}
 $$
 where $U$ is the forgetful functor mapping $(J,k:FJ\to K)$ to just $J$.
 Therefore we can rewrite $\Lan{(D\circ G)}{\pi}(K)$ up to isomorphism as the colimit of this simpler composition,
 $$ 
 \begin{tikzcd}
 F/K \ar{r}{U} & \cat{J} \ar{r}{D} & \cat{C}
 \end{tikzcd}
 $$
 which gives the pointwise left Kan extension of $D$ along $F$.
\end{proof}

\begin{lemma}\label{imimeso}
 Let $\cat{C}$ be a locally small category. The functor $\im\im:\cat{Diag(Diag(C))}\to \cat{PPC}$ is essentially surjective on objects.
\end{lemma}

Recall that the following diagram commutes only up to natural isomorphism,
$$
 \begin{tikzcd}[sep=small]
  \cat{Diag(Diag(C))} \ar{dd}{\im} \ar{rr}{\im_*} && \cat{Diag(PC)} \ar{dd}{\im} \\
  & \cong \\
  \cat{P(Diag(C))} \ar{rr}{\im_\sharp} && \cat{PPC}
 \end{tikzcd}
 $$
and that, by convention, we denoted by $\im\im$ the top-right path. 

\begin{proof}[Proof of \Cref{imimeso}]
 Since $\im$ is essentially surjective, it suffices to prove that $\cat{Diag}\im$ is as well.
 In other words, we have to prove that every small diagram of small presheaves can be obtained from a diagram of diagrams (by taking the image presheaf of each subdiagram).
 
 Let $\cat{I}$ be a small category, and let $D:\cat{I}\to \cat{PC}$ be a diagram. Explicitly, for each object $I$ of $\cat{I}$ we have a small presheaf $DI:\cat{C}^\op\to\cat{Set}$, and for each morphism of $\cat{I}$ we have a morphism of presheaves.
 We can now take the category of elements of each $DI$, as described in Section~\ref{catel}. For each $I$ of $\cat{I}$ we have a discrete fibration $\pi_I:\groth^o DI \to \cat{C}$, and for each morphism $i:I\to J$ of $\cat{I}$ we have a commutative triangle as follows.
 $$
 \begin{tikzcd}[column sep=small]
  \groth^o DI \ar{dr}[swap]{\pi_I} \ar{rr}{\groth^o Di} && \groth^o DJ \ar{dl}{\pi_J} \\
  & \cat{C}
 \end{tikzcd}
 $$
 However, the categories $\groth^o DI$ are large in general, and so they do not form the desired diagram of (small) diagrams yet. We then proceed as follows. Since each presheaf $DI$ is small, for each $I$ of $\cat{I}$ there exists (\Cref{ffconfinal}) a small full subcategory $\cat{S}_I$ of $\groth^o DI$, such that the inclusion functor is confinal. Denote by $F_I:\cat{S}_I\to \cat{C}$ the composition (or restriction)
 $$
 \begin{tikzcd}
 \cat{S}_I \ar[hookrightarrow]{r} & \groth^o DI \ar{r}{\pi_i} & \cat{C} .
 \end{tikzcd}
 $$
 By construction, $\im(F_i)\cong DI$. The assignment $I\mapsto \cat{S}_I$ is not functorial in $I$: for each morphism $i:I\to J$ of $\cat{I}$, we get the following commutative diagram.
 $$
 \begin{tikzcd}[column sep=small]
  \cat{S}_I \ar[hookrightarrow]{d} && \cat{S}_J \ar[hookrightarrow]{d} \\
  \groth^o DI \ar{dr}[swap]{\pi_I} \ar{rr}{\groth^o Di} && \groth^o DJ \ar{dl}{\pi_J} \\
  & \cat{C}
 \end{tikzcd}
 $$
 However, in principle we cannot lift $\groth^o Di$ to get a commutative square, that is, a priori the restriction of $\groth^o Di$ to $\cat{S}_I$ does not necessarily land in $\cat{S}_J$. We now extend the $\cat{S}_I$, as follows.
 
 Let $I$ be an object of $\cat{I}$, and consider the slice category $\cat{I}/I$, whose objects are pairs $(H,h)$ where $H$ is an object of $\cat{I}$, and $h:H\to I$ is an arrow of $\cat{I}$. This category is small, since its set of objects is given by
 $$
 \coprod_{H\in\mathrm{Ob}(\cat{I})} \Hom_\cat{I}(H,I) ,
 $$
 which is a small union of small sets.
 Now define $\cat{T}_I$ to be the full subcategory of $\groth^o DI$ whose set of objects is given by
 $$
 \bigunion_{(H,h)\in \cat{I}/I} \mathrm{Image} \left(\begin{tikzcd}
                                              \cat{S}_H \ar[hookrightarrow]{d} \\
                                              \groth^o DH \ar{r}{\groth^o Dh} & \groth^o DI
                                             \end{tikzcd}
\right) .
 $$
 That is, an object of $E_I$ is in the form $\groth^o Dh(C,x)$ for some object $(C\in\cat{C},x\in DH(C))$ in $\cat{S}_H$ and some morphism $h:H\to I$ of $\cat{I}$.
 The category $\cat{T}_I$ is small, since its set of objects is a small union of small sets. Moreover, since we can pick $h$ to be the identity $I\to I$, the category $\cat{S}_I$ is fully embedded into $\cat{T}_I$, i.e.~we have a commutative diagram of inclusions
 $$
 \begin{tikzcd}[column sep=small]
  \cat{S}_I \ar[hookrightarrow]{dr} \ar[hookrightarrow]{rr} && \cat{T}_I \ar[hookrightarrow]{dl} \\
  & \groth^o DI
 \end{tikzcd}
 $$
 which are all confinal by \Cref{confinalsubcat}. In particular, if we denote by $E_I: \cat{T}_I\to \cat{C}$ the composition (or restriction)
 $$
 \begin{tikzcd}
 \cat{T}_I \ar[hookrightarrow]{r} & \groth^o DI \ar{r}{\pi_i} & \cat{C} ,
 \end{tikzcd}
 $$
 we have again that $\im(E_I)\cong DI$. 
 
 This time, the assignment $I\mapsto \cat{T}_I$ is functorial. To see this, let $i:I\to J$ be a morphism of $\cat{I}$, and form the following commutative diagram.
 $$
 \begin{tikzcd}[column sep=small]
  \cat{T}_I \ar[hookrightarrow]{d} && \cat{T}_J \ar[hookrightarrow]{d} \\
  \groth^o DI \ar{dr}[swap]{\pi_I} \ar{rr}{\groth^o Di} && \groth^o DJ \ar{dl}{\pi_J} \\
  & \cat{C}
 \end{tikzcd}
 $$
 We claim that the restriction of $\groth^o Di$ to the subcategory $\cat{T}_I$ lands in $\cat{T}_J$. Any object of $\cat{T}_I$ is of the form $\groth^o Dh(C,x)$ for some $(C,x)$ in $\cat{S}_H$ and some $h:H\to I$ of $\cat{I}$. We then have
 $$
 \groth^o Di\left( \groth^o Dh (C,x) \right) = \groth^o D(i\circ h), (C,x),
 $$
 which lies in $\cat{T}_J$ since $i\circ h:H\to J$ belongs to the slice category $\cat{I}/J$. We can then complete the diagram to a commutative diagram
 $$
 \begin{tikzcd}[column sep=small]
  \cat{T}_I \ar{rr}{\cat{T}_i} \ar[hookrightarrow]{d} && \cat{T}_J \ar[hookrightarrow]{d} \\
  \groth^o DI \ar{dr}[swap]{\pi_I} \ar{rr}{\groth^o Di} && \groth^o DJ \ar{dl}{\pi_J} \\
  & \cat{C}
 \end{tikzcd}
 \qquad \mbox{or equivalently} \qquad
 \begin{tikzcd}[column sep=small]
  \cat{T}_I \ar{rr}{\cat{T}_i} \ar{dr}[swap]{E_I} && \cat{T}_J \ar{dl}{E_J} \\
  & \cat{C}
 \end{tikzcd}
 $$
 The assignment $I\mapsto\cat{T}_I$, $i\mapsto \cat{T}_i$ is a functor $\cat{I}\to\cat{Cat}$, and the corresponding assignment $I\mapsto (\cat{T}_I,E_I:\cat{E_I}\to\cat{C})$ gives a functor $E:\cat{I}\to \cat{Diag(C)}$. 
 
 This functor $E:\cat{I}\to \cat{Diag(C)}$ is the diagram of diagrams that we want. Indeed, $\cat{Diag}(\im)(E)$ is given by the postcomposition with the image
 $$
 \begin{tikzcd}
  \cat{I} \ar{r}{E} & \cat{Diag(C)} \ar{r}{\im} & \cat{PC}\;,
 \end{tikzcd}
 $$
 and we know that for each $I$ of $\cat{I}$, $\im(E_I) \cong DI$. To show that $\cat{Diag}(\im)(E)\cong D$, and hence conclude the proof, we have to show that the isomorphism $\im(E_I) \to DI$ is natural in $I$. So let $i:I\to J$ be a morphism of $\cat{I}$. 
 We have to show that the following diagram of $\cat{PC}$ commutes,
 \begin{equation}\label{natEI}
 \begin{tikzcd}
  \im(E_I) \ar{r}{\im(E_i)} \ar{d}{\cong} & \im(E_J) \ar{d}{\cong} \\
  DI \ar{r}{Di} & DJ 
 \end{tikzcd}
 \end{equation}
 where the isomorphisms are the ones given in \Cref{whichiso}.
 Diagram~\eqref{natEI} commutes since the vertices are the colimits of the functors $\eta\circ E_i$, $\eta\circ E_J$, $\eta\circ\pi_I$ and $\eta\circ\pi_J$ obtained by the following commutative diagram of $\cat{CAT}$,
 \begin{equation}\label{squareoflargediags}
 \begin{tikzcd}[column sep=small]
  \cat{T}_I \ar[bend left]{dddrrr}{E_I} \ar{rr}{\cat{T}_i} \ar[hookrightarrow]{dd} && \cat{T}_J \ar[hookrightarrow, crossing over]{dd} \ar[bend left]{dddr}{E_J} \\ \\
  \groth^o DI \ar[bend right=20]{drrr}[swap]{\pi_I} \ar{rr}{\groth^o Di} && \groth^o DJ \ar{dr}[swap]{\pi_J} \\
  && & \cat{C} \ar{r}{\eta} & \cat{PC}
 \end{tikzcd}
 \end{equation}
 and the induced maps between their colimits are the arrows of \eqref{natEI}. Since the square in \eqref{squareoflargediags} commutes, \eqref{natEI} commutes too (by uniqueness of the induced map). 
\end{proof}

We are now finally ready to prove the theorem.

\begin{proof}[Proof of \Cref{kanpevP}]
Since $\im$ is a morphism of monads, the following diagram commutes up to natural isomorphism.
\begin{equation}\label{pev2pev}
\begin{tikzcd}
 \cat{Diag(C)} \ar{d}{\im} & \cat{Diag(Diag(C))} \ar{l}[swap]{\mu} \ar{r}{\cat{Diag} \colim} \ar{d}{\im\,\im} & \cat{Diag(C)} \ar{d}{\im} \\
 \cat{PC} & \cat{PPC} \ar{l}[swap]{\mu} \ar{r}{\cat{P}k} & \cat{PC}
\end{tikzcd}
\end{equation}
In particular, this can be interpreted as ``$\im$, as a morphism of monads, preserves partial evaluations''.
With this and the previous lemmas in mind, let's proceed with the proof.

 \begin{itemize}
  \item $\ref{condPkan}\Rightarrow\ref{condPpev}$: Consider $D:\cat{J}\to\cat{C}$ and $F:\cat{J}\to\cat{K}$ as in \ref{condPkan}. 
  By \Cref{cofibrep} there exists a diagram $D:\cat{H}\to\cat{C}$ together with a confinal functor $G:\cat{H}\to\cat{J}$ and an opfibration $\pi:\cat{H}\to\cat{K}$ with $\Lan{D\circ G}{\pi}\cong\Lan{D}{F}$. Since $G$ is confinal, by \Cref{confinallemma} we have that $\im (D\circ G) \cong \im D \cong P$. 
  Now, by \Cref{weakform}, there exists a diagram of diagrams $E\in\cat{Diag(Diag(C))}$ such that $\mu(E)\cong D\circ G$ and $\cat{Diag}\colim (E)\cong \Lan{D\circ G}{\pi}$. Chasing diagram~\eqref{pev2pev}, we see that $\im\im(E)$ is the desired partial evaluation between $\im (D\circ G) \cong \im D \cong P$ and $\im(\Lan{D\circ G}{\pi})\cong\im(\Lan{D}{F})\cong Q$.
 
  \item $\ref{condPpev}\Rightarrow\ref{condPkan}$: Let $\Phi$ be a partial evaluation from $P$ to $Q$. By \Cref{imimeso}, $\im\im$ is essentially surjective on objects, so that there exists $E\in \cat{Diag(Diag(C))}$ with $\im\im(E)\cong \Phi$. Chasing diagram~\eqref{pev2pev}, let $D=\mu(E)$ and $D'=\cat{Diag}\colim(E)$, so that $\im D \cong P$ and $\im D'\cong Q$. By \Cref{weakform}, $D'$ can be written as the pointwise left Kan extension of $D$ along a split opfibration.
  \qedhere
 \end{itemize}

\end{proof}

\subsection{Comparison with measure theory}\label{integrals}

Coends, in particular when they denote weighted colimits, are often considered similar to integrals in analysis, which is why they are normally denoted by an integral sign.
The correspondence can roughly be summarized in terms of monads, by saying that the monad $\cat{P}$ on $\cat{CAT}$ behaves similarly to the Giry monad $P$ on the category of (say) measurable spaces \cite{girymonad}, and to other probability and measure monads. In particular,
\begin{itemize}
 \item small presheaves on a category are similar to measures on a measurable space;
 \item cocomplete categories (categories where one can take colimits) are similar to algebras over probability monads (spaces where one can take integrals or expectation values, such as the real line);
 \item in a (sufficiently) cocomplete category $\cat{C}$,  given a small presheaf $P\in \cat{PC}$ the following coend
 $$
 \int^{X\in\cat{C}} PX\times X
 $$
 is similar to the following integral of a measure $p$ on, say, real numbers,
 $$
 \int_{\R} x\,dp(x) ;
 $$ 
 \item as Kleisli morphisms, (small) profunctors are similar to Markov kernels. 
\end{itemize}

An introduction to probability monads, for the interested reader, can be found in \cite[Chapter~1]{thesis} as well as in \cite[Section~6]{pev}. 

We now wish to emphasize that \Cref{kanpevP} adds a further analogy between integrals and coends: just like Kan extensions can be thought of as ``partial colimits'', conditional expectations can be thought of as ``partial expectations''.
In particular, we compare \Cref{kanpevP} to a similar theorem for a probability monad on the category of metric spaces, the \emph{Kantorovich monad} (see \cite{kantorovich} as well as the more introductory material in \cite[Section~6]{pev}). 

The following statement is known (\cite[Theorem 6.13]{pev} or \cite[Theorem 4.2.14]{thesis}).

\begin{thm}\label{condexpthm}
 Let $(A,e)$ be a Banach space (an algebra of the Kantorovich monad). Let $p,q\in PA$ be Radon probability measures on $A$ of finite first moment. The following conditions are equivalent:
 \begin{itemize}
  \item There exist probability spaces $(\Omega,\mu)$ and $(\Omega',\mu')$, random variables $f:\Omega\to A$ and $g:\Omega'\to A$ with image measures $p$ and $q$ respectively, and a measure-preserving map $m:\Omega\to\Omega'$ such that $g$ is the conditional expectation of $f$ along (the pull-back sigma-algebra induced by) $m$ -- as in the following not necessarily commutative diagram:
  \begin{equation}\label{condexpdiag}
  \begin{tikzcd}[row sep=small]
   \Omega \ar{dr}{f} \ar{dd}[swap]{m} \\
   & A \\
   \Omega' \ar{ur}[swap]{g = \mathbbm{E}(f|m)}
  \end{tikzcd}
  \end{equation}
  \item There is a partial evaluation $k\in PPA$ from $p$ to $q$ (for the Kantorovich monad).
  $$
  \begin{tikzcd}
   & k \ar[mapsto]{dl}[swap]{E} \ar{dr}{Pe} \\
   p && q
  \end{tikzcd}
  $$
 \end{itemize}
\end{thm}

The similarity to \Cref{kanpevP} is evident. 
One could say intuitively that \emph{if coends are analogous to integrals, Kan extensions are analogous to conditional expectations}. The former can be interpreted as partial (weighted) colimits, and the latter as partial (weighted) averages.
Notice also that the diagram \eqref{condexpdiag} does not commute, but the conditional expectation map, just like a Kan extension, can be thought of as the one ``making the diagram as close as possible to commuting''.

It has to be noted that, while we proved \Cref{kanpevP} by invoking an analogous (but simpler and stronger) statement for the monad of diagrams, \Cref{condexpthm} was proven directly, using measure-theoretic methods.
In the statement of the \Cref{condexpthm} it seems that random variables play somewhat the role that diagrams play in \Cref{kanpevP} (or at least, of diagrams equipped with weights, analogous to the measures on $\Omega$ and $\Omega'$). However, currently it is not clear whether random variables, or related structures, may form a monad analogous to $\cat{Diag}$, and on which category such a structure could be found.


\appendix

\section{Some 2-dimensional monad theory}\label{pseudomonads}

Here we recall some definitions of 2-dimensional monad theory, and spell out explicitly the definitions that we use in the rest of this work.
We also prove a result, \Cref{rescalthm}, which extends to pseudomonads the ``restriction of scalars'' construction.
We follow the definitions given in \cite[Section~2.12]{lucatelli}. We refer the interested reader to that article, as well as \cite{pseudomonads} and \cite{marmolejo-doctrines}, for further details.

\subsection{Pseudomonads and their morphisms}

\begin{deph}\label{defpseudomonad}
 A \emph{pseudomonad} on a strict $2$-category $\cat{K}$ consists of 
 \begin{itemize}
  \item a pseudofunctor $T:\cat{K}\to\cat{K}$, together with
  \item pseudonatural transformations $\eta:\id_\cat{K}\Rightarrow T$ and $\mu:TT\Rightarrow T$, which we call \emph{unit} and \emph{multiplication} respectively, and
  \item invertible modifications
  $$
  \begin{tikzcd}[sep=large]
   T \nat{r}{\eta T} \ar[Rightarrow,""{name=ID, above right}]{dr}[swap]{\id} & |[alias=TT]| TT \nat{d}{\mu} \\
   & T
   \ar[triple, from=TT, to=ID, "\ell", shorten <= 0.2em, shorten >= 0.2em]
  \end{tikzcd}
  \qquad
  \begin{tikzcd}[sep=large]
   T \nat{r}{T\eta} \ar[Rightarrow,""{name=ID, above right}]{dr}[swap]{\id} & |[alias=TT]| TT \nat{d}{\mu}  \\
   & T
   \ar[triple, from=ID, to=TT, "r", shorten <= 0.2em, shorten >= 0.2em, swap]
  \end{tikzcd}
  \qquad
  \begin{tikzcd}[sep=large]
   TTT \nat{r}{T \mu} \nat{d}[swap]{\mu T} & |[alias=TTUR]| TT \nat{d}{\mu} \\
   |[alias=TTDL]| TT \nat{r}[swap]{\mu} & T
   \ar[triple, from=TTUR, to=TTDL, "a", shorten <= 0.8em, shorten >= 0.8em]
  \end{tikzcd}
 $$
  which we call (left and right) unitors and associator, respectively, such that
  \item the following two coherence laws are satisfied,
  $$
  \begin{tikzcd}[column sep=small]
   \mu \circ T\mu \circ T\eta T \ar[triple]{rr}{a (T\eta T)} \ar[triple,swap]{dr}{\mu (\widehat{T\ell})} && \mu\circ\mu T\circ T\eta T  \\ 
   & \mu \circ \id_{TT} \ar[triple]{ur}[swap]{\mu\circ rT}
  \end{tikzcd}
  $$
  $$
  \begin{tikzcd}[column sep=small]
   & \mu\circ T\mu\circ TT\mu \ar[triple]{dl}[swap]{a(TT\mu)} \ar[triple]{dr}{\mu(\widehat{Ta})} \\
   \mu\circ \mu T \circ TT\mu \ar[triple]{dd}[swap]{\mu(\mu_\mu)} && \mu\circ T\mu\circ T\mu T \ar[triple]{dd}{a(T\mu T)} \\ \\
   \mu\circ T\mu\circ \mu TT \ar[triple]{dr}[swap]{a(\mu TT)} && \mu\circ\mu T\circ T\mu T \ar[triple]{dl}{\mu(aT)} \\
   & \mu\circ \mu T\circ \mu TT
  \end{tikzcd}
  $$
  which we call the unit and pentagon\footnote{This name comes from the theory of monoidal categories, where the diagram has the shape of a pentagon (in that case, the analogue of $\mu$ is strictly natural, and so the arrow corresponding to $\mu_\mu$ is an identity).} condition, respectively, where $\widehat{T\ell} = \upsilon^{-1}\circ T\ell\circ\kappa$, $\widehat{Ta}=\kappa^{-1}\circ Ta\circ\kappa$,  
  and $\kappa$ and $\upsilon$ are the compositor and unitor, respectively, of the pseudofunctor $T$.
  The arrow $\mu_\mu$ is the invertible 3-cell filling the pseudonaturality square of $\mu$ along the morphism $\mu$ itself, of (2-cell) components
  $$
  \begin{tikzcd}
	{TTTTX} & {TTTX} \\
	{TTTX} & {TTX}
	\arrow["{\mu_{TTX}}"', from=1-1, to=2-1]
	\arrow["{\mu_{TX}}", from=1-2, to=2-2]
	\arrow["{T\mu_X}"', from=2-1, to=2-2]
	\arrow["{TT\mu_X}", from=1-1, to=1-2]
	\arrow[Rightarrow, from=1-2, to=2-1, shorten <=6pt, shorten >=6pt, "(\mu_{\mu})_X"]
  \end{tikzcd}
  $$
  for each object $X$ of $\cat{K}$.
 \end{itemize}
\end{deph}

For some readers it may be easier to interpret the coherence conditions above if we express them as 2-dimensional diagrams. In particular, the unit and pentagon conditions can be expressed as the following commutative cubes, respectively,
$$
\begin{tikzcd}[column sep=small]
 TT \nat{rr}{T\eta T} \nat{dd}[swap]{T\eta T} \nat{ddrr}[near end]{\id} && |[alias=TTTUR]| TTT \nat{dr}{T\mu} \nat{dd}{T\mu} \\
 & |[alias=ID]| && TT \nat{dd}{\mu} \\
 |[alias=TTTDL]| TTT \nat{dr}[swap]{\mu T} \nat{rr}[swap]{\mu T} && TT \nat{dr}{\mu} \\
 & TT \nat{rr}[swap]{\mu} && T
 \ar[triple, from=TTTUR, to=ID, "\widehat{T\ell}", shorten <= 0.5em, shorten >= 0.5em, swap]
 \ar[triple, from=ID, to=TTTDL, "rT", shorten <= 0.5em, shorten >= 0.5em, swap]
\end{tikzcd}
\qquad=\qquad
\begin{tikzcd}[column sep=small]
 TT \nat{rr}{T\eta T} \nat{dd}[swap]{T\eta T} \nat{dr}[swap, near end]{T\eta T} && TTT \nat{dr}{T\mu} \\
 & TTT \nat{rr}{T\mu} \nat{dd}{\mu T} && |[alias=TTUR]| TT \nat{dd}{\mu} \\
 TTT \nat{dr}[swap]{\mu T} \\
 & |[alias=TTDL]|  TT \nat{rr}[swap]{\mu} && T
    \ar[triple, from=TTUR, to=TTDL, "a", shorten <= 2em, shorten >= 2em]
\end{tikzcd}
$$
 where the unfilled 3-cells are identities, and $\widehat{T\ell}$ denotes the following composite cell,
 $$
 
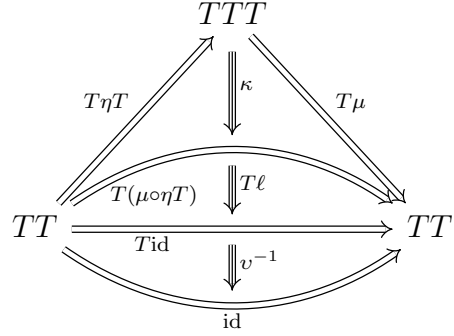
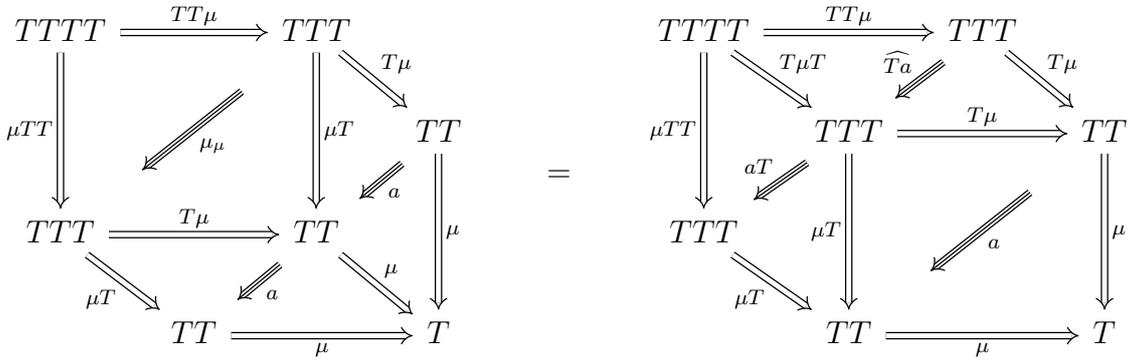
\begin{tikzcd}[sep=large]
  & |[alias=TTT]| TTT \nat{ddr}{T\mu} \\ \\
  TT \nat{uur}{T\eta T} 
  \ar[Rightarrow, bend left=35, ""{name=COMP}]{rr}[swap, pos=0.1, yshift=0.5ex]{T(\mu\circ \eta T)}
  \ar[Rightarrow, ""{name=TID}]{rr}[near start, swap]{T\id}
  \ar[Rightarrow, bend right=35, ""{name=ID}]{rr}[swap]{\id} 
  && TT
  \ar[triple, from=TTT, to=COMP, "\kappa", shorten <= 0.5em, shorten >= 0em]
  \ar[triple, from=COMP, to=TID, "T\ell", shorten <= 0.5em, shorten >= 0em]
  \ar[triple, from=TID, to=ID, "\upsilon^{-1}", shorten <= 0.5em, shorten >= 0em]
 \end{tikzcd}
 $$
and
$$
  \begin{tikzcd}[column sep=small]
    TTTT \nat{rr}{TT\mu} \nat{dd}[swap]{\mu TT} && |[alias=TTTUR]| TTT \nat{dr}{T\mu} \nat{dd}{\mu T} \\
    &&& |[alias=TTUR]| TT \nat{dd}{\mu} \\ 
    |[alias=TTTDL]| TTT \nat{rr}{T\mu} \nat{dr}[swap]{\mu T} && |[alias=TT]| TT \nat{dr}{\mu} \\
    & |[alias=TTDL]| TT \nat{rr}[swap]{\mu} && T
    \ar[triple, from=TTTUR, to=TTTDL, "\mu_\mu", shorten <= 2em, shorten >= 2em]
    \ar[triple, from=TTUR, to=TT, "a", shorten <= 0.5em, shorten >= 0.5em]
    \ar[triple, from=TT, to=TTDL, "a", shorten <= 0.5em, shorten >= 0.5em]
   \end{tikzcd}
      \qquad=\qquad
   \begin{tikzcd}[column sep=small]
    TTTT \nat{rr}{TT\mu} \nat{dd}[swap]{\mu TT} \nat{dr}{T\mu T} && |[alias=TTTUR]| TTT \nat{dr}{T\mu} \\
    & |[alias=TTT]| TTT \nat{rr}{T \mu} \nat{dd}[swap]{\mu T} && |[alias=TTUR]| TT \nat{dd}{\mu} \\ 
    |[alias=TTTDL]| TTT \nat{dr}[swap]{\mu T} \\
    & |[alias=TTDL]| TT \nat{rr}[swap]{\mu} && T
    \ar[triple, from=TTTUR, to=TTT, "\widehat{Ta}", shorten <= 0.5em, shorten >= 0.5em, swap]
    \ar[triple, from=TTT, to=TTTDL, "aT", shorten <= 0.5em, shorten >= 0.5em, swap]
    \ar[triple, from=TTUR, to=TTDL, "a", shorten <= 2em, shorten >= 2em]
   \end{tikzcd} 
  $$
 where $\widehat{Ta}$ denotes the following composite cell.
 $$
 
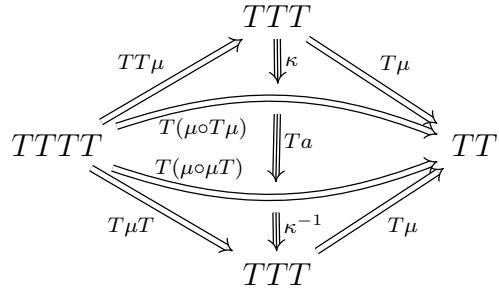
\begin{tikzcd}[sep=large]
 & |[alias=TOP]| TTT \nat{dr}{T\mu} \\
 TTTT \nat{ur}{TT\mu} \nat{dr}[swap]{T\mu T}
 \ar[Rightarrow, bend left=20, ""{name=COMPUP}]{rr}[swap, pos=0.1, yshift=0.4ex]{T(\mu\circ T\mu)}
 \ar[Rightarrow, bend right=20, ""{name=COMPDOWN}]{rr}[pos=0.1, yshift=-0.6ex]{T(\mu\circ\mu T)}
 && TT \\
 & |[alias=BOTTOM]| TTT \nat{ur}[swap]{T\mu}
 \ar[triple, from=TOP, to=COMPUP, "\kappa"]
 \ar[triple, from=COMPUP, to=COMPDOWN, "Ta", shorten <=0.5em]
 \ar[triple, from=COMPDOWN, to=BOTTOM, "\kappa^{-1}", shorten <=0.5em]
 \end{tikzcd}
 $$
 
We now give the definition of a morphism of pseudomonads (over the same category).

\begin{deph}\label{defmorphmonads}
 Let $\cat{K}$ be a strict 2-category. Let $(T,\mu,\eta,\ell,r,a)$ and $(T',\mu',\eta',\ell',r',a')$ be pseudomonads on $\cat{K}$. A \emph{(pseudo-) morphism of pseudomonads} consists of 
 \begin{itemize}
  \item a pseudonatural transformation $\lambda:T \Rightarrow T'$, together with
  \item invertible modifications as follows,
  $$
  
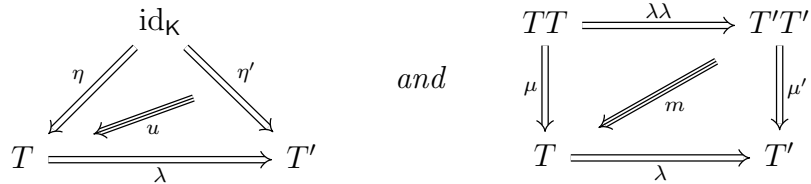
\begin{tikzcd}[row sep=large]
   & \id_\cat{K} \ar[Rightarrow, ""{name=ETAP}]{dr}{\eta'} \ar[Rightarrow]{dl}[swap]{\eta} & \\
   |[alias=T]| T \ar[Rightarrow]{rr}[swap]{\lambda}  && T'
   \ar[triple, from=ETAP, to=T, "u", shorten <= 1.2em, shorten >= 1.5em]
  \end{tikzcd}
  \qquad\mbox{and}\qquad
  \begin{tikzcd}[row sep=large]
   TT \ar[Rightarrow]{rr}{\lambda\lambda} \ar[Rightarrow]{d}[swap]{\mu} && |[alias=TPTP]| T'T' \ar[Rightarrow]{d}{\mu'} \\
  |[alias=T]| T  \ar[Rightarrow]{rr}[swap]{\lambda} && T' 
   \ar[triple, from=TPTP, to=T, "m", shorten <= 1em, shorten >= 1em]
  \end{tikzcd}
  $$
  where we recall that $\lambda\lambda$, as in \Cref{immm}, is shorthand for the following composition,
  $$
  \begin{tikzcd}
   TT \ar[Rightarrow]{r}{T\lambda} & TT' \nat{r}{\lambda T'} & |[alias=TPTP]| T'T' \;,
  \end{tikzcd}
  $$
  such that
  \item the following diagrams commute,
  $$
  \begin{tikzcd}
   \mu'\circ (\eta'T')\circ \lambda \ar[triple]{d}[swap]{\ell'\,\lambda} \ar[triple]{r}{\mu' \,\widehat{u T'\eta}} & \mu' \circ \lambda\lambda \circ \eta T \ar[triple]{d}{m\,(\eta T)} \\ 
   \lambda   & \lambda\circ\mu\circ \eta T \ar[triple]{l}{\lambda(\ell)}
  \end{tikzcd}
  \qquad
  \begin{tikzcd}
  \lambda \ar[triple]{d}[swap]{r'\,\lambda} \ar[triple]{r}{\lambda(r)}  & \lambda\circ\mu\circ T\eta    \\
   \mu'\circ (T'\eta')\circ \lambda  \ar[triple]{r}{\mu' \,\widehat{Tu\lambda}} & \mu' \circ \lambda\lambda \circ T\eta  \ar[triple]{u}[swap]{m\,(T\eta)}
  \end{tikzcd}
  $$
  $$
  \begin{tikzcd}[column sep=tiny]
   & \mu'\circ T'\mu'\circ \lambda\lambda\lambda \ar[triple]{dl}[swap]{a'(\lambda\lambda\lambda)} \ar[triple]{dr}{\mu'\,\widehat{Tm\lambda}} \\
   \mu'\circ\mu'T'\circ \lambda\lambda\lambda \ar[triple]{dd}[swap]{\mu'\,\widehat{mT'\mu}} && \mu'\circ\lambda\lambda\circ T\mu \ar[triple]{dd}{m\,(T\mu)} \\ \\
   \mu'\circ\lambda\lambda\circ\mu T \ar[triple]{dr}[swap]{m\,(\mu T)} && \lambda\circ\mu\circ T\mu \ar[triple]{dl}{\lambda a} \\
   & \lambda\circ\mu\circ\mu T
  \end{tikzcd}
  $$
  which we call left unitality, right unitality, and associativity conditions, respectively, where 
  \begin{itemize}
  \item $\widehat{uT'\eta}$ denotes $(\lambda T'\, \eta_\lambda) \circ (uT'\,\lambda)$,
  \item $\widehat{Tu\lambda}$ denotes $\lambda T' \,( \kappa^{-1}\circ Tu) \circ \lambda_{\eta'}^{-1}$,
  \item $\widehat{mT'\mu}$ denotes $(\lambda T'\,\mu_\lambda)\circ (mT'\, TT\lambda)$, and
  \item $\widehat{Tm\lambda}$ denotes $(\lambda T'\,(\kappa^{-1}\circ Tm\circ \kappa\kappa))\circ (\lambda_{\mu'}^{-1}\,T\lambda T'\,TT\lambda)$.
  \end{itemize}
 \end{itemize}
\end{deph}
We can write the coherence conditions 2-dimensionally, as follows.
The unitality conditions are the following two commutative prisms,
$$
\begin{tikzcd}[sep=large]
 & T' \ar[Rightarrow, ""{name=ID}]{dr}[swap]{\id} \ar[Rightarrow]{r}{\eta'T'} & |[alias=TPTP]| T'T' \ar[Rightarrow]{d}{\mu'} \\
 T \ar[Rightarrow]{ur}{\lambda} \ar[Rightarrow]{dr}[swap]{\id} && T' \\
 & T \ar[Rightarrow]{ur}[swap]{\lambda}
 \ar[triple, from=TPTP, to=ID, "\ell'", shorten <=0.3em, shorten >=0.2em]
\end{tikzcd}
\qquad=\qquad
\begin{tikzcd}[sep=large]
 & T' \ar[Rightarrow]{r}{\eta'T'} 
 & T'T' \ar[Rightarrow]{d}{\mu'} \\
 T \ar[Rightarrow]{r}{\eta T} \ar[Rightarrow]{ur}{\lambda} \ar[Rightarrow, ""{name=ID}]{dr}[swap]{\id} 
 \ar[phantom, bend left, ""{name=LET, below, pos=0.55}]{urr}
 \ar[phantom, bend right, ""{name=ETLL, above, pos=0.45}]{urr}
 & |[alias=TT]| TT \ar[Rightarrow]{ur}{\lambda\lambda} \ar[Rightarrow]{d}{\mu} 
 \ar[phantom, bend left=45, ""{name=LLM, below, pos=0.7}]{r}
 \ar[phantom, bend right=45, ""{name=ML, above, pos=0.3}]{r}
 & T' \\
 & T \ar[Rightarrow]{ur}[swap]{\lambda}
 \ar[triple, from=TT, to=ID, "\ell", shorten <=0.3em, shorten >=0.2em]
  \ar[triple, from=LET, to=ETLL, shorten <= 1em, shorten >= 1em, "\widehat{uT'\eta}", swap]
  \ar[triple, from=LLM, to=ML, shorten <=0.5em, shorten >=0.5em, "m"]
\end{tikzcd}
$$
where $\widehat{uT'\eta}$ denotes the following composite,
$$
\begin{tikzcd}[column sep=huge, row sep=large]
 T \nat{rr}{\lambda} \nat{dd}[swap]{\eta T}  && T' \ar[Rightarrow, ""{name=ETAT}]{dd}{\eta' T'} \ar[Rightarrow]{ddl}[swap]{\eta T'} 
 \ar[triple, shorten <=7em, shorten >=4em]{ddll}[swap, pos=0.6]{\eta_\lambda}
 \\ \\
 TT \nat{r}[swap]{T\lambda} 
 & TT' 
 \ar[triple, from=ETAT, "uT'", shorten <=2em, shorten >=2.5em]
 \nat{r}[swap]{\lambda T'} & T'T'
\end{tikzcd}
$$
and
$$
\begin{tikzcd}[sep=large]
 & T' \ar[Rightarrow]{d}[swap, near end]{T'\eta'} \ar[Rightarrow, ""{name=ID, below}]{dr}{\id}  \\
 T \ar[Rightarrow]{d}[swap]{T\eta} \ar[Rightarrow]{ur}{\lambda} 
 \ar[phantom, bend left, ""{name=LET, below, pos=0.7}]{r}
 \ar[phantom, bend right, ""{name=ETLL, above, pos=0.4}]{r}
 & |[alias=TPTP]| T'T' \ar[Rightarrow]{r}[swap, near start]{\mu'}  & T' \\
 |[alias=TT]| TT \ar[Rightarrow]{ur}[swap]{\lambda\lambda} \ar[Rightarrow]{r}[swap]{\mu} 
 \ar[phantom, bend left=45, ""{name=LLM, below, pos=0.6}]{urr}
 \ar[phantom, bend right=45, ""{name=ML, above, pos=0.4}]{urr} 
 & T \ar[Rightarrow]{ur}[swap]{\lambda}
 \ar[triple, from=ID, to=TPTP, "r'", shorten <=0.3em, shorten >=0.2em, swap]
   \ar[triple, from=LET, to=ETLL, shorten <= 0.2em, shorten >= 0.2em, "\widehat{Tu\lambda}", swap, near end]
  \ar[triple, from=LLM, to=ML, shorten <=2em, shorten >=2em, "m"]
\end{tikzcd}
\qquad=\qquad
\begin{tikzcd}[sep=large]
 & T' \ar[Rightarrow, ""{name=IDP}]{dr}{\id} \\
 T \ar[Rightarrow]{d}[swap]{\eta} \ar[Rightarrow]{ur}{\lambda} \ar[Rightarrow, ""{name=ID, below}]{dr}{\id} && T' \\
 |[alias=TT]| TT \ar[Rightarrow]{r}[swap]{\mu}  & T \ar[Rightarrow]{ur}[swap]{\lambda}
 \ar[triple, from=ID, to=TT, "r", shorten <=0.2em, shorten >=0.2em]
\end{tikzcd}
$$
where $\widehat{Tu\lambda}$ denotes the following composite and the unfilled 3-cells are intended to be identities;
$$
\begin{tikzcd}[column sep=huge, row sep=large]
 T \nat{rr}{\lambda} \nat{dd}[swap]{T\eta} \ar[Rightarrow, bend left=20, ""{name=TETAP, below}]{ddr}[near end]{T\eta'} \ar[Rightarrow, bend right=20, ""{name=TCOMP, pos=0.4}]{ddr}[near end, swap, yshift=0.2ex, xshift=0.8ex]{T(\lambda\circ\eta)} && T' \nat{dd}{T'\eta'} 
 \ar[triple, shorten <= 2em, shorten >= 2em]{ddl}{\lambda_{\eta'}^{-1}} \\ \\
 TT \nat{r}[swap]{T\lambda} 
 \ar[triple, from=TCOMP, swap, shorten <=1em, shorten >=1em, near start]{}[inner sep=-0.5ex]{\kappa^{-1}}
 & TT' \nat{r}[swap]{\lambda T'} & T'T'
 \ar[triple, from=TETAP, to=TCOMP, shorten <=0.5em, "Tu", pos=0.6]
\end{tikzcd}
$$
associativity is the following commutative cube,
$$
\begin{tikzcd}[column sep=small]
 TTT \ar[Rightarrow]{rr}{\lambda\lambda\lambda} \ar[Rightarrow]{dd}[swap]{\mu T} && T'T'T' \ar[Rightarrow]{dd}{\mu'T'} \ar[Rightarrow]{dr}{T'\mu'} 
 \ar[triple, shorten <=2em, shorten >=2em]{ddll}{\widehat{mT'\mu}} \\
 &&& T'T' \ar[Rightarrow]{dd}{\mu'}
 \ar[triple, shorten <=0.5em, shorten >=0.5em]{dl}{a'} \\
 TT \ar[Rightarrow]{dr}[swap]{\mu} \ar[Rightarrow]{rr}{\lambda\lambda} && T'T' \ar[Rightarrow]{dr}{\mu'} 
 \ar[triple, shorten <=0.5em, shorten >=0.5em]{dl}{m} \\
 & T \ar[Rightarrow]{rr}[swap]{\lambda} && T'
\end{tikzcd}
\qquad=\qquad
\begin{tikzcd}[column sep=small]
 TTT \ar[Rightarrow]{dr}[near end]{T\mu} \ar[Rightarrow]{rr}{\lambda\lambda\lambda} \ar[Rightarrow]{dd}[swap]{\mu T} && T'T'T' \ar[Rightarrow]{dr}{T'\mu'}
 \ar[triple, shorten <=0.5em, shorten >=0.5em]{dl}{\widehat{Tm\lambda}} \\
 & TT \ar[Rightarrow]{dd}{\mu} \ar[Rightarrow]{rr}[swap]{\lambda\lambda} 
 \ar[triple, shorten <=0.5em, shorten >=0.5em]{dl}{a} && T'T' \ar[Rightarrow]{dd}{\mu'} 
 \ar[triple, shorten <=2em, shorten >=2em]{ddll}{m} \\
 TT \ar[Rightarrow]{dr}[swap]{\mu} \\
 & T \ar[Rightarrow]{rr}[swap]{\lambda} && T'
\end{tikzcd}
$$
where $\widehat{mT'\mu}$ and $\widehat{Tm\lambda}$ denote the following compositions, respectively,
$$
\begin{tikzcd}[row sep=large]
 TTT \nat{dd}[swap]{\mu T} \nat{r}{TT\lambda} & TTT' \nat{r}{T\lambda T'} \nat{dd}{\mu T'} 
 \ar[triple, shorten <=2em, shorten >=2em]{ddl}{\mu_\lambda}
 & TT'T' \nat{r}{\lambda T'T'} & T'T'T' \nat{dd}{\mu'T'} 
 \ar[triple, shorten <=3em, shorten >=3em]{ddll}{m T'} \\ \\
 TT \nat{r}[swap]{T\lambda} & TT' \nat{rr}[swap]{\lambda T'} && T'T'
\end{tikzcd}
$$
and
$$
\begin{tikzcd}[row sep=large]
 TTT \nat{dd}[swap]{T\mu} \nat{r}{TT\lambda} 
 \ar[Rightarrow, bend left=20, ""{name=TCUP, above, near end}, ""{name=TCDOWN, below}]{ddrr}[pos=0.3, inner sep=0ex]{T(\mu'\circ\lambda T'\circ T\lambda)}
 \ar[Rightarrow, bend right=20, ""{name=TDUP, above}, ""{name=TDDOWN, below, near start}]{ddrr}[swap, pos=0.7, inner sep=0ex]{T(T\lambda\circ\mu)}
 & TTT' \nat{r}{T\lambda T'} & |[alias=TR]| TT'T' \nat{dd}{T\mu'} \nat{r}{\lambda T'T'} & T'T'T' \nat{dd}{T'\mu'}
 \ar[triple, shorten <=2em, shorten >=2em]{ddl}{\lambda_{\mu'}^{-1}} \\ \\
 |[alias=BL]| TT \nat{rr}[swap]{T\lambda} && TT' \nat{r}[swap]{\lambda T'} & T'T'
 \ar[triple, from=TR, to=TCUP, "\kappa\kappa", shorten <=0.5em, shorten >=0.5em]
 \ar[triple, from=TCDOWN, to=TDUP, "Tm", shorten <=0.3em, shorten >=0.3em]
 \ar[triple, from=TDDOWN, to=BL, "\kappa^{-1}", shorten <=0.5em, shorten >=0.5em]
\end{tikzcd}
$$

Finally, 2-cells of pseudomonads are defined as follows.

\begin{deph}\label{monad2cell}
 Let $\cat{K}$ be a strict 2-category. Let $(T,\mu,\eta,\ell,r,a)$ and $(T',\mu',\eta',\ell',r',a')$ be pseudomonads on $\cat{K}$. Let $(\lambda, u, m)$ and $(\xi,v,n)$ be (pseudo)morphisms $T\to T'$. A \emph{2-cell of monads} is a modification 
 $$
 \begin{tikzcd}
  T \ar[Rightarrow, bend left, ""{name=L,below}]{rr}{\lambda} \ar[Rightarrow, bend right, ""{name=LP,above}]{rr}[swap]{\xi} && T'
  \ar[triple, from=L, to=LP, "t"]
 \end{tikzcd}
 $$
 such that the following diagrams commute,
 $$
 \begin{tikzcd}[row sep=small]
  & \lambda\circ\eta \ar[triple]{dd}{t\,\eta} \\
  \eta' \ar[triple]{ur}{u} \ar[triple]{dr}[swap]{v} \\
  & \xi\circ \eta
 \end{tikzcd}
 \qquad
 \begin{tikzcd}[row sep=large]
  \mu'\circ (\lambda\lambda) \ar[triple]{r}{m} \ar[triple]{d}{\mu\,(tt)} & \lambda\circ\mu \ar[triple]{d}{t\,\mu} \\
  \mu'\circ (\xi\xi) \ar[triple]{r}{n} & \xi\circ\mu
 \end{tikzcd}
 $$
 which we call ``unit'' and ``multiplication'' conditions, respectively.
\end{deph}

As above, it may be helpful to write the conditions 2-dimensionally. The unit condition forms the following commutative ``cone'',
$$
  \begin{tikzcd}[row sep=large, column sep=small]
   & \id_\cat{K} \ar[Rightarrow, ""{name=ETAP, pos=0.3}]{dr}{\eta'} \ar[Rightarrow]{dl}[swap]{\eta} & \\
   |[alias=T]| T \ar[Rightarrow, bend left, ""{name=L,below}]{rr}[pos=0.56]{\lambda} \ar[Rightarrow, bend right, ""{name=LP,above}]{rr}[swap]{\xi} && T'
  \ar[triple, from=L, to=LP, "t"]
   \ar[triple, from=ETAP, to=T, "u", shorten <= 1.2em, shorten >= 2em, swap, near start, shift right=1]
  \end{tikzcd}
  \qquad=\qquad
  \begin{tikzcd}[row sep=large, column sep=small]
   & \id_\cat{K} \ar[Rightarrow, ""{name=ETAP}]{dr}{\eta'} \ar[Rightarrow]{dl}[swap]{\eta} & \\
   |[alias=T]| T \ar[Rightarrow, bend right]{rr}[swap]{\xi}  && T'
   \ar[triple, from=ETAP, to=T, "v", shorten <= 1.2em, shorten >= 1.5em, shift left]
  \end{tikzcd}
$$
and the multiplication condition forms the following commutative ``cylinder'',
$$
  \begin{tikzcd}[sep=large]
   TT \ar[Rightarrow, bend left]{r}{\lambda\lambda} \ar[Rightarrow]{d}[swap]{\mu} & |[alias=TPTP]| T'T' \ar[Rightarrow]{d}{\mu'} \\
  |[alias=T]|  T \ar[Rightarrow, bend left, ""{name=L,below}]{r}{\lambda} \ar[Rightarrow, bend right, ""{name=LP,above}]{r}[swap]{\xi}  & T'
   \ar[triple, from=TPTP, to=T, "m", shorten <= 1.5em, shorten >= 2em, swap, shift right=4, pos=0.4]
   \ar[triple, from=L, to=LP, "t"]
  \end{tikzcd}
  \qquad=\qquad
  \begin{tikzcd}[sep=large]
   TT \ar[Rightarrow, bend left, ""{name=L,below, pos=0.54}]{r}{\lambda\lambda} \ar[Rightarrow, bend right, ""{name=LP,above}]{r}[swap]{\xi\xi} \ar[Rightarrow]{d}[swap]{\mu} & |[alias=TPTP]| T'T' \ar[Rightarrow]{d}{\mu'} \\
  |[alias=T]| T  \ar[Rightarrow, bend right]{r}[swap]{\xi} & T' 
   \ar[triple, from=L, to=LP, "tt"]
   \ar[triple, from=TPTP, to=T, "n", shorten <= 2em, shorten >= 1.5em, shift left=4]
  \end{tikzcd}
  $$
In line with our usual convention for horizontal composition, $tt$ denotes the following 3-cell.
$$
\begin{tikzcd}[sep=large]
    TT \ar[Rightarrow, bend left, ""{name=TL, below}]{r}{T\lambda}
    \ar[Rightarrow, bend right, ""{name=BL, above}]{r}[swap]{T\lambda}
    & TT' \ar[Rightarrow, bend left, ""{name=TR, below}]{r}{\lambda T'} 
    \ar[Rightarrow, bend right, ""{name=BR, above}]{r}[swap]{\lambda T'}
    & T'T' 
    \ar[triple, from=TL, to=BL, "Tt"]
    \ar[triple, from=TR, to=BR, "t T'"]
\end{tikzcd}
$$

\subsection{Pseudoalgebras and their morphisms}\label{secpseudoalg}

\begin{deph}\label{defpseudoalgebra}
 Let $(T,\mu,\eta,\ell,r,a)$ be a pseudomonad on $\cat{K}$. A \emph{pseudoalgebra} over the pseudomonad $T$ consists of
 \begin{itemize}
  \item an object $A$ of $\cat{K}$, together with
  \item a morphism $e:TA\to A$, and
  \item invertible 2-cells
  $$
  \begin{tikzcd}[sep=large]
   A \ar[""{name=ID}]{dr}[swap]{\id} \ar{r}{\eta_A} & TA 
   \ar[Rightarrow, to=ID, "\iota", shorten >= 0.1em, shorten <= 0.3em] \ar{d}{e} \\
   & A
  \end{tikzcd}
  \qquad\mbox{and}\qquad
  \begin{tikzcd}[sep=large]
   TTA \ar{d}[swap]{\mu_A} \ar{r}{Te} & TA \ar{d}{e}
   \ar[Rightarrow, shorten >= 1.2em, shorten <= 1.2em]{dl}{\gamma}\\
   TA \ar{r}[swap]{e} & A
  \end{tikzcd}
  $$
  such that
  \item the following diagrams commute,
  $$
  \begin{tikzcd}[column sep=small]
   e \circ Te \circ T\eta_A  \ar[Rightarrow]{rr}{\gamma \, (T\eta_A)} \ar[Rightarrow,swap]{dr}{e \, (\widehat{T\iota})} && e \circ\mu_A \circ T\eta_A  \\ 
   & e \circ \id_{TA} \ar[Rightarrow]{ur}[swap]{e \, r_A}
  \end{tikzcd}
  $$
  $$
  \begin{tikzcd}[column sep=tiny]
   & e\circ Te \circ TTe \ar[Rightarrow]{dl}[swap]{\gamma\,(TTe)} \ar[Rightarrow]{dr}{e\,(\widehat{T\gamma})}\\
   e\circ\mu_A\circ TTe \ar[Rightarrow]{dd}[swap]{e\,\mu_e} && e\circ Te\circ T\mu_A \ar[Rightarrow]{dd}{\gamma\,(T\mu_A)} \\ \\
   e\circ Te\circ\mu_{TA} \ar[Rightarrow]{dr}[swap]{\gamma\,\mu_A} && e\circ\mu_A\circ T\mu_A \ar[Rightarrow]{dl}{e\,a_A} \\
   & e\circ\mu\circ \mu_{TA}
  \end{tikzcd}
  $$
  which we call unit and multiplication conditions, respectively, where $\widehat{T\iota} = \upsilon^{-1}\circ T\iota\circ\kappa$, and $\widehat{T\gamma}=\kappa^{-1}\circ T\gamma\circ \kappa$.
 \end{itemize}
 We call the pseudoalgebra \emph{normal} if the 2-cell $\iota$ is the identity.
\end{deph}

We can write the coherence diagrams 2-dimensionally as well, as follows. Here is the unit condition,
$$
\begin{tikzcd}[column sep=small]
 TA \ar{rr}{T\eta} \ar{dd}[swap]{T\eta} \ar{ddrr}[near end]{\id} && |[alias=TTTUR]| TTA \ar{dr}{Te} \ar{dd}{Te} \\
 & |[alias=ID]| && TA \ar{dd}{e} \\
 |[alias=TTTDL]| TTA \ar{dr}[swap]{\mu} \ar{rr}[swap]{\mu} && TA \ar{dr}{e} \\
 & TA \ar{rr}[swap]{e} && A
 \ar[Rightarrow, from=TTTUR, to=ID, "\widehat{T\iota}", shorten <= 0.5em, shorten >= 0.5em, swap]
 \ar[Rightarrow, from=ID, to=TTTDL, "r", shorten <= 0.5em, shorten >= 0.5em, swap]
\end{tikzcd}
\qquad=\qquad
\begin{tikzcd}[column sep=small]
 TA \ar{rr}{T\eta} \ar{dd}[swap]{T\eta} \ar{dr}[swap, near end]{T\eta} && TTA \ar{dr}{Te} \\
 & TTA \ar{rr}{Te} \ar{dd}{\mu} && |[alias=TTUR]| TA \ar{dd}{e} \\
 TTA \ar{dr}[swap]{\mu} \\
 & |[alias=TTDL]|  TA \ar{rr}[swap]{e} && A
    \ar[Rightarrow, from=TTUR, to=TTDL, "\gamma", shorten <= 2em, shorten >= 2em]
\end{tikzcd}
$$
where the unfilled 2-cells are identities, and $\widehat{T\iota}$ denotes the following composite cell,
 $$
 \begin{tikzcd}[sep=large]
  & |[alias=TTT]| TTA \ar{ddr}{Te} \\ \\
  TA \ar{uur}{T\eta} 
  \ar[bend left=35, ""{name=COMP}]{rr}[swap, pos=0.1, yshift=0.5ex]{T(e\circ \eta)}
  \ar[""{name=TID}]{rr}[near start, swap]{T\id}
  \ar[bend right=35, ""{name=ID}]{rr}[swap]{\id} 
  && TA
  \ar[Rightarrow, from=TTT, to=COMP, "\kappa", shorten <= 0.5em, shorten >= 0em]
  \ar[Rightarrow, from=COMP, to=TID, "T\iota", shorten <= 0.5em, shorten >= 0em]
  \ar[Rightarrow, from=TID, to=ID, "\upsilon^{-1}", shorten <= 0.5em, shorten >= 0em]
 \end{tikzcd}
 $$
and here is the multiplication condition,
$$
   \begin{tikzcd}[column sep=small]
    TTTA \ar{rr}{TTe} \ar{dd}[swap]{\mu} && |[alias=TTTUR]| TTA \ar{dr}{Te} \ar{dd}{\mu} \\
    &&& |[alias=TTUR]| TA \ar{dd}{e} \\ 
    |[alias=TTTDL]| TTA \ar{rr}{Te} \ar{dr}[swap]{\mu} && |[alias=TT]| TA \ar{dr}{e} \\
    & |[alias=TTDL]| TA \ar{rr}[swap]{e} && A
    \ar[Rightarrow, from=TTUR, to=TT, "\gamma", shorten <= 0.5em, shorten >= 0.5em]
    \ar[Rightarrow, from=TT, to=TTDL, "\gamma", shorten <= 0.5em, shorten >= 0.5em]
    \ar[Rightarrow, from=TTTUR, to=TTTDL, "\mu_e", shorten <= 2em, shorten >= 2em]
   \end{tikzcd}
      \qquad=\qquad
   \begin{tikzcd}[column sep=small]
    TTTA \ar{rr}{TTe} \ar{dd}[swap]{\mu} \ar{dr}[swap]{T\mu} && |[alias=TTTUR]| TTA \ar{dr}{Te} \\
    & |[alias=TTT]| TTA \ar{rr}{Te} \ar{dd}[swap]{\mu} && |[alias=TTUR]| TA \ar{dd}{e} \\ 
    |[alias=TTTDL]| TTA \ar{dr}[swap]{\mu} \\
    & |[alias=TTDL]| TA \ar{rr}[swap]{e} && A
    \ar[Rightarrow, from=TTTUR, to=TTT, "\widehat{T\gamma}", shorten <= 0.5em, shorten >= 0.5em, swap]
    \ar[Rightarrow, from=TTT, to=TTTDL, "a", shorten <= 0.5em, shorten >= 0.5em, swap]
   \ar[Rightarrow, from=TTUR, to=TTDL, "\gamma", shorten <= 2em, shorten >= 2em]
   \end{tikzcd}
  $$
  where $\widehat{T\gamma}$ denotes the following composite cell.
 $$
 \begin{tikzcd}[sep=large]
 & |[alias=TOP]| TTA \ar{dr}{Te} \\
 TTTA \ar{ur}{TTe} \ar{dr}[swap]{T\mu}
 \ar[bend left=20, ""{name=COMPUP}]{rr}[swap, pos=0.1, yshift=0.4ex]{T(e\circ Te)}
 \ar[bend right=20, ""{name=COMPDOWN}]{rr}[pos=0.1, yshift=-0.6ex]{T(e\circ\mu)}
 && TA \\
 & |[alias=BOTTOM]| TTA \ar{ur}[swap]{Te}
 \ar[Rightarrow, from=TOP, to=COMPUP, "\kappa"]
 \ar[Rightarrow, from=COMPUP, to=COMPDOWN, "T\gamma", shorten <=0.5em]
 \ar[Rightarrow, from=COMPDOWN, to=BOTTOM, "\kappa^{-1}", shorten <=0.5em]
 \end{tikzcd}
 $$

\begin{deph}\label{defmorphalg}
 Let $(A,e_A,\iota_A,\gamma_A)$ and $(B,e_B,\iota_B,\gamma_B)$ be pseudoalgebras over the pseudomonad $(T,\mu,\eta,\ell,r,a)$. A \emph{(strong) morphism of pseudoalgebras} consists of
 \begin{itemize}
  \item a morphism $f:A\to B$ of $\cat{K}$, together with 
  \item an invertible 2-cell 
  $$
  \begin{tikzcd}
   TA \ar{d}[swap]{e_A} \ar{r}{Tf} & TB \ar{d}{e_B}
   \ar[Rightarrow, shorten <= 1em, shorten >= 1em]{dl}{\phi} \\
   A \ar{r}[swap]{f} & B
  \end{tikzcd}
  $$
  such that
  \item the following diagrams commute,
  $$
  \begin{tikzcd}
  e_B\circ\eta_B\circ f \ar{d}{\iota_B} \ar{r}{e_B\,\eta_f} & e_B\circ Tf\circ\eta_A \ar{d}{\phi\,\eta_A} \\ 
  f & f\circ e_A\circ\eta_A \ar{l}{f\,\iota_A}
  \end{tikzcd}
  \qquad
  \begin{tikzcd}[row sep=small, column sep=-1em]
   & e_B\circ Te_B\circ TTf \ar{dr}[near end]{e_B\,(\widehat{T\phi})} \ar{dl}[swap, near end]{\gamma_B\,(TTf)} \\
   e_B\circ\mu_B\circ TTf \ar{dd}[swap]{e_B\,\mu_f} && e_B\circ Tf\circ Te_A \ar{dd}{\phi\,(Te_A)} \\ \\
   e_B\circ Tf \circ \mu_A \ar{dr}[swap]{\phi\,\mu_A} && f\circ e_A\circ Te_A \ar{dl}{f\,\gamma_A} \\
   & f\circ e_A\circ\mu_A
  \end{tikzcd}
  $$
  which we call the unit and multiplication condition, respectively, where $\widehat{T\phi}=\kappa^{-1}\circ T\phi\circ \kappa$.
 \end{itemize}
\end{deph}
As above, it may be helpful to draw the coherence conditions in a 2-dimensional way. The unit condition is the following commutative prism,
$$
\begin{tikzcd}[sep=large]
 & B \ar{r}{\eta} \ar[""{name=ID, above}]{dr}[swap]{\id} & TB \ar{d}{e_B}
 \ar[Rightarrow, to=ID, "\iota_B", shorten <=0.3em, shorten >=0.3em] \\
 A \ar{ur}{f}  \ar{dr}[swap]{\id} && B \\
 & A \ar{ur}[swap]{f}
\end{tikzcd}
\qquad=\qquad
\begin{tikzcd}[sep=large]
 & B \ar{r}{\eta}  
 \ar[Rightarrow, shorten <=0.7em, shorten >=0.7em]{d}{\eta_f}
 & TB \ar{d}{e_B}
 \ar[Rightarrow, shorten <=3em, shorten >=3em]{ddl}{\phi}
  \\
 A \ar{ur}{f} \ar{r}{\eta} \ar[""{name=ID, above}]{dr}[swap]{\id} & TA \ar{ur}{Tf} \ar{d}{e_A}
 \ar[Rightarrow, to=ID, "\iota_A", shorten <=0.3em, shorten >=0.3em]
 & B \\
 & A \ar{ur}[swap]{f}
\end{tikzcd}
$$
where again the unfilled 2-cell is an identity, and the multiplication condition is the following commutative cube,
$$
\begin{tikzcd}[column sep=small]
 TTA \ar{dd}[swap]{\mu} \ar{rr}{TTf} && TTB \ar{dd}{\mu} \ar{dr}{Te_B} 
 \ar[Rightarrow, shorten <=2em, shorten >=2em]{ddll}{\mu_f} 
 \\
 &&& TB \ar{dd}{e_B} 
 \ar[Rightarrow, shorten <=0.5em, shorten >=0.5em]{dl}{\gamma_B}
 \\
 TA \ar{rr}{Tf} \ar{dr}[swap]{e_A} && TB \ar{dr}{e_B} 
 \ar[Rightarrow, shorten <=0.5em, shorten >=0.5em]{dl}{\phi} \\
 & A \ar{rr}[swap]{f} && B
\end{tikzcd}
\qquad=\qquad
\begin{tikzcd}[column sep=small]
 TTA \ar{dr}[swap]{Te_A} \ar{dd}[swap]{\mu} \ar{rr}{TTf} && TTB \ar{dr}{Te_B} \ar[Rightarrow, shorten <=0.5em, shorten >=0.5em]{dl}{\widehat{T\phi}}  \\
 & TA \ar{rr}[swap]{Tf} \ar{dd}{e_A} 
 \ar[Rightarrow, shorten <=0.5em, shorten >=0.5em]{dl}{\gamma_A}
 && TB \ar{dd}{e_B}
 \ar[Rightarrow, shorten <=2em, shorten >=2em]{ddll}{\phi} \\
 TA \ar{dr}[swap]{e_A}  \\
 & A \ar{rr}[swap]{f} && B
\end{tikzcd}
$$
where $\widehat{T\phi}$ denotes the following composite cell.
 $$
 \begin{tikzcd}[sep=large]
 & |[alias=TOP]| TTB \ar{dr}{Te_B} \\
 TTA \ar{ur}{TTf} \ar{dr}[swap]{Te_A}
 \ar[bend left=20, ""{name=COMPUP}]{rr}[swap, pos=0.1, yshift=0.4ex]{T(e_B\circ Tf)}
 \ar[bend right=20, ""{name=COMPDOWN}]{rr}[pos=0.1, yshift=-0.6ex]{T(f\circ e_A)}
 && TB \\
 & |[alias=BOTTOM]| TTA \ar{ur}[swap]{Tf}
 \ar[Rightarrow, from=TOP, to=COMPUP, "\kappa"]
 \ar[Rightarrow, from=COMPUP, to=COMPDOWN, "T\phi", shorten <=0.5em]
 \ar[Rightarrow, from=COMPDOWN, to=BOTTOM, "\kappa^{-1}", shorten <=0.5em]
 \end{tikzcd}
 $$

\begin{deph}\label{alg2cell}
 Let $(A,e_A,\iota_A,\gamma_A)$ and $(B,e_B,\iota_B,\gamma_B)$ be pseudoalgebras over the pseudomonad $(T,\mu,\eta,\ell,r,a)$. Let $(f,\phi)$ and $(g,\chi)$ be strong morphisms of pseudoalgebras. A \emph{2-cell of pseudoalgebras} is a 2-cell
 $$
 \begin{tikzcd}[sep=large]
  A \ar[bend left, ""{name=F, below}]{r}{f} \ar[bend right, ""{name=FP, above}]{r}[swap]{g} & B
  \ar[Rightarrow, from=F, to=FP, "\alpha"]
 \end{tikzcd}
 $$
 such that the following diagram commutes.
 $$
 \begin{tikzcd}
  e_B\circ Tf \nat{r}{\phi} \nat{d}{e_B\,T\alpha} & f\circ e_A \nat{d}{\alpha\,e_A} \\
  e_B\circ Tg \nat{r}{\chi} & g\circ e_A
 \end{tikzcd}
 $$
\end{deph}

If we draw the condition in a 2-dimensional way, we get the following commutative ``cylinder'':
$$
 \begin{tikzcd}[sep=large]
  TA \ar[bend left]{r}{Tf} \ar{d}[swap]{e_A} & TB  \ar{d}{e_B} \ar[Rightarrow, shorten <= 1em, shorten >= 2em, shift right=4, pos=0.4]{dl}[swap]{\phi}
   \\
  A \ar[bend left, ""{name=F, below}]{r}{f} \ar[bend right, ""{name=FP, above}]{r}[swap]{g} & B
  \ar[Rightarrow, from=F, to=FP, "\alpha"]
 \end{tikzcd}
 \qquad=\qquad
 \begin{tikzcd}[sep=large]
  TA \ar[bend left, ""{name=F, below, pos=0.52}]{r}{Tf} \ar[bend right, ""{name=FP, above}]{r}[swap]{Tg} \ar{d}[swap]{e_A} & TB \ar{d}{e_B}
  \ar[Rightarrow, from=F, to=FP, "T\alpha"] \ar[Rightarrow, shorten <= 2em, shorten >= 1em, shift left=4, pos=0.55]{dl}{\chi} \\
  A  \ar[bend right]{r}[swap]{g} & B
 \end{tikzcd}
$$

It is immediate from the definitions to check that, given any object $X$ of $\cat{K}$, the object $TX$ is canonically a pseudoalgebra, with structure morphism $\mu_X:TTX\to TX$ and 2-cells $\ell_X:\mu_X\circ\eta_{TX}\Rightarrow \id_{TX}$ and $a_X:\mu_X\circ T\mu_X\Rightarrow \mu_X\circ\mu_{TX}$. We call this a ``free algebra'', analogously to the 1-dimensional case.
Moreover, by naturality of $\mu$, for any morphism $f:X\to Y$ of $\cat{K}$, the morphism $Tf:TX\to TY$ gives a morphism of pseudoalgebras. Similarly, for every $f,g:X\to Y$ and $\alpha:f\Rightarrow g$, the 2-cell $T\alpha:Tf\Rightarrow Tg$ is automatically a 2-cell of algebras.

\subsection{Restriction of scalars for pseudomonads}\label{restrictionofscalars}

It is very well known that in the one-dimensional context, a morphism of monads induces a pullback functor between the algebras, sometimes named ``restriction of scalars'' after its instance in ring theory \cite[Theorem~3 in Section~3.6]{ttt}.
A similar phenomenon occurs in two dimensions, as follows. We use this statement in \Cref{pbfunctor}, see there also for further context.

\begin{thm}\label{rescalthm}
 Let $\cat{K}$ be a (strict) 2-category. Let $(T,\mu,\eta,\ell,r,a)$ and $(T',\mu',\eta',\ell',r',a')$ be pseudomonads on $\cat{K}$, and let $(\lambda, u, m)$ be a pseudomorphism of monads from $T$ to $T'$. 
 Each $T'$-pseudoalgebra $(A,e',\iota',\gamma')$, defines canonically a $T$-pseudoalgebra structure on $A$ with the following structure 1- and 2-cells.
 \begin{equation}\label{drescal}
 e\coloneqq 
 \begin{tikzcd}
  TA \ar{d}{\lambda_A} \\
  T'A \ar{d}{e'} \\
  A
 \end{tikzcd}
 \qquad
 \iota\coloneqq
 \begin{tikzcd}[column sep=large]
  & |[alias=TA]| TA \ar{d}{\lambda_A} \\
  A \ar[""{name=ID,above}]{dr}[swap]{\id} \ar[""{name=ETAP,above}]{r}[swap]{\eta'_A} \ar{ur}{\eta_A} & |[alias=TPA]| T'A \ar{d}{e'} \\
  & A
  \ar[Rightarrow, from=TA, to=ETAP, "u^{-1}_A", shorten <= 0.8em, shorten >= 0.5em, outer sep=-2pt]
  \ar[Rightarrow, from=TPA, to=ID, "\iota'", shorten <= 0.2em, shorten >= 0.2em]
 \end{tikzcd}
 \qquad
 \gamma\coloneqq
 \begin{tikzcd}
  TTA \ar{r}{T\lambda_A} \ar{dd}[swap]{\mu_A} & TT'A \ar{d}[swap]{\lambda_{T'A}} \ar{r}{Te'}
  \ar[Rightarrow, shorten <= 2em, shorten >= 2em, swap]{ddl}{m} & TA \ar{d}{\lambda_A} 
  \ar[Rightarrow, shorten <= 0.5em, shorten >= 0.5em, swap]{dl}{\lambda_{e'}}
  \\
  & T'T'A \ar{d}[swap]{\mu'_A} \ar{r}{T'e'} & T'A \ar{d}{e'}
  \ar[Rightarrow, shorten <= 0.5em, shorten >=0.5em ]{dl}{\gamma'}
  \\
  TA \ar{r}[swap]{\lambda_A} & T'A \ar{r}[swap]{e'} & A
 \end{tikzcd}
 \end{equation}
 Moreover, this construction defines a 2-functor between the categories of pseudoalgebras $\lambda^*:\cat{K}^{T'}\to\cat{K}^T$.
\end{thm}

In analogy with the 1-dimensional case, we call $\lambda^*$ the \emph{restriction of scalar} functor.

\begin{proof}
The unit diagram for $(A,e,\iota,\gamma)$, obtained by plugging \eqref{drescal} into the unit diagram of \Cref{defpseudoalgebra}, can be decomposed in the following way, where the whiskerings have been suppressed for reasons of space.
$$
\begin{tikzpicture}[baseline= (a).base]
 \node[scale=0.55] (a) at (0,0){
\begin{tikzcd}[column sep=tiny, row sep=large]
    {e'\circ\lambda_A\circ Te'\circ T\lambda_A\circ T\eta_A} &&[+10pt]& {e'\circ T'e'\circ(\lambda\lambda)_A\circ T\lambda_A} &&[-20pt]& {e'\circ \mu'_A\circ (\lambda\lambda)_A\circ T\lambda_A} &&& {e'\circ\lambda_A\circ\mu_A\circ T\eta_A} &&&\\
	&& {e'\circ\lambda_A\circ Te'\circ T\eta'_A} &&  {e'\circ T'e'\circ\lambda_{T'A}\circ T\eta'_A} && {e'\circ \mu'_A\circ\lambda_{T'A}\circ T\eta'_A} \\
	&&&&& {e'\circ T'e'\circ T'\eta'_A\circ T\lambda_A} & {e'\circ\mu'_A\circ T\eta'_A\circ\lambda_A} \\
	\\
	&&&&&& {e'\circ\lambda_A} 
	\arrow[Rightarrow, "{\gamma'}", from=1-4, to=1-7]
	\arrow[Rightarrow, "{\lambda_{e'}}", from=1-1, to=1-4]
	\arrow[Rightarrow, "{m_A}", from=1-7, to=1-10]
	\arrow[Rightarrow, "{Tu^{-1}_A\circ\kappa}"', from=1-1, to=2-3]
	\arrow[Rightarrow, "{Tu^{-1}_A\circ\kappa}", from=1-4, to=2-5, near end]
	\arrow[Rightarrow, "{\lambda_{e'}}", from=2-3, to=2-5]
	\arrow[Rightarrow, "{\lambda_{\eta'}}", from=2-5, to=3-6, near end]
	\arrow[Rightarrow, "{\gamma'}", from=2-5, to=2-7]
	\arrow[Rightarrow, "{Tu^{-1}_A\circ\kappa}", from=1-7, to=2-7]
	\arrow[Rightarrow, "{\lambda_{\eta'}}", from=2-7, to=3-7]
	\arrow[Rightarrow, "{\gamma'}", from=3-6, to=3-7]
	\arrow[Rightarrow, "{r'_A}"', from=5-7, to=3-7]
	\arrow[Rightarrow, "{\widehat{T'\iota'_A}}", from=3-6, to=5-7]
	\arrow[Rightarrow, "{\widehat{T\iota'_A}}"', from=2-3, to=5-7]
	\arrow[Rightarrow, "{r_A}"', from=5-7, to=1-10]
\end{tikzcd}
};
\end{tikzpicture}
$$
Now,
\begin{itemize}
 \item the region on the right commutes by the right unit condition of $(\lambda,u,m)$, as in \Cref{defmorphmonads} (recall that all the arrows of the diagram are invertible);
 \item the triangle in the center bottom is the right unitality condition of the algebra $(A,e',\iota',\gamma')$, as in \Cref{defpseudoalgebra};
 \item the region on the bottom left commutes by pseudonaturality of $\lambda$;
 \item finally, the remaining parallelograms commute by the interchange law. 
\end{itemize}

The multiplication diagram for $(A,e,\iota,\gamma)$, analogously obtained by plugging \eqref{drescal} into the multiplication diagram of \Cref{defpseudoalgebra}, can be decomposed as follows, where again the whiskerings have been suppressed, and the hat denotes the suitable application of the compositors.
$$
\begin{tikzpicture}[baseline= (a).base]
 \node[scale=.58] (a) at (0,0){
 \begin{tikzcd}[column sep=-6em]
	&&& {e'\circ\lambda_A\circ Te'\circ T\lambda_A\circ TTe'\circ TT\lambda_A} \\
	&& {e'\circ Te'\circ(\lambda\lambda)_A\circ TTe'\circ TT\lambda_A} && {e'\circ\lambda_A\circ Te'\circ TTe'\circ T\lambda_{T'A}\circ TT\lambda_A} \\
	& {e'\circ \mu'_A\circ(\lambda\lambda)_A\circ TTe'\circ TT\lambda_A} && {e'\circ Te'\circ \lambda_{T'A}\circ TTe'\circ T\lambda_{T'A}\circ TT\lambda_A} && {e'\circ\lambda_A\circ Te'\circ T\mu'_A\circ T\lambda_{T'A}\circ TT\lambda_A} \\
	{e'\circ\lambda_A\circ\mu_A\circ TTe'\circ TT\lambda_A} && {e'\circ \mu'_A\circ \lambda_{T'A}\circ TTe'\circ T\lambda_{T'A}\circ TT\lambda_A} &&&& {e'\circ\lambda_A\circ Te'\circ T\lambda_A\circ T\mu_A} \\
	&&& {e'\circ Te'\circ T'T'e\circ (\lambda\lambda\lambda)_A} && {e'\circ T'e'\circ\lambda_{T'A}\circ T\mu'_A\circ T\lambda_{T'A}\circ TT\lambda_A} \\
	&& {e'\circ\mu'_A\circ T'T'e\circ (\lambda\lambda\lambda)_A} && {e'\circ T'e'\circ T'\mu'_A\circ (\lambda\lambda\lambda)_A} && {e'\circ T'e'\circ(\lambda\lambda)_A\circ T\mu_A} \\
	&&&&& {e'\circ\mu'_A\circ\lambda_{T'A}\circ T\mu'_A\circ T\lambda_{T'A}\circ TT\lambda_A} \\
	{e'\circ\lambda_A\circ Te'\circ\mu_{T'A}\circ TT\lambda_A} && {e'\circ T'e'\circ\mu'_{T'A}\circ (\lambda\lambda\lambda)_A} && {e'\circ \mu'_A\circ T'\mu'_A\circ (\lambda\lambda\lambda)_A} && {e'\circ \mu'_A\circ(\lambda\lambda)_A\circ T\mu_A} \\
	& {e'\circ T'e'\circ(\lambda\lambda)_A\circ\mu_{TA}} && {e'\circ \mu'_A\circ\mu'_{T'A}\circ (\lambda\lambda\lambda)_A} \\
	{e'\circ\lambda_A\circ Te'\circ T\lambda_a\circ\mu_{TA}} && {e'\circ \mu'_A\circ(\lambda\lambda)_A\circ\mu_{TA}} &&&& {e'\circ \lambda_A\circ\mu_A\circ T\mu_A} \\
	& {e'\circ T'e'\circ(\lambda\lambda)_A\circ\mu_{TA}} && {} \\
	&& {e'\circ \mu'_A\circ(\lambda\lambda)_A\circ\mu_{TA}} \\
	&&& {e'\circ\lambda_A\circ\mu_A\circ\mu_{TA}}
	\arrow[Rightarrow, "{\lambda_{e'}}"', from=1-4, to=2-3, near end]
	\arrow[Rightarrow, "{\gamma'}"', from=2-3, to=3-2, near end]
	\arrow[Rightarrow, "{m_A}"', from=3-2, to=4-1, near end]
	\arrow[Rightarrow, "{\widehat{T\lambda_{e'}}}", from=1-4, to=2-5, near end]
	\arrow[Rightarrow, "{\widehat{T\gamma'}}", from=2-5, to=3-6, near end]
	\arrow[Rightarrow, "{\widehat{Tm_A}}", from=3-6, to=4-7, near end]
	\arrow[Rightarrow, "{\widehat{T\lambda_{e'}}}", from=2-3, to=3-4, near end]
	\arrow[Rightarrow, "{\lambda_{e'}}"', from=2-5, to=3-4, near end]
	\arrow[Rightarrow, "{\widehat{T\lambda_{e'}}}", from=3-2, to=4-3, near end]
	\arrow[Rightarrow, "{m_A}"', from=12-3, to=13-4]
	\arrow[Rightarrow, "{\gamma'}"', from=11-2, to=12-3]
	\arrow[Rightarrow, "{\lambda_{e'}}"', from=10-1, to=11-2]
	\arrow[Rightarrow, "{\widehat{Tm_A}}", from=5-6, to=6-7, near end]
	\arrow[Rightarrow, "{\lambda_{e'}}"', from=3-6, to=5-6]
	\arrow[Rightarrow, "{\gamma'}"', from=5-6, to=7-6]
	\arrow[Rightarrow, "{\lambda_{\mu'}}", from=7-6, to=8-5]
	\arrow[Rightarrow, "{a'_A}", from=8-5, to=9-4]
	\arrow[Rightarrow, "{\gamma'}", from=8-3, to=9-4, near end]
	\arrow[Rightarrow, "{\lambda_{\mu'}}"', from=5-6, to=6-5, near end]
	\arrow[Rightarrow, "{\gamma'}"', from=6-5, to=8-5]
	\arrow[Rightarrow, "{\mu'_{e'}}"', from=6-3, to=8-3]
	\arrow[Rightarrow, "{\gamma'}"', from=5-4, to=6-3, near end]
	\arrow[Rightarrow, "{\lambda_{T'e'}}", from=3-4, to=5-4]
	\arrow[Rightarrow, "{\gamma'}"', from=3-4, to=4-3, near end]
	\arrow[Rightarrow, "{\lambda_{T'e'}}"', from=4-3, to=6-3]
	\arrow[Rightarrow, "{\widehat{T'\gamma'}}", from=5-4, to=6-5, near end]
	\arrow[Rightarrow, "{\widehat{Tm_A}}"', from=7-6, to=8-7]
	\arrow[Rightarrow, "{a_A}", from=10-7, to=13-4]
	\arrow[Rightarrow, "{\gamma'}", from=6-7, to=8-7]
	\arrow[Rightarrow, "{\lambda_{e'}}", from=4-7, to=6-7]
	\arrow[Rightarrow, "{m_A}", from=8-7, to=10-7]
	\arrow[Rightarrow, "{m_{T'A}}"', from=8-3, to=9-2, near end]
	\arrow[Rightarrow, "{\lambda_{e'}}", from=8-1, to=9-2, near end]
	\arrow[Rightarrow, "{\mu_\lambda}", from=9-2, to=11-2]
	\arrow[Rightarrow, "{m_{T'A}}", from=9-4, to=10-3]
	\arrow[Rightarrow, "{\mu_\lambda}", from=10-3, to=12-3]
	\arrow[Rightarrow, "{\gamma'}"', from=9-2, to=10-3]
	\arrow[Rightarrow, "{\mu_\lambda}"', from=8-1, to=10-1]
	\arrow[Rightarrow, "{\mu_{e'}}"', from=4-1, to=8-1]
\end{tikzcd}
};
\end{tikzpicture}
$$
Now,
\begin{itemize}
 \item the bottom right region commutes by the associativity condition of $(\lambda,u,m)$, as in \Cref{defmorphmonads};
 \item the top right hexagon commutes by pseudonaturality of $\lambda$;
 \item the center hexagon commutes by the multiplication condition of the algebra $(A,e',\iota',\gamma')$, as in \Cref{defpseudoalgebra};
 \item the top left hexagon commutes by the modification property for $m$;
 \item all the remaining parallelograms commute by the interchange law.
\end{itemize}

Therefore, $(A,e,\iota,\gamma)$ is a $T$-pseudoalgebra.

For functoriality, consider now $T'$-pseudoalgebras $(A,e'_A,\iota'_A,\gamma'_A)$ and $(B,e'_B,\iota'_B,\gamma'_B)$, and a morphism of $T'$-pseudoalgebras $(f,\phi)$ from $A$ to $B$. The composite 2-cell
$$
\begin{tikzcd}
 TA \ar{r}{Tf} \ar{d}[swap]{\lambda_A} & TB \ar{d}{\lambda_B} \ar[Rightarrow, "\lambda_f", shorten <= 1em, shorten >= 1em, swap]{dl} \\
 T'A \ar{r}{T'f} \ar{d}[swap]{e'_A} & T'B \ar{d}{e'_B} \ar[Rightarrow, "\phi", shorten <= 1em, shorten >= 1em]{dl} \\
 A \ar{r}[swap]{f} & B
\end{tikzcd}
$$
makes $f$ a morphism between the $T$-pseudoalgebra structures defined above. Indeed, the unit condition, obtained by plugging \eqref{drescal} into \Cref{defmorphalg}, can be decomposed as follows, again omitting the whiskering.
$$
\begin{tikzcd}[column sep=small]
	{e'_B\circ\lambda_B\circ\eta_B\circ f} && {e'_B\circ \lambda_B\circ Tf\circ\eta_A} \\
	{e'_B\circ\eta'_B\circ f} & {e'_B\circ T'f\circ\eta'_A} & {e'_B\circ T'f\circ\lambda_A\circ\eta_A} \\
	{f} & {f\circ e'_A\circ\eta'_A} & {f\circ e'_A\circ\lambda_A\circ\eta_A}
	\arrow[Rightarrow, "{u^{-1}_B}"', from=1-1, to=2-1]
	\arrow[Rightarrow, "{\iota'_B}"', from=2-1, to=3-1]
	\arrow[Rightarrow, "{\iota'_A}", from=3-2, to=3-1]
	\arrow[Rightarrow, "{u^{-1}_A}", from=3-3, to=3-2]
	\arrow[Rightarrow, "{\lambda_f}", from=1-3, to=2-3]
	\arrow[Rightarrow, "{\phi}", from=2-3, to=3-3]
	\arrow[Rightarrow, "{\eta_f}", from=1-1, to=1-3]
	\arrow[Rightarrow, "{u^{-1}_A}"', from=2-3, to=2-2]
	\arrow[Rightarrow, "{\phi}", from=2-2, to=3-2]
	\arrow[Rightarrow, "{\eta'_f}", from=2-1, to=2-2]
\end{tikzcd}
$$
Now
\begin{itemize}
 \item the top region commutes by the modification property for $u$;
 \item the bottom left rectangle commutes by the unit condition for $(f,\phi)$ as in \Cref{defmorphalg};
 \item the bottom right rectangle commutes by the interchange law.
\end{itemize}

Similarly, the multiplication condition can be decomposed as follows, where again the whiskerings are omitted.
$$
\begin{tikzpicture}[baseline= (a).base]
 \node[scale=.75] (a) at (0,0){
\begin{tikzcd}[column sep=-4em]
	&&& {e'_B\circ\lambda_B\circ Te'_B\circ T\lambda_B\circ TTf} \\
	&& {e'_B\circ T'e'_B\circ (\lambda\lambda)_B\circ TTf} && {e'_B\circ\lambda_B\circ Te'_B\circ TT'f\circ T\lambda_A} \\
	& {e'_B\circ\mu'_B\circ(\lambda\lambda)_B\circ TTf} && {e'_B\circ T'e'_B\circ\lambda_{T'B}\circ TT'f\circ T\lambda_A} && {e'_B\circ\lambda_B\circ Tf\circ Te'_A\circ T\lambda_A} \\
	{e'_B\circ \lambda_B\circ \mu_B\circ TTf} && {e'_B\circ\mu'_B\circ\lambda_{T'B}\circ TT'f\circ T\lambda_A} &&& {} \\
	&&& {e'_B\circ T'e'_B\circ T'T'f\circ(\lambda\lambda)_A} && {e'_B\circ T'f\circ\lambda_A\circ Te'_A\circ T\lambda_A} \\
	&& {e'_B\circ\mu'_B\circ T'T'f\circ(\lambda\lambda)_A} && {e'_B\circ Tf\circ T'e'_A\circ (\lambda\lambda)_A} \\
	&&&&& {f\circ e'_A\circ\lambda_A\circ Te'_A\circ T\lambda_A} \\
	{e'_B\circ\lambda_B\circ Tf} && {e'_B\circ T'f\circ \mu'_A\circ(\lambda\lambda)_A} && {f\circ e'_A\circ T'e'_A\circ(\lambda\lambda)_A} \\
	& {e'_B\circ T'f\circ \lambda_A\circ\mu_A} && {f\circ e'_A\circ\mu'_A\circ(\lambda\lambda)_A} \\
	&& {f\circ e'_A\circ\lambda_A\circ\mu_A}
	\arrow[Rightarrow, "{\widehat{T\lambda_f}}", from=1-4, to=2-5, near end]
	\arrow[Rightarrow, "{\widehat{T\phi}}", from=2-5, to=3-6, near end]
	\arrow[Rightarrow, "{\phi}", from=5-6, to=7-6]
	\arrow[Rightarrow, "{\lambda_{e'_A}}", from=7-6, to=8-5]
	\arrow[Rightarrow, "{\gamma'_A}", from=8-5, to=9-4]
	\arrow[Rightarrow, "{\lambda_f}", from=3-6, to=5-6]
	\arrow[Rightarrow, "{m_A}", from=9-4, to=10-3]
	\arrow[Rightarrow, "{\lambda_{e'_B}}"', from=1-4, to=2-3, near end]
	\arrow[Rightarrow, "{\gamma'_B}"', from=2-3, to=3-2, near end]
	\arrow[Rightarrow, "{m_B}"', from=3-2, to=4-1, near end]
	\arrow[Rightarrow, "{\mu_f}"', from=4-1, to=8-1]
	\arrow[Rightarrow, "{\lambda_f}"', from=8-1, to=9-2]
	\arrow[Rightarrow, "{\phi}"', from=9-2, to=10-3]
	\arrow[Rightarrow, "{\widehat{T\lambda_f}}", from=3-2, to=4-3, near end]
	\arrow[Rightarrow, "{\lambda_f}"', from=4-3, to=6-3]
	\arrow[Rightarrow, "{\mu'_f}"', from=6-3, to=8-3]
	\arrow[Rightarrow, "{m_A}"', from=8-3, to=9-2, near end]
	\arrow[Rightarrow, "{\widehat{T\lambda_f}}", from=2-3, to=3-4, near end]
	\arrow[Rightarrow, "{\gamma'_B}"', from=3-4, to=4-3, near end]
	\arrow[Rightarrow, "{\lambda_f}", from=3-4, to=5-4]
	\arrow[Rightarrow, "{\gamma'_B}"', from=5-4, to=6-3, near end]
	\arrow[Rightarrow, "{\phi}", from=8-3, to=9-4, near end]
	\arrow[Rightarrow, "{\widehat{T'\phi}}", from=5-4, to=6-5, near end]
	\arrow[Rightarrow, "{\phi}", from=6-5, to=8-5]
	\arrow[Rightarrow, "{\lambda_{e'_B}}"', from=2-5, to=3-4, near end]
	\arrow[Rightarrow, "{\lambda_{e'_A}}"', from=5-6, to=6-5, near end]
\end{tikzcd}
};
\end{tikzpicture}
$$
Now,
\begin{itemize}
 \item the region on the far left commutes by the modification property for $m$;
 \item the hexagon on the bottom commutes by the multiplication condition for $(f,\phi)$ as in \Cref{defmorphalg};
 \item the hexagon on the top right commutes by pseudonaturality of $\lambda$;
 \item all the remaining parallelograms commute by the interchange law.
\end{itemize}
This makes $(f,\phi\,\lambda)$ a pseudomorphism of $T$-pseudoalgebras.

Finally, to prove 2-functoriality, let $(f,\phi)$ and $(g,\chi)$ be pseudomorphisms of $T'$-pseudoalgebras $A\to B$, and let $\alpha:f\Rightarrow g$ be a 2-cell of $T'$-pseudoalgebras. We have that $\alpha$ is canonically also a 2-cell of $T$-pseudoalgebras, since the relevant diagram can be decomposed as follows,
$$
 \begin{tikzcd}
  e'_B\circ\lambda_B\circ Tf \nat{d}{T\alpha} \nat{r}{\lambda_f} & e'_B\circ T'f\circ\lambda_A \nat{r}{\phi} \nat{d}{T'\alpha} & f\circ e'_A\circ\lambda_A \nat{d}{\alpha} \\
  e'_B\circ\lambda_B\circ Tg \nat{r}[swap]{\lambda_g} & e'_B\circ T'g\circ\lambda_A \nat{r}[swap]{\chi} & g\circ e'_A\circ\lambda_A
 \end{tikzcd}
 $$
 again omitting the whiskerings, and now
 \begin{itemize}
  \item the left square commutes by pseudonaturality of $\lambda$;
  \item the right square commutes since $\alpha$ is a 2-cell of pseudoalgebras, as in \Cref{alg2cell}.
 \end{itemize}

 This action on pseudoalgebras, their morphisms and their 2-cells defines then a 2-functor from the 2-category of $T'$-pseudoalgebras to the 2-category of $T$-pseudoalgebras (notice the direction).
\end{proof}

We encourage the readers more familiar with 2-dimensional diagrams to rewrite the proof using 2-cells and, for clarity, we suggest to dedicate one direction in each diagram to the transformation $\lambda$.

\printbibliography
\addcontentsline{toc}{section}{\bibname}

\end{document}